\documentclass[final,1p,times]{elsarticle}

\usepackage{amssymb}
\usepackage{amsthm}
\usepackage{amsmath,amssymb,amsopn,amsfonts,mathrsfs,amsbsy,amscd}
\usepackage{longtable}
\usepackage{longtable}
\usepackage{caption}
\usepackage{multirow}
\usepackage{tabularx}
\usepackage{rotating}
\usepackage{array}
\usepackage{lineno}

\newcommand{\Id}{\operatorname{Id}}

\newcommand{\AdS}{\mathrm{AdS}}
\newcommand{\dS}{\mathrm{dS}}

\newcommand{\diag}{\mathrm{diag}}

\newcommand{\Mat}{\mathrm{Mat}}

\newcommand{\prs}{\langle\;,\;\rangle}

\newcommand{\too}{\longrightarrow}
\newcommand{\Span}{\mathrm{Span}}

\newcommand{\esp}{\quad\mbox{and}\quad}

\newcommand{\so}{\mathfrak{so}}
\def\br{[\;,\;]}

\newcommand{\G}{\mathfrak{g}}
\newcommand{\g}{\mathfrak{g}}

\newcommand{\af}{\mathfrak{a}}

\newcommand{\Ki}{{\mathrm{K}}}

\newcommand{\h}{{\mathfrak{h}}}
\newcommand{\aff}{\mathfrak{aff}}

\newcommand{\Lu}{{\mathrm{L}}}
\newcommand{\Ru}{{\mathrm{R}}}

\newcommand{\ad}{{\mathrm{ad}}}

\newcommand{\tr}{{\mathrm{tr}}}

\newcommand{\ass}{{\mathrm{ass}}}

\newcommand{\B}{{\cal B}}

\newcommand{\na}{\nabla}

\newcommand{\al}{\alpha}
\newcommand{\End}{\mathrm{End}}
\newcommand{\be}{\beta}
\newcommand{\ga}{\gamma}

\newcommand{\e}{\epsilon}

\newcommand{\la}{\lambda}

\newtheorem{theo}{Theorem}[section]
\newtheorem{pr}{Proposition}[section]
\newtheorem{Le}{Lemma}[section]

\font\bb=msbm10

\def\B{\hbox{\bb B}}
\def\R{\hbox{\bb R}}

\begin{document}


\begin{frontmatter}
	
	
	
	
	\title{  On Lorentzian Lie groups of Nonzero Constant Curvature }
	
	\author[label1]{ Mohamed Boucetta}
	\address[label1]{ Cadi-Ayyad University\\
		BP 549 Marrakech Morocco\\e-mail: m.boucetta@uca.ac.ma  
		
	}
	
	
	
	
	
	\begin{abstract} 
		
		We investigate Lie groups endowed with left-invariant Lorentzian metrics of nonzero constant sectional curvature. This work continues our previous study of flat Lorentzian Lie groups and aims at a systematic description of the Lorentzian constant-curvature case.
		
		As a first step, we revisit Heintze's theory of negatively curved Riemannian Lie groups and show that it yields a complete characterization of the Lie algebras of Lie groups admitting a left-invariant Riemannian metric of negative constant curvature. More generally, we extend classical results of Nomizu and Barnet to the pseudo-Riemannian setting by constructing a large family  of Lie groups carrying incomplete left-invariant pseudo-Riemannian metrics of constant sectional curvature.
		
		We then turn to the Lorentzian case. We obtain a detailed description of the Lie algebras of Lie groups admitting a left-invariant Lorentzian metric of nonzero constant curvature, according to the causal nature of their center and derived ideal. This description is complete except when the derived ideal is Lorentzian; in that case, we obtain a complete classification in dimensions three and four. 
		
		We prove that 
		\({\mathrm{SL}(2,\mathbb R)}\) occupies a distinguished position among Lie groups endowed with a left-invariant Lorentzian metric.
		Indeed, we show that every left-invariant Lorentzian metric of nonzero constant curvature on \(\mathrm{SL}(2,\mathbb{R})\) is bi-invariant and therefore complete. Also if a Lorentzian Lie group of nonzero constant curvature admits a spacelike left-invariant Killing vector field, then it is locally isomorphic to \(\mathrm{SL}(2,\mathbb{R})\). We prove that a semisimple  Lie group carries a complete left-invariant Lorentzian metric of nonzero constant curvature if and only if it is locally isomorphic to \(\mathrm{SL}(2,\mathbb{R})\).
		
		Finally, we derive a complete classification of Lorentzian Lie algebras of nonzero constant curvature in dimensions at most four.
		
	\end{abstract}

	\begin{keyword} Lie algebras \sep Lie groups  \sep  constant curvature \sep left-invariant Lorentzian metrics \sep Lorentzian space forms

		\MSC 53C50 \sep
		\MSC 22E25 \sep \MSC 53C30\sep \MSC 17D25\sep\MSC 53C05\sep \MSC 22E60 
		
		
	\end{keyword}

\end{frontmatter}

\section{Introduction}

A Lie group \( G \) endowed with a left-invariant pseudo-Riemannian metric is called a \emph{pseudo-Riemannian Lie group}. When the metric is positive definite (respectively, of signature \((-,+,\ldots,+)\)), we shall refer to it as \emph{Riemannian} (respectively, \emph{Lorentzian}). The Lie algebra \( \mathfrak{g} \) of such a group, equipped with the induced inner product, is called a \emph{pseudo-Euclidean Lie algebra}; accordingly, we use the terms \emph{Euclidean} and \emph{Lorentzian} in the corresponding cases.

Let \((G,h)\) be a pseudo-Riemannian Lie group, and denote by \(\nabla\) its Levi-Civita connection. Identifying \( \mathfrak{g} \) with the space of left-invariant vector fields on \(G\), the connection induces a bilinear product \( \cdot \) on \( \mathfrak{g} \) defined by
\[
X \cdot Y = \nabla_X Y, \qquad X,Y \in \mathfrak{g}.
\]This is characterized by the two relations:
\[ [X,Y]=X\cdot Y-Y\cdot X\esp \langle X\cdot Y,Z\rangle+\langle X\cdot Z,Y\rangle=0. \]
\((G,h)\) has constant sectional curvature if there exists $k\in\R$ such that, for any $X,Y,Z\in\G$,
\[ \Ki(X,Y)=k X\wedge Y \]where
$\Ki(X,Y)=\na_{[X,Y]}-[\na_X,\na_Y]$ and $(X\wedge Y)(Z)=\langle X,Z\rangle Y-\langle Y,Z\rangle X$.

The study of left-invariant pseudo-Riemannian metrics on Lie groups lies at the crossroads of differential geometry, Lie theory and mathematical physics. Among such metrics, those having constant sectional curvature occupy a distinguished position since they provide homogeneous models for space forms and play a fundamental role in general relativity.

In the Riemannian setting,  a connected and simply connected Lie group admits a left-invariant Riemannian metric of positive constant sectional curvature if and only if it is isomorphic to \(\mathrm{SU}(2)\) endowed with a negative scalar multiple of its Killing metric. This follows immediately from the fact that the metric is complete and the group is diffeomorphic to a sphere. It is also  a particular case of Wallach's theorem \cite[Theorem 2.1]{Wallach}.

The situation is markedly different for negative constant curvature.  In \cite{milnor}, J. Milnor proved that every left-invariant Riemannian metric on a Lie group of special type has constant negative sectional curvature. Recall that a Lie group is said to be of special type if its Lie algebra \(\mathfrak{g}\) satisfies the property that \([x,y]\) belongs to the linear span of \(x\) and \(y\) for all \(x,y\in\mathfrak{g}\).

Lie groups of special type form only a particular subclass of the broader family of negatively curved Lie groups introduced by Heintze in \cite{Heintz}. Indeed, Heintze's results implicitly provide a complete characterization of the Lie algebras of Lie groups carrying a left-invariant Riemannian metric of negative constant curvature. While this characterization is not stated explicitly in \cite{Heintz}, it can be extracted without difficulty from his analysis.

The Lorentzian case is considerably richer and has received much less attention in the literature. To the best of our knowledge, only two works have specifically addressed Lie groups endowed with left-invariant Lorentzian metrics of constant sectional curvature.

In \cite{Nomizu}, K. Nomizu proved that every left-invariant Lorentzian metric on a Lie group of special type has constant sectional curvature. Unlike the Riemannian case, however, the sign of the curvature depends on the chosen metric and may be positive, negative, or zero.

Later, in \cite{barnet}, F. Barnet showed that every Lie group admitting a left-invariant Riemannian metric of negative constant curvature also admits a left-invariant Lorentzian metric of positive constant curvature. Moreover, he proved that such a Lie group admits a left-invariant Lorentzian metric of negative or zero constant curvature if and only if its Lie algebra contains a one-dimensional ideal.

Apart from these papers, only partial results are available, mainly through the classification of left-invariant Lorentzian metrics on three-dimensional Lie groups \cite{Chakkar}, where the metrics of constant sectional curvature can be identified explicitly.

The purpose of this paper is to undertake a systematic study of Lorentzian Lie groups of nonzero constant sectional curvature, continuing our previous investigation of flat Lorentzian Lie groups \cite{Boucetta}.

Our first contribution is twofold. On the one hand, we clarify the Riemannian situation underlying Heintze's work \cite{Heintz}; on the other hand, we extend the results of Nomizu and Barnet \cite{Nomizu,barnet} to the general pseudo-Riemannian setting and get a large class of incomplete pseudo-Riemannian Lie groups with constant curvature. More precisely, we prove the following theorem.

\begin{theo}\label{euclidean}
	Let $(\G,\langle\cdot,\cdot\rangle)$ be a pseudo-Euclidean vector space, let $B:\G\to\G$ be a skew-symmetric endomorphism, let $u\in\G$ with $\langle u,u\rangle\not=0$, and let $\lambda\in\R^*$ satisfy $Bu=0$. Define
	\[
	E=B+\lambda\;\mathrm{id}_{\G}
	\]
	and
	\[
	[x,y]=\langle x,u\rangle E(y)-\langle y,u\rangle E(x), \qquad x,y\in\G.
	\]
	Then $(\G,[\cdot,\cdot],\langle\cdot,\cdot\rangle)$ is a pseudo-Euclidean Lie algebra whose associated simply connected Lie group endowed with the   left-invariant pseudo-Riemannian metric associated to $\prs$  has constant  sectional curvature
	\[
	k=-\langle u,u\rangle\lambda^2
	\]and the metric is complete if and only if $\prs$ is definite.
	
	Moreover, every Lie algebra of a Riemannian Lie group endowed with a left-invariant metric of negative constant curvature is isomorphic to one of the above models.
\end{theo}
Note that the incompleteness of these metrics in the particular case where \(B=0\) and \(\prs\) is Lorentzian was established in \cite{Guediri}.

Our second contribution concerns the Lorentzian setting, where we obtain a substantial extension of the existing theory. More precisely, we provide  a detailed description of the Lie algebras of Lie groups admitting a left-invariant Lorentzian metric of nonzero constant curvature, according to the causal nature of their center and derived ideal. This description is complete except when the derived ideal is Lorentzian; in that case, we obtain a complete classification in dimension four.

A central outcome of our work is that $\mathrm{SL}(2,\mathbb R)$ occupies a distinguished position among Lorentzian Lie groups of nonzero constant curvature. We prove that every left-invariant Lorentzian metric of nonzero constant curvature on ${\mathrm{SL}(2,\mathbb R)}$ is necessarily bi-invariant and hence complete. Moreover, if a Lorentzian Lie group of nonzero constant curvature admits a non-central spacelike left-invariant Killing vector field, then it is locally isomorphic to $\mathrm{SL}(2,\mathbb R)$. Finally, we show that a semisimple Lie group carries a complete left-invariant Lorentzian metric of nonzero constant curvature if and only if it is locally isomorphic to $\mathrm{SL}(2,\mathbb R)$.

These results are summarized in the following theorem.

\begin{theo}\label{main}
	Let $(G,h)$ be a Lorentzian Lie group of nonzero constant curvature, and let
	$(\mathfrak g,[\cdot,\cdot],\langle\cdot,\cdot\rangle)$ denote its associated Lorentzian Lie algebra. Then the following assertions hold:
	
	\begin{enumerate}
		
		\item If the center $Z(\mathfrak g)$ is nontrivial, then $\mathfrak g$ belongs to one of the families described in Table~\ref{1}.
		
		\item If the derived ideal $[\mathfrak g,\mathfrak g]$ is nondegenerate and Euclidean or degenerate, then $\mathfrak g$ is isomorphic to one of the models described in Table \ref{2}.
		
		\item If $\dim\G\leq 4$ then $\G$ is isomorphic to one of the models described in Tables \ref{3} and \ref{4}.
		
	\end{enumerate}
	
	Furthermore, if the metric $h$ is complete and $\mathfrak g$ is semisimple, then
	$\mathfrak g \simeq \mathfrak{sl}(2,\mathbb R).$
	
	Equivalently, the only semisimple Lie groups admitting a complete left-invariant Lorentzian metric of nonzero constant curvature are those locally isomorphic to $\mathrm{SL}(2,\mathbb R)$.
\end{theo}

The preceding theorem provides a substantial step towards the classification of Lorentzian Lie algebras of nonzero constant sectional curvature. In particular, it yields a complete classification in dimensions at most four and identifies several families in arbitrary dimension. Nevertheless, a complete classification in higher dimensions requires the investigation of two remaining cases.

The first concerns perfect Lie algebras with trivial center, that is, Lie algebras satisfying
\[
[\mathfrak g,\mathfrak g]=\mathfrak g,
\qquad Z(\mathfrak g)=\{0\}.
\]
Although the theorem establishes a strong rigidity result in the semisimple case under the additional assumption of geodesic completeness, the general problem remains open without these restrictions. In particular, it would be interesting to determine whether non-semisimple perfect Lie algebras can admit Lorentzian metrics of nonzero constant sectional curvature and to understand the extent to which the completeness assumption influences the possible algebraic structures.

The second case concerns nonperfect Lie algebras whose derived ideal is nondegenerate and Lorentzian. In this situation, the orthogonal complement of the derived ideal is a nontrivial Euclidean subspace. The main challenge is to determine the restrictions imposed by the constant curvature condition on the Lie algebra structure and the corresponding Lorentzian metric. While this case has been completely resolved in dimensions three and four, its treatment in arbitrary dimension remains an open problem.

Consequently, resolving these two cases would complete the algebraic classification of Lorentzian Lie groups with nonzero constant sectional curvature, without imposing any completeness assumption on the metric.

The paper is organized as follows.

Section~\ref{section3} is devoted to the proof of Theorem \ref{euclidean}.  We introduce a family of pseudo-Euclidean Lie algebras associated with a skew-symmetric endomorphism and a distinguished vector, and we prove that the corresponding simply connected Lie groups carry left-invariant metrics of constant sectional curvature which are complete if and only if these metrics are definite. As an application, we show that every Lie algebra of a Riemannian Lie group of negative constant curvature is isomorphic to one of these models, thereby making explicit the description implicit in Heintze's work.

In Section~\ref{section4}, we investigate Lorentzian Lie groups of nonzero constant curvature whose center is nontrivial. We obtain a complete description of the corresponding Lorentzian Lie algebras and derive several structural consequences.

Section~\ref{section5} deals with the case where the derived ideal is nondegenerate and Euclidean. We prove that such Lie algebras are necessarily isomorphic to the models introduced in Section~\ref{section3}.

In Section \ref{section5bis}, we provide a complete classification of Lorentzian Lie algebras of dimensions three and four with nonzero constant sectional curvature and a nondegenerate Lorentzian derived ideal.

In Section~\ref{section6}, we study the case where the derived ideal is degenerate. We determine all possible algebraic structures arising in this situation and obtain the families appearing in Table~\ref{2}.

In Section~\ref{section7} we show that $\mathrm{SL}(2,\mathbb R)$ occupies a distinguished position among Lorentzian Lie groups of nonzero constant curvature and we prove  the last part of Theorem \ref{main}.

The results obtained in Sections \ref{section4}-\ref{section7} give a complete proof of Theorem \ref{main}.

Finally, Section~\ref{section9} is devoted to an appendix in which we collect the lengthy computational details used throughout the paper.

Tables~\ref{3} and~\ref{4} summarize consequences of our classification results. To obtain the normal forms of the Lie brackets and the corresponding metrics displayed in these tables, several changes of basis are required. Since these reductions are straightforward but rather lengthy, we omit their details.

Through this paper, we will use the following terminology. A \emph{pseudo-Euclidean vector space} is a finite-dimensional real vector space \(V\) endowed with a nondegenerate symmetric bilinear form \(\langle \cdot,\cdot \rangle\) of signature \((q,n-q)\). The cases \(q=0\) and \(q=1\) correspond to \emph{Euclidean} and \emph{Lorentzian} structures, respectively.

Let \((V,\langle \cdot,\cdot \rangle)\) be a pseudo-Euclidean vector space. For any endomorphism $F:V\too V$, we denote by $F^*$ its adjoint and by \(\mathfrak{so}(V)\) the Lie algebra of skew-symmetric endomorphisms of \(V\). For any $x,y\in V$, $x\wedge y\in\so(V)$ is given by
\begin{equation}\label{wedge} (x\wedge y)(z)=\langle x,z\rangle y-\langle y,z\rangle x \end{equation}
and for every $F\in\so(V)$, we have
\begin{equation}\label{F}
	[F,x\wedge y]=F(x)\wedge y+x\wedge F(y).
\end{equation}

\section{Proof of Theorem \ref{euclidean}}\label{section3}

In this section, we give a complete proof of Theorem~\ref{euclidean}. Before proceeding, let us briefly discuss the case of positive constant curvature.

Let $(G,h)$ be a connected and simply connected Riemannian Lie group of positive constant sectional curvature. By the classification of simply connected Riemannian space forms, $G$ is diffeomorphic to a sphere $S^n$. Among spheres of dimension at least two, $S^3$ is the only one admitting a Lie group structure. Hence $G\simeq \mathrm{SU}(2)$, and $(G,h)$ is isometric to $\mathrm{SU}(2)$ endowed with a bi-invariant metric of positive constant sectional curvature.

The study of pseudo-Riemannian Lie groups can be approached from both algebraic and geometric viewpoints. From the algebraic point of view, a pseudo-Riemannian Lie group $(G,h)$ gives rise to its Lie algebra $(\G,\br)$ endowed with the nondegenerate symmetric bilinear form
\[
\prs=h(e).
\]
We refer to $(\G,\br,\prs)$ as a pseudo-Euclidean Lie algebra. We use the term Euclidean when $\prs$ is positive definite and Lorentzian when $\prs$ has signature $(-,+,\ldots,+)$.

An important algebraic tool associated with $(\G,\br,\prs)$ is its Levi-Civita product $\cdot$, defined by the Koszul formula
\begin{equation}\label{Levi}
	2\langle u\cdot v,w\rangle
	=
	\langle [u,v],w\rangle
	+\langle [w,u],v\rangle
	+\langle [w,v],u\rangle,
	\qquad u,v,w\in\G.
\end{equation}
We denote by $\Lu_u,\Ru_u:\G\too\G$ the corresponding left and right multiplication operators, respectively:
\[
\Lu_uv=u\cdot v,
\qquad
\Ru_uv=v\cdot u.
\]
For every $u\in\G$, the operator $\Lu_u$ is skew-symmetric with respect to $\prs$, and
\[
\ad_u=\Lu_u-\Ru_u,
\]
where $\ad_uv=[u,v]$.

We say that $(\G,\br,\prs)$ has constant sectional curvature $k$ if
\[
\Ki(u,v)
=
\Lu_{[u,v]}-[\Lu_u,\Lu_v]
=
k\,u\wedge v,
\]
for all $u,v\in\G$.

The following proposition provides a broad family of pseudo-Euclidean Lie algebras of constant sectional curvature. Depending on the causal character of $u$, the curvature may be positive or negative. This construction generalizes the classes introduced by Milnor, Nomizu, and Barnet~\cite{milnor,Nomizu,barnet} and constitutes the first part of Theorem~\ref{euclidean}.

\begin{pr}\label{euclideanbis}
	Let $(\G,\langle\cdot,\cdot\rangle)$ be a pseudo-Euclidean vector space, let
	$B\in\so(\G)$, let $u\in\G$ satisfy
	$\langle u,u\rangle\neq0$ and
	$
	Bu=0,
	$
	and let $\lambda\in\R^*$. Define $E=B+\lambda\,\Id_{\G}$
	and
	\[
	[x,y]
	=
	\langle x,u\rangle E(y)
	-
	\langle y,u\rangle E(x),
	\qquad x,y\in\G.
	\]
	Then $(\G,\br,\prs)$ is a pseudo-Euclidean Lie algebra whose associated simply connected Lie group, endowed with the left-invariant pseudo-Riemannian metric induced by $\prs$, has nonzero constant sectional curvature
	\[
	k=-\langle u,u\rangle\lambda^2.
	\]
	Moreover, the metric is complete if and only if $\prs$ is definite.
\end{pr}

\begin{proof}
	It is easy to check that $\br$ satisfies the Jacobi identity and $(\G,\br)$ is a Lie algebra. To show that the curvature is constant,
	we shall also use the  identity \eqref{F}.

	Using the expression of the Lie bracket together with the Koszul formula~\eqref{Levi}, one readily checks that the Levi-Civita product is given by
	\begin{equation}\label{B}
		\Lu_x
		=
		\langle x,u\rangle B+\lambda\,x\wedge u,
		\qquad x\in\G.
	\end{equation}
	We now compute the curvature operator $
	\Ki(x,y)=\Lu_{[x,y]}-[\Lu_x,\Lu_y].$
	First,
	\begin{align*}
		\Lu_{[x,y]}&=\langle [x,y],u\rangle B+\la [x,y]\wedge u\\
		&=\langle x,u\rangle \langle E(y),u\rangle B-\langle y,u\rangle \langle E(x),u\rangle B+\la \langle x,u\rangle E(y)\wedge u-\la \langle y,u\rangle E(x)\wedge u\\
		&=\langle x,u\rangle \langle E^*(u),y\rangle B-\langle y,u\rangle \langle E^*(u),x\rangle B+\la \langle x,u\rangle E(y)\wedge u-\la \langle y,u\rangle E(x)\wedge u\\
		&=\la \langle x,u\rangle B(y)\wedge u-\la \langle y,u\rangle B(x)\wedge u
		+\la^2 \langle x,u\rangle y\wedge u-\la^2 \langle y,u\rangle x\wedge u.
	\end{align*}	
	\begin{align*} 
		-[\Lu_x,\Lu_y]&=-[\langle x,u\rangle B+\la x\wedge u,\langle y,u\rangle B+\la y\wedge u]\\
		&= -\la \langle x,u\rangle[B,y\wedge u]+\la \langle y,u\rangle[B,x\wedge u]-\la^2[x\wedge u,y\wedge u]\\
		&=-\la \langle x,u\rangle B(y)\wedge u+\la \langle y,u\rangle B(x)\wedge u
		+\la ^2\langle y,u\rangle x\wedge u -\la ^2\langle x,u\rangle y\wedge u
		-\la^2\langle u,u\rangle x\wedge y.
	\end{align*}
	Adding these two identities yields
	\[
	\Ki(x,y)
	=
	-\lambda^2\langle u,u\rangle\,x\wedge y.
	\]
	Hence the associated left-invariant pseudo-Riemannian metric has constant sectional curvature
	\[
	k=-\lambda^2\langle u,u\rangle\neq0.
	\]
	
	It remains to study completeness. Let $(G,h)$ denote the simply connected pseudo-Riemannian Lie group associated with $(\G,\br,\prs)$. If $\prs$ is definite, then the metric is complete.
	
	Suppose now that $\prs$ is indefinite. A curve $c:I\too G$ is a geodesic if and only if its left logarithmic derivative
	\[
	x(t)
	=
	T_{c(t)}\ell_{c(t)^{-1}}\bigl(c'(t)\bigr)
	\]
	satisfies
	\begin{equation}\label{geo}
		x'(t)=-x(t)\cdot x(t).
	\end{equation}
	Consequently, the metric is complete if and only if every solution of \eqref{geo} is defined on the whole real line.
	
	Let $x:I\too\G$ be a solution of \eqref{geo} with $x(0)=x_0$. Since $\Lu_x$ is skew-symmetric,
	\[
	\frac{d}{dt}\langle x(t),x(t)\rangle
	=
	-2\langle x(t)\cdot x(t),x(t)\rangle
	=
	0.
	\]
	Thus
	$
	r=\langle x(t),x(t)\rangle
	$
	is constant. Set
	\[
	c=\langle u,u\rangle,
	\qquad
	a(t)=\langle x(t),u\rangle.
	\]
	By \eqref{B}, equation \eqref{geo} becomes
	\[
	x'(t)
	=
	-a(t)B(x(t))
	+\lambda a(t)x(t)
	-\lambda r u.
	\]
	Taking the scalar product with $u$ and using $Bu=0$, we obtain
	\[
	a'(t)
	=
	\lambda\bigl(a(t)^2-cr\bigr).
	\]
	Since $\prs$ is indefinite and $c\neq0$, we may choose $x_0\in\G$ such that
	\[
	cr
	=
	\langle u,u\rangle\langle x_0,x_0\rangle
	=
	-1.
	\]
	Therefore
	\[
	a'(t)=\lambda\bigl(a(t)^2+1\bigr),
	\]
	and hence
	\[
	a(t)=\tan\bigl(\lambda t+\mu\bigr),
	\qquad
	\mu=\arctan(a(0)).
	\]
	This solution blows up in finite time. Consequently, the corresponding solution of \eqref{geo} cannot be defined on all of $\R$, and the metric is not complete.
\end{proof}

The following theorem extracts from Heintze's work~\cite{Heintz} an explicit algebraic characterization of Euclidean Lie algebras of negative constant sectional curvature.

\begin{theo}\label{constant}
	Let $(\G,\br,\prs)$ be a Euclidean Lie algebra. Then $\G$ has constant negative sectional curvature $k$ if and only if there exist a skew-symmetric endomorphism $B\in\so(\G)$, a scalar $\la\in\R^*$, and a unit vector $u\in\G$ such that $Bu=0$ and
	\[
	[x,y]
	=
	\langle x,u\rangle E(y)
	-
	\langle y,u\rangle E(x),
	\qquad x,y\in\G,
	\]
	where
	$
	E=B+\la\Id_\G.
	$
	Moreover,
	$
	k=-\la^2.
	$
\end{theo}

\begin{proof}
	Let $(\G,\br,\prs)$ be a Euclidean Lie algebra of constant negative sectional curvature $k$, and let $\cdot$ denote its Levi-Civita product.
	
	We first note that
	\[
	[\G,\G]^\perp
	=
	\{x\in\G:\Ru_x^*=\Ru_x\}.
	\]
	Moreover, if $x\in[\G,\G]^\perp$, then $x\cdot x=0.$
	
	According to~\cite[Propositions 1 and 2]{Heintz}, $\G$ is solvable,
	$\dim[\G,\G]=\dim\G-1,$
	and there exists a unit vector $
	u\in[\G,\G]^\perp$
	such that the symmetric part of
	$
	(\ad_u)_{\mid[\G,\G]}
	$
	is positive definite. Since
	\[
	\ad_u=\Lu_u-\Ru_u,
	\]
	$\Lu_u$ is skew-symmetric, and $\Ru_u$ is symmetric because
	$u\in[\G,\G]^\perp$, it follows that
	$
	-(\Ru_u)_{\mid[\G,\G]}
	$
	is positive definite. Replacing $u$ by $-u$ if necessary, we may therefore assume that
	\[
	D:=(\Ru_u)_{\mid[\G,\G]}
	\]
	is symmetric and positive definite.
	
	Since $[\G,\G]$ has codimension one and $u\in[\G,\G]^\perp$ is a unit vector,
	\[
	\G=[\G,\G]\stackrel{\perp}{\oplus}\R u.
	\]
	
	We now determine $D$. For any $x,y,z\in\G$, the constant-curvature identity can be written as
	\[
	(x\cdot y)\cdot z-x\cdot(y\cdot z)
	-(y\cdot x)\cdot z+y\cdot(x\cdot z)
	=
	k\bigl(\langle x,z\rangle y-\langle y,z\rangle x\bigr).
	\]
	Let $x\in[\G,\G]$ be an eigenvector of $D$ with eigenvalue $\la>0$. Thus
	\[
	x\cdot u=\la x.
	\]
	Taking $y=z=u$ in the preceding curvature identity and using $u\cdot u=0$, we obtain
	\[
	\la^2x-(u\cdot x)\cdot u+\la\,u\cdot x=-kx.
	\]
	Since $\Ru_u$ is symmetric,
	\[
	\begin{aligned}
		\langle (u\cdot x)\cdot u,x\rangle
		=
		\langle \Ru_u(u\cdot x),x\rangle
		=
		\langle u\cdot x,\Ru_u x\rangle
		=
		\la\langle u\cdot x,x\rangle
		=0,
	\end{aligned}
	\]
	where the last equality follows from the skew-symmetry of $\Lu_u$. Taking the scalar product of the curvature identity with $x$ therefore gives
	\[
	\la^2|x|^2=-k|x|^2.
	\]
	Hence
	$
	\la^2=-k.
	$
	Since every eigenvalue of the positive definite symmetric operator $D$ is positive, all its eigenvalues are equal to
	$
	\la=\sqrt{-k}.
	$
	Consequently,
	\[
	D=\la\Id_{[\G,\G]},
	\qquad
	\la=\sqrt{-k}.
	\]
	Using the orthogonal decomposition
	$
	\G=[\G,\G]\stackrel{\perp}{\oplus}\R u,
	$
	the Levi-Civita product can now be written, for $a,b\in[\G,\G]$, as
	\[
	a\cdot b
	=
	a\star b-\la\langle a,b\rangle u,
	\qquad
	a\cdot u=\la a,
	\qquad
	u\cdot a=Ba,
	\]
	where $B\in\so([\G,\G])$
	 and $\star$ is the Levi-Civita product associated with the restriction of $\prs$ to $[\G,\G]$. In particular,
	\[
	\langle a\star b,c\rangle
	+
	\langle b,a\star c\rangle
	=0,
	\qquad a,b,c\in[\G,\G].
	\]
	
	Let $\Ki^*$ denote the curvature associated with $\star$. We have,
	 for any $a,b,c\in[\G,\G]$:
	\begin{align*}
		\Ki(a,b)c&=	(a.b).c-a.(b.c)-(b.a).c+b.(a.c)\\
		&=(a\star b).c-\la\langle a,b\rangle Bc-a.(b\star c)+\la^2\langle b,c\rangle a
		-(b\star a).c+\la\langle b,a\rangle Bc+b.(a\star c)-\la^2\langle a,c\rangle b\\
		&=(a\star b)\star c-\la\langle a\star b,c\rangle u-\la\langle a,b\rangle Bc
		-a\star(b\star c)+\la\langle a,b\star c\rangle u+\la^2\langle b,c\rangle a\\
		&-(b\star a)\star c+\la\langle b\star a,c\rangle u+\la\langle b,a\rangle Bc
		+b\star(a\star c)-\la\langle b,a\star c\rangle u-\la^2\langle a,c\rangle b\\
		&=\Ki^*(a,b)c-\la^2 (a\wedge b)(c),
	\end{align*}
	Since
	$
	\Ki(a,b)=k\,a\wedge b
	\;\text{and}\;
	k=-\la^2,
	$
	we obtain
	$
	\Ki^*(a,b)=0.
	$
	Thus the Euclidean Lie algebra $[\G,\G]$, endowed with the induced metric, is flat.
	
	Since $\G$ is solvable, its derived ideal $[\G,\G]$ is nilpotent. A nilpotent Lie group carrying a flat left-invariant Riemannian metric is abelian~\cite{milnor}. Hence, for any $a,b\in[\G,\G]$,
	\[
	[a,b]=0\esp [u,a]
	=
	u\cdot a-a\cdot u
	=
	B(a)-\la a.
	\]
	Replacing  $\la$ by $-\la$ , we may write
	\[
	[u,a]=B(a)+\la a,
	\qquad a\in[\G,\G],
	\]
	with $B$ skew-symmetric and $\la^2=-k$.
	
	Extending $B$ to $\G$ by setting $B(u)=0$ and putting
	\[
	E=B+\la\Id_\G,
	\]
	we obtain, for all $x,y\in\G$,
	\[
	[x,y]
	=
	\langle x,u\rangle E(y)
	-
	\langle y,u\rangle E(x).
	\]
	
	Conversely, any bracket of this form defines, by Proposition~\ref{euclideanbis}, a Euclidean Lie algebra of constant sectional curvature
	$
	k=-\la^2<0.
	$
	This completes the proof.
\end{proof}

We shall repeatedly use the following standard linear-algebraic fact.

\begin{Le}\label{adB}
	Let $V$ be a finite-dimensional real vector space and let $B\in\End(V)$. If
	\[
	\lambda_1,\ldots,\lambda_n
	\]
	are the complex eigenvalues of $B$, counted with algebraic multiplicities, then the complex eigenvalues of
	\[
	\ad_B:\End(V)\longrightarrow\End(V),
	\qquad
	A\longmapsto[B,A]=BA-AB,
	\]
	are
	\[
	\lambda_i-\lambda_j,
	\qquad 1\leq i,j\leq n.
	\]
	In particular, if $V$ is Euclidean and $B$ is skew-symmetric, then $0$ is the only possible real eigenvalue of $\ad_B$.
\end{Le}

\section{Lorentzian Lie algebras of nonzero constant curvature with nontrivial center}\label{section4}

In this section, we show that the center of a Lorentzian Lie algebra of nonzero constant curvature, if nontrivial, is necessarily one-dimensional. We then determine all such Lie algebras by distinguishing three cases according to the causal character of a generator of the center: spacelike, timelike, and lightlike.

\begin{pr}\label{center}
	Let $(\G,\br,\prs)$ be a Lorentzian Lie algebra of nonzero constant curvature, and let $Z(\G)$ denote its center. Then $\dim Z(\G)\leq 1$.
\end{pr}

\begin{proof}
	Let $\cdot$ denote the Levi-Civita product of $(\G,\br,\prs)$. By the Koszul formula~\eqref{Levi}, if $a,b\in Z(\G)$, then $a\cdot b=0$. Consequently, $\Ki(a,b)c=0$ for any $a,b,c\in Z(\G)$.
	
	Since $(\G,\br,\prs)$ has constant curvature $k\neq0$, we also have
	\[
	\Ki(a,b)c
	=
	k\bigl(\langle a,c\rangle b-\langle b,c\rangle a\bigr).
	\]
	Hence $\langle a,c\rangle b-\langle b,c\rangle a=0$ for any $a,b,c\in Z(\G)$.
	
	Suppose, by contradiction, that $\dim Z(\G)\geq2$, and let $a,b\in Z(\G)$ be linearly independent. Taking $c=a$ in the preceding identity gives
	$\langle a,a\rangle b-\langle a,b\rangle a=0$, and hence
	$\langle a,a\rangle=\langle a,b\rangle=0$. Similarly, taking $c=b$ gives $\langle b,b\rangle=\langle a,b\rangle=0$.
	
	Thus $\Span\{a,b\}$ is a two-dimensional totally isotropic subspace. This is impossible in a Lorentzian vector space, since its maximal totally isotropic subspaces are one-dimensional. Therefore $\dim Z(\G)\leq1$.
\end{proof}

\subsection{The case of a nondegenerate center}

We begin the study of Lorentzian Lie algebras of nonzero constant curvature with nondegenerate center by establishing a convenient decomposition of their Levi-Civita product.

\begin{pr}\label{Z}
	Let $(\G,\br,\prs)$ be a Lorentzian Lie algebra of nonzero constant curvature, and assume that its center $Z(\G)$ is nontrivial and nondegenerate. Then:
	\begin{enumerate}
		\item $Z(\G)=\R e$, where $\langle e,e\rangle=\al\in\{-1,1\}$. Moreover, $\h=e^\perp$ is nondegenerate and
		$\G=\R e\stackrel{\perp}{\oplus}\h$.
		
		\item There exist a skew-symmetric endomorphism $B\in\so(\h)$ and a  product $\star$ on $\h$ such that, for any $a,b\in\h$,
		\begin{equation}\label{LZ}
			e\cdot a=a\cdot e=Ba,
			\qquad
			a\cdot b=a\star b-\al\langle Ba,b\rangle e.
		\end{equation}
		In particular, the bracket $[a,b]_\star=a\star b-b\star a$ is a Lie bracket and $\star$ is the Levi-Civita product of $(\h,\br_\star,\prs_{|h})$.
		
		\item $(\G,\br,\prs)$ has constant curvature $k$ if and only if, for any $a,b\in\h$,
		\begin{equation}\label{eqZ}
			\begin{cases}
				\Ki^*(a,b)=2\al\langle Ba,b\rangle B
				+k\,a\wedge b+\al\,Ba\wedge Bb,\\
				[B,\Lu_a]=0,\\
				B^2=-k\al\,\Id_{\h},
			\end{cases}
		\end{equation}
		where $\Lu_a$ denotes the left multiplication by $a$ associated with $\star$, and
		$\Ki^*(a,b)=\Lu_{[a,b]_\star}-[\Lu_a,\Lu_b]$ is the curvature operator of $(\h,[\, ,\,]_\star,\prs)$.
	\end{enumerate}
\end{pr}

\begin{proof}
	By Proposition~\ref{center}, we have $Z(\G)=\R e$, where
	$\langle e,e\rangle=\al\in\{-1,1\}$. Since the center is nondegenerate, $\h=e^\perp$ is nondegenerate and
	$\G=\R e\stackrel{\perp}{\oplus}\h$.
	
	Since $e$ is central, the Koszul formula~\eqref{Levi} gives $e\cdot e=0$. Moreover, $\ad_e=0$ implies $\Lu_e=\Ru_e$. Since $\Lu_e$ is skew-symmetric and $\Lu_e e=0$, it leaves $\h$ invariant. Put $B=(\Lu_e)_{|\h}$. Then $B\in\so(\h)$ and, since $\Lu_e=\Ru_e$, we have
	$e\cdot a=a\cdot e=Ba$ for every $a\in\h$.
	
	For $a,b\in\h$, decompose $a\cdot b$ according to
	$\G=\R e\stackrel{\perp}{\oplus}\h$. Since
	\[
	\langle a\cdot b,e\rangle
	=
	-\langle b,a\cdot e\rangle
	=
	-\langle b,Ba\rangle
	=
	-\langle Ba,b\rangle,
	\]
	there exists a bilinear product $\star$ on $\h$ such that
	$a\cdot b=a\star b-\al\langle Ba,b\rangle e$. This proves~\eqref{LZ}.
	
	The corresponding Lie bracket is therefore
	\[
	[a,b]
	=
	[a,b]_\star-2\al\langle Ba,b\rangle e,
	\qquad
	[a,b]_\star=a\star b-b\star a.
	\]
	Since the bracket of $\G$ satisfies the Jacobi identity, its $\h$-component gives
	\[
	[[a,b]_\star,c]_\star
	+[[b,c]_\star,a]_\star
	+[[c,a]_\star,b]_\star=0.
	\]
	Thus $[\, ,\,]_\star$ is a Lie bracket on $\h$. Moreover, the skew-symmetry of the left multiplication operators associated with $\star$ and the identity
	$a\star b-b\star a=[a,b]_\star$ show that $\star$ is precisely the Levi-Civita product of the pseudo-Euclidean Lie algebra
	$(\h,[\, ,\,]_\star,\prs)$.
	
	We now characterize the constant-curvature condition. Having in mind \eqref{LZ}, let us compute the curvature of $(\G,\br,\prs)$. Put $A=-\al B$.
	For $a,b,c\in \h$:
	\begin{align*}
		\Ki(a,b)c&=	(a.b).c-a.(b.c)-(b.a).c+b.(a.c)\\
		&=(a\star b).c+\langle Aa,b\rangle Bc-a.(b\star c)-\langle Ab,c\rangle B a
		-(b\star a).c-\langle Ab,a\rangle Bc+b.(a\star c)+\langle Aa,c\rangle B b\\
		&=(a\star b)\star c+\langle A(a\star b),c\rangle e+\langle Aa,b\rangle Bc
		-a\star(b\star c)-\langle Aa,b\star c\rangle e-\langle Ab,c\rangle B a\\
		&-(b\star a)\star c-\langle A(b\star a),c\rangle e-\langle Ab,a\rangle Bc
		+b\star(a\star c)+\langle Ab,a\star c\rangle e+\langle Aa,c\rangle Bb\\
		&=\Ki^*(a,b)c+\left(\langle Aa,b\rangle-\langle Ab,a\rangle\right)Bc
		+\langle Aa,c\rangle Bb-\langle Ab,c\rangle B a\\
		&+\langle A([a,b]_\star)-a\star Ab+b\star Aa,c\rangle e.
	\end{align*}So the relation $\Ki(a,b)=ka\wedge b$ is equivalent to the first relation in  \eqref{eqZ} and the relation
	\begin{equation}\label{plus}
		B([a, b]_\star)-a\star B(b)+b\star B(a)=0.
	\end{equation}
	
	For any $a,b\in\h$,
	\begin{align*}
		\Ki(a,b)e-k(a\wedge b)(e)&=	(a.b).e-a.(b.e)-(b.a).e+b.(a.e)-k(\langle a,e\rangle b-\langle b,e\rangle a)\\
		&=B(a\star b)- a.B(b)-B(b\star a)+ b.B(a)\\
		&=B([a, b]_\star)-a\star B(b)-\langle Aa,Bb\rangle e+b\star B(a)+\langle Ab,Ba\rangle e\\
		&=B([a, b]_\star)-a\star B(b)+b\star B(a).
	\end{align*}So we get \eqref{plus}.
	On the other hand,
	\begin{align*}
		\Ki(a,e)b-k(a\wedge e)(b)&=(a.e).b-a.(e.b)-(e.a).b+e.(a.b)-k\langle a,b\rangle e\\
		&=B(a).b-a.Bb-Ba.b+B(a\star b)-k\langle a,b\rangle e\\
		&=-a\star B(b)-\al\langle B^2a,b\rangle e+B(a\star b)-k\langle a,b\rangle e
	\end{align*} and we get the second and third identities in \eqref{eqZ}.
	\begin{align*}
		\Ki(a,e)e-k(a\wedge e)(e)&=(a.e).e-(e.a).e+e.(a.e)+k\al a\\
		&=B^2a-B^2a+B^2a+k\al a\\
		&=B^2a+k\al a.
	\end{align*}So $B^2=-k\al \mathrm{Id}_{\h}$. 
	Note that $[B,\Lu_a]=0$ implies  \eqref{plus}.
\end{proof}

We first consider the case where the center is spacelike.

\begin{pr}\label{Zspacelike}
	Let $(\G,\br,\prs)$ be a Lorentzian Lie algebra of nonzero constant curvature $k$ and assume that its center is nontrivial, nondegenerate, and spacelike. Then $\G$ is isomorphic to $\aff(\R)\oplus\R$. More precisely, there exists a basis $\B=(E_1,E_2,E_3)$ such that
	\[
	[E_1,E_2]=E_2
	\qquad\text{and}\qquad
	\Mat(\prs,\B)=
	\begin{pmatrix}
		\mu&0&0\\
		0&0&1\\
		0&1&1
	\end{pmatrix},
	\quad \mu>0.
	\]
	Moreover, the constant curvature is $k=-\frac{1}{4\mu}$.
\end{pr}

\begin{proof}
	By Proposition~\ref{Z}, we have $\G=\R e\stackrel{\perp}{\oplus}\h$, where $\langle e,e\rangle=1$, $\h$ is Lorentzian, and
	\[
	e\cdot a=a\cdot e=Ba,\qquad
	a\cdot b=a\star b-\langle Ba,b\rangle e,
	\]
	with $B\in\so(\h)$. Moreover,
	\begin{equation}\label{spacelikeZ}
		B^2=-k\Id_{\h},\qquad [B,\Lu_a]=0,
	\end{equation}
	and
	\begin{equation}\label{curvaturestar}
	\Ki^*(a,b)
	=
	2\langle Ba,b\rangle B
	+k\,a\wedge b+Ba\wedge Bb.
	\end{equation}
	
	We first show that $\dim\h=2$. Since $B^2=-k\Id_{\h}$ with $k\neq0$, the Lorentzian normal form of a skew-symmetric endomorphism shows that $B$ is of hyperbolic type. Hence there exist isotropic vectors $z,\bar z\in\h$, with $\langle z,\bar z\rangle=1$, a Euclidean subspace $\h_0=\Span\{z,\bar z\}^{\perp}$, and $\rho\neq0$ such that
	\[
	\h=\R z\oplus\h_0\oplus\R\bar z,
	\qquad
	Bz=\rho z,\qquad B\bar z=-\rho\bar z,
	\]
	and $\h$ is $B$-invariant \cite{MP}. 
	
	Denote by $\overline B$ the restriction of $B$ to $\h_0$. From $B^2=-k\Id_{\h}$ we obtain $k=-\rho^2<0$ and
	$\overline B^2=-k\Id_{\h_0}$. If $\h_0\neq\{0\}$, then, since $\overline B$ is skew-symmetric on the Euclidean space $\h_0$,
	\[
	0<-\tr(\overline B^2)=k\dim\h_0,
	\]
	which contradicts $k<0$. Hence $\h_0=\{0\}$ and therefore $\dim\h=2$.

	Choose an orthogonal basis $(x,y)$ of $\h$ such that
	$\langle x,x\rangle=1$ and $\langle y,y\rangle=-1$. Then
	\[
	\Mat(B,(x,y))
	=
	\begin{pmatrix}
		0&\rho\\
		\rho&0
	\end{pmatrix},
	\qquad \rho\neq0,
	\]
	and $B^2=-k\Id_{\h}$ gives $k=-\rho^2$. Since $\so(\h)$ is one-dimensional, the first relation in~\eqref{curvaturestar} reduces to
	\begin{equation}\label{eq1}
		\Lu_{[x,y]_\star}
		=
		-4\rho^2
		\begin{pmatrix}
			0&1\\
			1&0
		\end{pmatrix}.
	\end{equation}
	In particular, $\h$ is nonabelian.
	
	Since every nonabelian two-dimensional Lie algebra has one-dimensional derived ideal, we distinguish its causal character. The derived ideal cannot be degenerate: otherwise a generator $v$ of $[\h,\h]$ belongs to $[\h,\h]^\perp$, hence $v\star v=0$ and $\Lu_v=0$, contradicting~\eqref{eq1}.
	
	If $[\h,\h]$ is spacelike, we may take $[\h,\h]=\R x$. Then $\Lu_y=0$, and writing
	$\Lu_x=\mu_0 \begin{pmatrix}
		0&1\\
		1&0
	\end{pmatrix}$, we have
	\[ \Lu_{[x,y]_\star}=\Lu_{x\star y}=\mu_0\Lu_x=\mu_0^2\begin{pmatrix}
		0&1\\
		1&0
	\end{pmatrix} \]
	and the relation~\eqref{eq1} yields $\mu_0^2=-4\rho^2$, which is impossible.
	
	Thus $[\h,\h]$ is timelike. We may therefore assume that
	$[\h,\h]=\R y$. Then $\Lu_x=0$, and
	\[
	\Mat(\Lu_y,(x,y))
	=
	\begin{pmatrix}
		0&\mu_0\\
		\mu_0&0
	\end{pmatrix}.
	\]
	Equation~\eqref{eq1} gives $\mu_0^2=4\rho^2$. After replacing $x$ by $-x$ if necessary, we may assume $\mu_0=-2\rho$. The nonzero bracket of $\G$ is then
	\[
	[x,y]=2\rho(y+e).
	\]
	
	Set $E_1=x/(2\rho)$, $E_2=y+e$, and $E_3=e$. Then
	$[E_1,E_2]=E_2$, so $\G\simeq\aff(\R)\oplus\R$, and
	\[
	\Mat(\prs,(E_1,E_2,E_3))
	=
	\begin{pmatrix}
		\mu&0&0\\
		0&0&1\\
		0&1&1
	\end{pmatrix},
	\qquad
	\mu=\frac{1}{4\rho^2}>0.
	\]
	Since $k=-\rho^2$, we finally obtain $k=-\frac{1}{4\mu}$.
\end{proof}

	We now consider the case where the center is timelike.
	
	\begin{pr}\label{Ztimelike}
		Let $(\G,\br,\prs)$ be a Lorentzian Lie algebra of nonzero constant curvature $k$ and assume that its center is nontrivial, nondegenerate, and timelike. Then:
		\begin{enumerate}
			\item $\G$ admits an orthogonal decomposition
			$\G=\R e\oplus\R h_0\oplus\R h_1\oplus\af$, where
			$\langle e,e\rangle=-1$, $\langle h_0,h_0\rangle=\langle h_1,h_1\rangle=1$, $\langle h_0,h_1\rangle=0$, and $\af$ is Euclidean;
			
			\item there exist $\la>0$ and two skew-symmetric endomorphisms $J,S:\af\too\af$ such that $J^2=-\Id_\af$ and $[J,S]=0$, and the nonzero Lie brackets are
			\[
			[h_0,a]=\la a+S(a),\qquad
			[h_0,h_1]=2\la(h_1+e),\qquad
			[a,b]=2\la\langle Ja,b\rangle(h_1+e),
			\]
			for any $a,b\in\af$. Moreover, $k=-\la^2$.
		\end{enumerate}
	\end{pr}
	
	\begin{proof}
		By Proposition~\ref{Z}, we have the orthogonal decomposition
		$\G=\R e\oplus\h$, where $\langle e,e\rangle=-1$ and $\h$ is Euclidean. Moreover, there exists $B\in\so(\h)$ such that
		\[
		e\cdot a=a\cdot e=Ba,\qquad
		a\cdot b=a\star b+\langle Ba,b\rangle e,
		\]
		and
		\begin{equation}\label{timelikeeq}
			\Ki^*(a,b)
			=
			-2\langle Ba,b\rangle B
			+k\,a\wedge b-Ba\wedge Bb,
			\qquad
			[B,\Lu_a]=0,\qquad
			B^2=k\Id_{\h}.
		\end{equation}
		The key point is that $(\h,\br_\star,\prs)$ carries a structure of a locally symmetric K\"ahler Lie algebra with negative sectional curvature, and this class of Lie algebras was characterized in \cite[Section~6]{Heintz}.
		
		\subsection*{The structure of locally symmetric K\"ahler Lie algebra on $(\h,\br_\star,\prs)$}

		Since $\h$ is Euclidean and $B$ is skew-symmetric, the last identity implies $k<0$. Put
		$J=B/\sqrt{-k}$. Then $J^2=-\Id_{\h}$ and $[J,\Lu_a]=0$ for every $a\in\h$. Thus $J$ is parallel, and $(\h,\br_\star,\prs,J)$ is a K\"ahler Lie algebra.
		
		Moreover, for any linearly independent $a,b\in\h$,
		\[
		\langle\Ki^*(a,b)a,b\rangle
		=
		-3\langle Ba,b\rangle^2
		+k\bigl(|a|^2|b|^2-\langle a,b\rangle^2\bigr)<0.
		\]
		The curvature $\Ki^*$ is also parallel. Indeed, using $[B,\Lu_a]=0$,  \eqref{timelikeeq} and the relation
		\[ [\Lu_a,b\wedge c]=(a\star b)\wedge c+b\wedge(a\star c) \]
		 we obtain
		\begin{align*}
			(\nabla_a\Ki^*)(b,c)
			&=[\Lu_a,\Ki^*(b,c)]
			-\Ki^*(a\star b,c)-\Ki^*(b,a\star c)\\
			&=k(a\star b)\wedge c
			+k\,b\wedge(a\star c)
			-(a\star Bb)\wedge Bc
			-Bb\wedge(a\star Bc)\\
			&\quad
			-2\langle Bc,a\star b\rangle B
			-2\langle B(a\star c),b\rangle B
			-k(a\star b)\wedge c\\
			&+B(a\star b)\wedge Bc
			-k\,b\wedge(a\star c)
			+Bb\wedge B(a\star c)
			=0.
		\end{align*}
		
		Locally symmetric K\"ahler Lie algebras with negative parallel curvature were characterized in~\cite[Section~6]{Heintz}. Consequently, $\h$ admits an orthogonal decomposition
		$\h=\R h\oplus\R Jh\oplus\af$, where $h$ is a unit vector, $\af$ is $J$-invariant, and
		\[
		[h,a]_\star=\la a+S(a),\qquad
		[h,Jh]_\star=2\la Jh,\qquad
		[a,b]_\star=2\la\langle Ja,b\rangle Jh,\qquad
		[a,Jh]_\star=0,
		\]
		for $a,b\in\af$, where $\la>0$, $S\in\so(\af)$, and $[J,S]=0$.
		
		A straightforward computation gives that the corresponding Levi-Civita product is
		\[
		\begin{cases}
			\Lu_hh=\Lu_hJh=0,\qquad \Lu_ha=S(a),\\
			\Lu_{Jh}h=-2\la Jh,\qquad
			\Lu_{Jh}Jh=2\la h,\qquad
			\Lu_{Jh}a=-\la Ja,\\
			\Lu_ah=-\la a,\qquad
			\Lu_aJh=-\la Ja,\qquad
			\Lu_ab=\la\langle a,b\rangle h-\la\langle a,Jb\rangle Jh,
		\end{cases}
		\]for any $a,b\in\af$.
		
		We now check the curvature relation in~\eqref{timelikeeq}. Since $B=\sqrt{-k}\,J$, its right-hand side is
		\[
		Q(a,b)
		=
		|k|\bigl(
		2\langle Jb,a\rangle J-a\wedge b-Ja\wedge Jb
		\bigr).
		\]
		The following computations show that $\Ki^*=Q$ and, at the same time, yield $\la^2=|k|$.
		
		First,
		\begin{align*}
			\Ki^*(h,Jh)h
			&=2\la\Lu_{Jh}h-\Lu_h\Lu_{Jh}h+\Lu_{Jh}\Lu_hh
			=-4\la^2Jh,\\
			Q(h,Jh)h&=-4|k|Jh.
		\end{align*}
		Hence $\la^2=|k|=-k$. From now on,
		\[
		Q(a,b)
		=
		\la^2\bigl(
		2\langle Jb,a\rangle J-a\wedge b-Ja\wedge Jb
		\bigr).
		\]
		For $a\in\af$,
		\begin{align*}
			\Ki^*(h,Jh)a
			&=2\la\Lu_{Jh}a-\Lu_h\Lu_{Jh}a+\Lu_{Jh}\Lu_ha\\
			&=-2\la^2Ja+\la\Lu_hJa+\Lu_{Jh}Sa
			=-2\la^2Ja,\\
			Q(h,Jh)a&=-2\la^2Ja.
		\end{align*}
		
		For $a,b\in\af$,
		\begin{align*}
			\Ki^*(a,h)h
			&=(-\la a-Sa)\star h+h\star(a\star h)\\
			&=\la^2a+\la Sa-\la Sa
			=\la^2a,\\
			Q(a,h)h&=\la^2a,
		\end{align*}
		and
		\begin{align*}
			\Ki^*(a,h)b
			&=-\la\,a\star b-(Sa)\star b-a\star S(b)\\
			&=-\la^2\langle a,b\rangle h
			+\la^2\langle a,Jb\rangle Jh,\\
			Q(a,h)b
			&=-\la^2\langle a,b\rangle h
			-\la^2\langle Ja,b\rangle Jh\\
			&=-\la^2\langle a,b\rangle h
			+\la^2\langle a,Jb\rangle Jh.
		\end{align*}
		
		Finally,
		\begin{align*}
			\Ki^*(a,b)h
			&=-2\la^2\langle Ja,b\rangle Jh
			=Q(a,b)h,
		\end{align*}
		while, for $c\in\af$,
		\begin{align*}
			\Ki^*(a,b)c
			&=
			2\la\langle Ja,b\rangle\,Jh\star c
			-\la\langle b,c\rangle\,a\star h
			+\la\langle b,Jc\rangle\,a\star Jh\\
			&\quad
			+\la\langle a,c\rangle\,b\star h
			-\la\langle a,Jc\rangle\,b\star Jh\\
			&=-2\la^2\langle Ja,b\rangle Jc
			+\la^2\langle b,c\rangle a
			-\la^2\langle b,Jc\rangle Ja\\
			&\quad
			-\la^2\langle a,c\rangle b
			+\la^2\langle a,Jc\rangle Jb\\
			&=-2\la^2\langle Ja,b\rangle Jc
			-\la^2(a\wedge b)c
			-\la^2(Ja\wedge Jb)c\\
			&=Q(a,b)c.
		\end{align*}
		Thus all the relations in~\eqref{timelikeeq} are satisfied.
		
		It remains to recover the bracket of $\G$. Set $h_0=h$ and $h_1=Jh$. Since $B=\la J$, Proposition~\ref{Z} gives
		\[
		[a,b]=[a,b]_\star+2\langle Ba,b\rangle e.
		\]
		Consequently,
		\[
		[h_0,a]=\la a+S(a),\qquad
		[h_0,h_1]=2\la(h_1+e),\qquad
		[a,b]=2\la\langle Ja,b\rangle(h_1+e),
		\]
		and $k=-\la^2$. This completes the proof.
	\end{proof}

\subsection{Lorentzian Lie algebras of nonzero constant curvature with degenerate center}

We first derive a convenient decomposition of the Levi-Civita product
for Lorentzian Lie algebras of nonzero constant curvature with
degenerate center.

\begin{pr}\label{ZD}
	Let $(\G,[\, ,\, ],\langle\,,\,\rangle)$ be a Lorentzian Lie algebra
	of nonzero constant curvature $k$ whose center is degenerate. Then:
	\begin{enumerate}
		\item $\G$ admits a Witt decomposition
		\[
		\G=\R e\oplus\h\oplus\R f,
		\]
		where $\h=(\R e\oplus\R f)^\perp$ is Euclidean,
		$\langle e,e\rangle=\langle f,f\rangle=0$, and
		$\langle e,f\rangle=1$.
		
		\item There exist linear maps $A,F:\h\longrightarrow\h$, with
		$F$ skew-symmetric, vectors $v,w\in\h$, and a bilinear product
		$\star:\h\times\h\longrightarrow\h$ such that
		$(\h,[\, ,\,]_\star,\langle\,,\,\rangle)$ is a Euclidean metric
		Lie algebra with Levi-Civita product $\star$, and the Levi-Civita
		product of $\G$ is given, for every $a,b\in\h$, by
		\begin{equation}\label{LeviZ}
			\begin{cases}
				e\cdot a=a\cdot e=\langle w,a\rangle e,\\[1mm]
				a\cdot b=a\star b+\langle Aa,b\rangle e,\\[1mm]
				e\cdot e=0,\qquad
				e\cdot f=f\cdot e=-w,\qquad
				f\cdot f=v,\\[1mm]
				f\cdot a=Fa-\langle a,v\rangle e+\langle a,w\rangle f,\\[1mm]
				a\cdot f=-Aa-\langle w,a\rangle f.
			\end{cases}
		\end{equation}
		The possibly nonzero Lie brackets are
		\begin{equation}\label{bracket} [a,b]=[a,b]_\star+\langle (A-A^*)a,b\rangle e
			\esp [f,a]=Fa+Aa-\langle a,v\rangle e+2\langle a,w\rangle f,
		\end{equation} for any $a,b\in\h$ and $[a,b]_\star=a\star b-b\star a$.

		\item The constant-curvature identity
		$\Ki(x,y)=k\,x\wedge y$, for all $x,y\in\G$, is equivalent to
		the following system, valid for all $a,b\in\h$:
		\begin{equation}\label{Zdeq}
			\begin{cases}
				k=-|w|^2,\;\;
				\langle w,a\rangle w+a\star w=-ka,\\
				\Ki^\star(a,b)=k\,a\wedge b,\\
				A([a,b]_\star)+b\star Aa-a\star Ab
				=\langle(A-A^*)b,a\rangle w
				+\langle a,w\rangle Ab-\langle b,w\rangle Aa,\\
				\langle[a,b]_\star,w\rangle=0,\\
				[F,\Lu_a]=\Lu_{Fa+Aa}
				+2\langle a,w\rangle F+Aa\wedge w,\\
				[F,A]=A^2-3\langle\bullet,w\rangle v
				-\langle\bullet,v\rangle w-\Ru_v,\\
				Fw=0.
			\end{cases}
		\end{equation}
		Here $\Lu_a$ and $\Ru_a$ denote, respectively, the left and right
		multiplications by $a$ with respect to $\star$, and
		\[
		\Ki^\star(a,b)
		=\Lu_{[a,b]_\star}-[\Lu_a,\Lu_b]
		\]
		is the curvature of the Euclidean metric Lie algebra
		$(\h,[\, ,\,]_\star,\langle\,,\,\rangle)$.
	\end{enumerate}
\end{pr}

\begin{proof}
	Since the center is degenerate, Proposition~\ref{center} implies that
	$Z(\G)=\R e$, where $e$ is a nonzero null vector. Choose a
	complementary subspace $\h$ of $\R e$ in $e^\perp$, so that
	\[
	e^\perp=\R e\oplus\h.
	\]
	Since $\G$ is Lorentzian, $\h$ is Euclidean. Moreover, $\h^\perp$ is
	a two-dimensional Lorentzian subspace containing $e$. We may therefore
	choose a null vector $f\in\h^\perp$ such that
	$\langle e,f\rangle=1$. Hence
	\[
	\G=\R e\oplus\h\oplus\R f.
	\]
	
	Since $e$ is central, $e\cdot e=0$ and
	$\Lu_e=\Ru_e$. Using the metric compatibility of the Levi-Civita
	product and the above Witt decomposition, there exist
	$D,A,F:\h\longrightarrow\h$, with $D$ and $F$ skew-symmetric,
	$v,w\in\h$, and a bilinear product $\star$ on $\h$ such that
	\[
	\begin{cases}
		e\cdot a=a\cdot e=Da+\langle w,a\rangle e,\\
		a\cdot b=a\star b+\langle Aa,b\rangle e-\langle Da,b\rangle f,\\
		e\cdot e=0,\qquad
		e\cdot f=f\cdot e=-w,\qquad
		f\cdot f=v,\\
		f\cdot a=Fa-\langle a,v\rangle e+\langle a,w\rangle f,\\
		a\cdot f=-Aa-\langle w,a\rangle f.
	\end{cases}
	\]
	Moreover, $\star$ is metric:
	\[
	\langle a\star b,c\rangle+\langle b,a\star c\rangle=0,
	\qquad a,b,c\in\h.
	\]
	We claim that $D=0$. Indeed, since $e$ is central, for any
	$a,b\in\h$,
	\begin{align*}
		\langle\Ki(a,e)b,e\rangle
		&=-\langle a\cdot(e\cdot b),e\rangle
		+\langle e\cdot(a\cdot b),e\rangle\\
		&=\langle e\cdot b,a\cdot e\rangle
		-\langle a\cdot b,e\cdot e\rangle\\
		&=\langle Db,Da\rangle
		=-\langle D^2a,b\rangle.
	\end{align*}
	On the other hand,
	$
	\langle k(a\wedge e)b,e\rangle=0.
	$
	Thus $D^2=0$. Since $D$ is skew-symmetric on the Euclidean space
	$\h$, it follows that $D=0$. This gives \eqref{LeviZ}.
	
	Now, for $a,b\in\h$,
	\[
	[a,b]=[a,b]_\star
	+\langle(A-A^*)a,b\rangle e.
	\]
	Since $e$ is central, the $\h$-component of the Jacobi identity in
	$\G$ shows that $[\, ,\,]_\star$ satisfies the Jacobi identity.
	Moreover,
	\[
	[a,b]_\star=a\star b-b\star a,
	\]
	and the left multiplications of $\star$ are skew-symmetric.
	Consequently, $\star$ is the Levi-Civita product of the Euclidean
	metric Lie algebra $(\h,[\, ,\,]_\star,\langle\,,\,\rangle)$.
	
	Finally, substituting \eqref{LeviZ} into
	$\Ki(x,y)=k\,x\wedge y$ yields precisely the system
	\eqref{Zdeq}; the corresponding computations are carried out in
	Section~\ref{section9}, Subsection~\ref{subsection1}.
\end{proof}

The preceding proposition shows that the determination of Lorentzian
Lie algebras of nonzero constant curvature with degenerate center
reduces to determining the Euclidean metric Lie algebra
$(\h,\br_\star,\prs)$ together with the quadruple $(A,F,v,w)$
satisfying \eqref{Zdeq}. We now solve this system by distinguishing three cases: $\dim\h=1$, $\dim\h=2$ and $\dim\h\geq2$.

\subsubsection{The case $\dim\h=1$}

Write $\h=\R g$, with $|g|=1$. Then $g\star g=0$, $F=0$,
$w=\al g$, $v=\be g$, and $Ag=cg$. Since
$k=-|w|^2=-\al^2\neq0$, we have $\al\neq0$. The sixth relation
in \eqref{Zdeq}, applied to $g$, yields
\[
c^2=4\al\be.
\]
Thus $\al\be\geq0$, and we may write
$c=2\e\sqrt{\al\be}$, where $\e^2=1$. Hence
$\G=\mathrm{span}\{e,f,g\}$ and, by \eqref{bracket}, its only
nonzero Lie bracket is
\[
[f,g]=-\be e+2\al f+2\e\sqrt{\al\be}\,g.
\]
Moreover, in the basis $(e,f,g)$, the metric is given by
\[
M(\prs,\{e,f,g\})=
\begin{pmatrix}
	0&1&0\\
	1&0&0\\
	0&0&1
\end{pmatrix}.
\]

Put
$
h=-\be e+2\al f+2\e\sqrt{\al\be}\,g.
$
Then $[f,g]=h$ and a direct computation gives
\[
[g,h]=-2\al h.
\]
Consider now the basis $\mathcal B=(E_1,E_2,E_3)$ defined by
\[
E_1=-\frac{1}{2\al}g
+\frac{\e\sqrt{\al\be}}{2\al^2}e,\qquad
E_2=\frac{1}{2\al}h
=f-\frac{\be}{2\al}e
+\frac{\e\sqrt{\al\be}}{\al}g,\qquad
E_3=e.
\]
Since $e$ is central,
$[E_1,E_3]=[E_2,E_3]=0$, whereas
\[
[E_1,E_2]
=-\frac{1}{4\al^2}[g,h]
=\frac{1}{2\al}h
=E_2.
\]
Thus the only nonzero Lie bracket in the basis $\mathcal B$ is
\[
[E_1,E_2]=E_2,
\]
and therefore
$
\G\simeq\aff(\R)\oplus\R.
$
Finally, using $c^2=4\al\be$, or equivalently
$\e^2\al\be=\al\be$, we obtain
\[
\operatorname{Mat}_{\mathcal B}
\langle\,\cdot\,,\,\cdot\,\rangle
=
\begin{pmatrix}
	\dfrac{1}{4\al^2}&0&0\\[2mm]
	0&0&1\\
	0&1&0
\end{pmatrix}.
\]

\subsubsection{Preliminary reductions for $\dim\h\geq2$}

One of the important consequence of \eqref{Zdeq} is that $(\h,\br_\star,\prs)$ is a Euclidean Lie algebra of nonzero constant curvature $k$. According to Theorem \ref{constant}, there exists $E=B+\la \mathrm{Id}_\h$ with $B\in\so(\h)$  and a unit vector $u\in\h$  such that $k=-\la^2$ and, for any $a,b\in\h$,
\[ [a,b]_\star=\langle a,u\rangle E(b)-\langle b,u\rangle E(a)\esp B(u)=0. \]
Moreover, the Levi-Civita product $\star$ is given by
\begin{equation}\label{lv} \Lu_a=\langle a,u\rangle B+\la a\wedge u,\quad a\in\h. \end{equation}

Before tackling the solutions of \eqref{Zdeq},  note first that $[\h,\h]_\star^\perp=\R u.$
Indeed, $c\in[\h,\h]_\star^\perp$ if and only if
\[
\langle{a},{u}\rangle E^*c-\langle{E(a)},{c}\rangle u=0
\qquad(a\in\h).
\]
Taking $a=u$ and using $Bu=0$ gives
\[
-Bc+\lambda c-\lambda\langle{u},{c}\rangle u=0.
\]
Taking the scalar product with $c$, we obtain
\[
\lambda\bigl(|c|^2-\langle{u},{c}\rangle^2\bigr)=0.
\]
Since $\lambda\neq0$ and $\h$ is Euclidean, $c$ is proportional to
$u$, proving the desired relation.

The first consequence of \eqref{Zdeq} is that $w=-\la u$. Indeed, from
the fourth equation in \eqref{Zdeq}, $w\in[\h,\h]_\star^\perp$ and hence $w=\al u$.  By using \eqref{lv}, the relation $\langle w,a\rangle w+a\star w=-k a$ writes
\[ \al^2\langle u,a\rangle u+\al \langle a,u\rangle B(u)
+\la\al (\langle a,u\rangle u-\langle u,u\rangle a)= \la^2 a.\]
So $\al=-\la$ and hence $w=-\la u$. This implies that $k=-\la^2=-|w|^2$ and one can check easily that, for any $a\in\h$, $\langle w,a\rangle w+a\star w=-k a$.

Let us simplify the third equation in \eqref{Zdeq}:
\[ A([a,b]_\star)+b\star Aa-a\star Ab=
\langle (A-A^*)b,a\rangle w+\langle a,w\rangle Ab-\langle b,w\rangle Aa. \]
The left hand $LH$ and the right hand $RH$ of this equation are given by
\begin{align*}
	LH&=A([a,b]_\star)+b\star Aa-a\star Ab\\
	&=\langle a,u\rangle A(E(b))-\langle b,u\rangle A(E(a))
	+\langle b,u\rangle BAa+\la \langle b,Aa\rangle u-\la \langle u,Aa\rangle b	-\langle a,u\rangle BAb-\la \langle a,Ab\rangle u+\la \langle u,Ab\rangle a,\\
	&=\langle a,u\rangle\left([A,B](b)+\la Ab\right)
	-\langle b,u\rangle\left([A,B](a)+\la Aa\right)+\la \langle b,Aa\rangle u-\la \langle u,Aa\rangle b	-\la \langle a,Ab\rangle u+\la \langle u,Ab\rangle a\\
	RH&=-\la	\langle (A-A^*)b,a\rangle u-\la\langle a,u\rangle Ab+\la \langle b,u\rangle Aa\\
	&=-\la	\langle Ab,a\rangle u+\la \langle Aa,b\rangle u-\la\langle a,u\rangle Ab+\la \langle b,u\rangle Aa.
\end{align*}
So the equation is equivalent to
\begin{equation}\label{eq6} \langle a,u\rangle ([A,B](b)+2\la Ab)
	-\langle b,u\rangle ([A,B](a) +2\la Aa )+\la(\langle u,Ab\rangle a-\langle u,Aa\rangle b)=0,\end{equation}
for any $a,b\in\h$.

We distinguish now the case $\dim\h=2$ and the case $\dim\h\geq3$.

\subsubsection{The case $\dim\h=2$}

Choosing $a\in\h$ such that $(u,a)$ is an orthonormal basis. Since $Bu=Fw=0$ and $B,F\in\so(\h)$  one must have $B=F=0$
and the  equation \eqref{eq6}
is equivalent to
\[ -2 Aa+\langle u,Au\rangle a-\langle u,Aa\rangle u=0.  \]
But $Aa=\langle Aa,u\rangle u+\langle Aa,a\rangle 	a$ and hence
\[ (\langle u,Au\rangle-2\langle Aa,a\rangle)a 	-3\langle u,A	a\rangle u=0. \]
Hence $\langle u,A	a\rangle=0$ and $\langle u,Au\rangle=2\langle Aa,a\rangle$. Consequently, the matrix of $A$ in the basis $\B=(u,a)$ has the form
\[ \mathrm{Mat}(A,\B)=\begin{pmatrix}2x&0\\
	y&x
\end{pmatrix},\; x,y\in\R. \]
 Having in mind \eqref{lv},	the fifth relation in \eqref{Zdeq} holds. The sixth relation in \eqref{Zdeq} is equivalent to
\[ \begin{cases}
	A^2u+3\la  v+\la \langle u, v\rangle u=0,\\
	A^2a+\la \langle u,v\rangle a=0.
\end{cases} \]
But $A^2u=4x^2u+3xy a$ and $A^2a=x^2a$.
So
\[ \langle u,v\rangle=-\frac{x^2}\la,\; \langle v,a\rangle =-\frac1{3\la}\langle A^2u,a\rangle=-\frac{xy}{\la}\esp v=-\frac{x}\la(xu+ya). \]
\[  \]
In this case, $\G=\mathrm{span}\{e,a,u,f\}$ and, by virtue of \eqref{bracket}, the nonvanishing Lie brackets are
\[ [a,u]=-\la a-y e,\; [f,a]=xa+\frac{xy}\la e,\;[f,u]=2x u+y a+\frac{x^2}\la e-2\la f. \] 

Consider the following basis:
\[
E_1=
f-\frac{y}{\lambda}a-\frac{x}{\lambda}u
-\frac{x^2+y^2}{2\lambda^2}e,\;
E_2=
a+\frac{y}{\lambda}e,\;
E_3=
-\frac{1}{2\lambda}u-\frac{x}{2\lambda^2}e,\;
E_4=e.
\]

We first compute the Lie brackets. Since $e$ is central, we have
\[
[E_i,E_4]=0,\qquad i=1,2,3.
\]
Moreover,
\begin{align*}
	[E_2,E_3]
	=
	-\frac{1}{2\lambda}[a,u]
	=
	\frac12a+\frac{y}{2\lambda}e
	=\frac12E_2.
\end{align*}
A direct computation also gives
\[
[E_1,E_3]=E_1,
\qquad
[E_1,E_2]=0.
\]
Therefore, the only nonzero Lie brackets in the basis
$\mathcal B=(E_1,E_2,E_3,E_4)$
are
\[
[E_1,E_3]=E_1,\qquad
[E_2,E_3]=\frac12E_2.
\]
Consequently,
$\G\simeq\mathfrak r_{3,\frac12}\oplus\R.
$

A straightforward computation gives
\[	\operatorname{Mat}_{\mathcal B}
\langle\cdot,\cdot\rangle
=
\begin{pmatrix}
	0&0&0&1\\[2mm]
	0&1&0&0\\[2mm]
	0&0&\dfrac{1}{4\lambda^2}&0\\[2mm]
	1&0&0&0
\end{pmatrix}.
\]

\subsubsection{The case $\dim\h\geq3$}

Consider the equation \eqref{eq6}.
If we take $a,b\in u^\perp$ such that $a\wedge b\not=0$, we get 
$\langle Aa,u\rangle=\langle Ab,u\rangle=0$ so $A(u^\perp)\subset u^\perp$. Put $Au=\al u+a_0$ where $a_0\in u^\perp$ and $\al=\langle Au,u\rangle$.

If we take $a=u$ and $b\in u^\perp$ in \eqref{eq6}, we get
\[  [A,B](b)+2\la Ab
-\la\al b=0. \]
So in restriction to $u^\perp$,
\[ [B,A-\frac12 \al \mathrm{Id}_{u^\perp}]=2\la (A-\frac12 \al \mathrm{Id}_{u^\perp}). \]
By using Lemma \ref{adB}, we get
\[ A_{|u^\perp}=\frac12 \al  \mathrm{Id}_{u^\perp}\esp Au=\al u+a_0. \]

Let us solve the fifth equation in \eqref{Zdeq}:
\[ [F,\Lu_a]=\Lu_{Fa+Aa}+2\langle a,w\rangle F+Aa\wedge w. \]
Having in mind that $w=-\la u$, $Fw=Fu=0$, $F\in\so(\h)$, \eqref{lv}, and the relation \eqref{F},
we get
\begin{align*}
	[F,\Lu_a]&=\langle a,u\rangle [F,B]+\la [F,a\wedge u]=\langle a,u\rangle [F,B]+\la
	Fa\wedge u,\\
	\Lu_{Fa+Aa}&=\langle Fa+Aa,u\rangle B+\la Fa\wedge u+\la Aa\wedge u=
	\langle Aa,u\rangle B+\la Fa\wedge u+\la Aa\wedge u.
\end{align*}So the fifth equation in \eqref{Zdeq} is equivalent to
\[ \langle a,u\rangle [F,B]=\langle Aa,u\rangle B-2\la \langle a,u\rangle F. \]
In  particular, for $a=u$,
\[ [B,F-\frac{\al}{2\la} B]=2\la\left(F-\frac{\al}{2\la} B\right). \]
By virtue of Lemma \ref{adB},  we get
$F=\frac{\al}{2\la} B.$

Let us tackle the sixth  equation in \eqref{Zdeq}. Note first that, by virtue of \eqref{lv},
\[ \Ru_v=\langle \bullet,u\rangle B(v)+\la\langle\bullet,v\rangle u-\la\langle u,v\rangle \mathrm{Id}_\h. \]
We have also $w=-\la u$. So the fifth equation in\eqref{Zdeq} can by written
\[ [F,A]=A^2+3\la \langle \bullet, u\rangle v
-\langle \bullet,u\rangle B(v)+\la\langle u,v\rangle \mathrm{Id}_\h.
 \]
Applying this identity to $a\in u^\perp$, we get
\[ \frac14\al^2a+\la \langle u,v\rangle a=0 \]
so $\langle u,v\rangle =-\frac{\al^2}{4\la}$. On the other hand, if we apply to $u$, we get
\[ [F,A](u)=A^2u+3\la v-B(v)-\frac{\al^2}{4}u. \]
Since $F(u)=0$ and $Au=\al u+a_0$, we get
\[
	[F,A](u)=\frac{\al}{2\la} B(a_0)\esp 
	A^2u=\al^2u+\frac32\al a_0.
\]
Thus
\[ (3\la \mathrm{Id_\h}-B)(v)=\frac{\al}{2\la} B(a_0)-\frac32\al a_0-\frac34\al^2u. \]
Put $v=-\frac{\al^2}{4\la} u+v_0$. Then
\[ (3\la \mathrm{Id_\h}-B)(v_0)=-\frac{\al}{2\la}(3\la \mathrm{Id}_{\h}-B)(a_0). \]
Since $B$ is skew-symmetric and $\lambda\neq0$,
$3\lambda\mathrm{Id}_\h-B$ is invertible. Consequently, $v_0=-\frac{\al}{2\la}a_0$. Finally,
\[ \begin{cases}
	F=\frac{\al}{2\la} B,\\
	A_{|u^\perp}=\frac{\al}2  \mathrm{Id}_{u^\perp}\esp Au=\al u+a_0,\\
	w=-\la u,\; v=-\frac{\al^2}{4\la} u-\frac{\al}{2\la}a_0.
\end{cases} \]

For completeness, $\G=\R e\oplus(u^\perp\oplus\R u)\oplus\R f$ and, by virtue of \eqref{bracket}, the corresponding possibly nonzero brackets, for
$a\in u^\perp$, are
\begin{equation*}\label{eq:brackets-high-dim}
	\begin{cases}
		[u,a]=Ba+\lambda a+\langle{a_0},{a}\rangle e,\\
		[f,a]=\frac\alpha{2\lambda}
		\bigl(Ba+\lambda a+\langle{a_0},{a}\rangle e\bigr),\\
		[f,u]=\alpha u+a_0+\frac{\alpha^2}{4\lambda}e-2\lambda f.
	\end{cases}
\end{equation*}

To summarize, we have shown the following theorem.

\begin{theo}\label{ZDC}
	Let $(\G,[\,\cdot\,,\,\cdot\,],\langle\,\cdot\,,\,\cdot\,\rangle)$ be a
	Lorentzian Lie algebra of constant sectional curvature $k$ whose center is
	degenerate. Then $k<0$, and one of the following mutually exclusive cases
	occurs.
	\begin{enumerate}
		\item The Lie algebra $\G$ is three-dimensional and isomorphic to
		$\aff(\R)\oplus\R$. More precisely, $\G$ admits a basis
		$\mathcal B=(E_1,E_2,E_3)$ in which its only nonzero Lie bracket is
		\[
		[E_1,E_2]=E_2,
		\]
		and the Gram matrix of the metric is
		\[
		\operatorname{Mat}_{\mathcal B}\langle\,\cdot\,,\,\cdot\,\rangle
		=
		\begin{pmatrix}
			\dfrac{1}{4\alpha^2}
			&
			0
			&
			0
			\\[2mm]
			0
			&
			0
			&
			1
			\\[2mm]
			0&1&0
		\end{pmatrix}.
			\]
		In this case,
		$
		k=-\alpha^2.
		$
		
		\item The Lie algebra $\G$ is four-dimensional and isomorphic to
		$\mathfrak r_{3,\frac12}\oplus\R.$
		More precisely, $\G$ admits a basis
		$\mathcal B=(E_1,E_2,E_3,E_4)$ in which the only nonzero brackets are
		\[
		[E_1,E_3]=E_1,
		\qquad
		[E_2,E_3]=\frac12E_2,
		\]
		and the Gram matrix of the metric is
		\[
		\operatorname{Mat}_{\mathcal B}\langle\,\cdot\,,\,\cdot\,\rangle
		=
		\begin{pmatrix}
			0&0&0&1\\[2mm]
			0&1&0&0\\[1mm]
			0&0&\dfrac{1}{4\lambda^2}&0\\[2mm]
			1&0&0&0
		\end{pmatrix},\quad \lambda\neq0.
			\]
		In this case, $k=-\lambda^2.$
		
		\item The dimension of $\G$ is at least five. There exists a Euclidean vector
		space $\h$, vectors $e,f\in\G$, a unit vector $u\in\h$, a vector
		$a_0\in u^\perp\subset\h$, a skew-symmetric endomorphism
		$B\colon\h\to\h$, and constants
		$\lambda>0, \alpha\in\R,$
		such that
		$\dim\h\geq3,
		\; Bu=0,$
		and
		$\G=\R e\oplus\h\oplus\R f$
		is an orthogonal Witt decomposition in the following sense:
		\[
		\langle e,e\rangle=\langle f,f\rangle=0,
		\qquad
		\langle e,f\rangle=1,
		\qquad
		\h=(\mathrm{Span}\{e,f\})^\perp.
		\]
		For every $a\in u^\perp$, the possibly nonzero brackets are
		\begin{equation}\label{eq:brackets-ZDC}
			\left\{
			\begin{aligned}
				{}[u,a]
				&=Ba+\lambda a+\langle a_0,a\rangle e,\\
				{}[f,a]
				&=\frac{\alpha}{2\lambda}
				\bigl(Ba+\lambda a+\langle a_0,a\rangle e\bigr),\\
				{}[f,u]
				&=\alpha u+a_0+\frac{\alpha^2}{4\lambda}e-2\lambda f.
			\end{aligned}
			\right.
		\end{equation}
		All brackets not obtained from \eqref{eq:brackets-ZDC} by skew-symmetry are
		zero. In this case,
		$k=-\lambda^2.$
	\end{enumerate}
\end{theo}

\section{Lorentzian Lie algebras of nonzero constant curvature whose derived ideal is nondegenerate and Euclidean}\label{section5}

In this section, we address the case where the derived ideal is nondegenerate and Euclidean (See Theorem \ref{D-Euclidean}). The case of a nondegenerate Lorentzian derived ideal is considerably more intricate. We give a complete description in dimension four (see Proposition \ref{prop:Lorentzian-derived-low-dimensional}), which already reveals the substantial difficulties inherent in extending the analysis to arbitrary dimension.

\begin{theo}\label{D-Euclidean} Let $(\G,\br,\prs)$ be a Lorentzian Lie algebra of nonzero constant
	sectional curvature $k$ such that $[\G,\G]\neq\G$ and the
	derived ideal $[\G,\G]$ is nondegenerate and Euclidean. Then $(\G,\br,\prs)$ is isomorphic to the model described in Proposition~\ref{euclideanbis} with $\langle u,u\rangle<0$.
	\end{theo}

\begin{proof}

Choose a unit timelike vector
$e\in[\G,\G]^\perp,\;
\langle e,e\rangle=-1,
$
and put $\h=e^\perp$. Since $[\G,\G]\subset\h$, the subspace $\h$ is
an ideal of $\G$. Moreover, the Levi-Civita product takes the form
\begin{equation}\label{eq}
	e\cdot a=Ea,\qquad
	a\cdot e=Sa,\qquad
	a\cdot b=a\star b+\langle Sa,b\rangle e,
\end{equation}
where $E\in\so(\h)$, $S$ is self-adjoint,
$E-S=\ad_e|_\h$ is a derivation of $\h$, and $\star$ is the
Levi-Civita product of the induced Euclidean metric Lie algebra
$(\h,\br_\star,\prs)$ where $[a,b]_\star=a\star b-b\star a$.

By a straightforward computation (See Section \ref{section9}, Subsection \ref{subsection2}),	the constant-curvature identity
	$\Ki(x,y)=k\,x\wedge y,
	x,y\in\G,$
	is equivalent to the following system of identities, valid for all
	$a,b\in\h$:
	\begin{equation}\label{eqgg1}
		\begin{cases}
			\Ki^*(a,b)
			=-Sa\wedge Sb+k\,a\wedge b,\\
			S\bigl([a,b]_\star\bigr)+b\star Sa-a\star Sb
			=0,\\
			Sa\star b+E(a\star b)-a\star Eb-Ea\star b
			=0,\\
			[S,E]
			=S^2- k\,\mathrm{Id}_{\h}.
		\end{cases}
	\end{equation}
	Here $\Ki^*$ denotes the curvature tensor of
	$(\h,\br_\star,\prs)$.

Taking the trace in the last identity of \eqref{eqgg1}, and using
$\operatorname{tr}[S,E]=0$, we obtain $\operatorname{tr}(S^2)=nk,$
where $n=\dim\h$. Since $\h$ is Euclidean and $S$ is symmetric,
$\operatorname{tr}(S^2)>0$; hence $k>0$.

Let $x\in\h$ be an eigenvector of $S$ with eigenvalue $\lambda$.
Applying the last identity of \eqref{eqgg1} 
to $x$, we obtain
\[
S(Ex)-\lambda Ex=(\lambda^2-k)x.
\]
Taking the scalar product with $x$, and using the fact that $S$ is
symmetric whereas $E$ is skew-symmetric, gives
\[
(\lambda^2-k)\langle x,x\rangle=0.
\]
Since the metric on $\h$ is positive definite, it follows that
$\lambda^2=k$. Consequently,
\[
S^2=k\Id_{\h}
\qquad\text{and}\qquad
[S,E]=0.
\]
Now $E-S$ is a derivation of $\h$ and commutes with $S$. Therefore,
by \cite[Lemma~1]{Heintz}, $S$ is itself a derivation. The spectral
decomposition of $S$ is
\[
\h=\h_{\sqrt{k}}\oplus\h_{-\sqrt{k}},
\]
where $\h_{\pm\sqrt{k}}$ denotes the eigenspace corresponding to the
eigenvalue $\pm\sqrt{k}$. Since $S$ is a derivation, we have
\[
[\h_\mu,\h_\nu]\subseteq\h_{\mu+\nu}.
\]
However, $S$ has no eigenvalues $0$ or $\pm2\sqrt{k}$. It follows that
all brackets between its eigenspaces vanish, and hence $\h$ is
abelian and $\Ki^*=0$.

The first identity in \eqref{eqgg1} therefore reduces to
\[
-Sa\wedge Sb+k\,a\wedge b=0.
\]
If both eigenspaces were nonzero, we could choose nonzero vectors
$a\in\h_{\sqrt{k}}$ and $b\in\h_{-\sqrt{k}}$. The preceding identity
would then yield
\[
2k\,a\wedge b=0,
\]
which is impossible because $a$ and $b$ belong to distinct eigenspaces
and are therefore linearly independent. Thus one of the two
eigenspaces must vanish, and consequently
\[
S=\pm\sqrt{k}\,\Id_{\h}.
\]

Finally, $E$ is an arbitrary skew-symmetric endomorphism of $\h$, and
$
\G=\R e\oplus\h,
$
where $\h$ is abelian and the only possibly nonzero Lie brackets are
\[
[e,a]=(E-S)a=E(a)+\lambda a,
\qquad a\in\h\esp \lambda=\mp\sqrt{k}.
\]
We thus recover precisely the family described in
Proposition~\ref{euclideanbis}.
\end{proof}

\section{Lorentzian Lie algebras $\G$ of nonzero constant curvature whose derived ideal is nondegenerate and Lorentzian: the case $\dim\G\in\{3,4\}$ }\label{section5bis}

The classification of Lorentzian Lie algebras of nonzero constant sectional curvature whose derived ideal is nondegenerate and Lorentzian appears to be a particularly challenging problem in arbitrary dimension. Unlike the case where the derived ideal is Euclidean, the indefinite signature introduces substantial difficulties. In particular, self-adjoint endomorphisms need not be diagonalizable, skew-adjoint endomorphisms may have nonzero real eigenvalues, and nilpotent Lorentzian Lie algebras may carry flat left-invariant metrics without being abelian. 

In this section, we restrict our attention to dimensions three and four, where a detailed analysis of the possible normal forms of the relevant endomorphisms allows us to obtain a complete classification. Even in these low dimensions, the arguments require a careful examination of several algebraic configurations, illustrating the obstacles that arise in the general case. 

\begin{Le}\label{lem:lorentzian-normal-forms}
	Let $(V,\langle\cdot,\cdot\rangle)$ be a real Lorentzian vector space of dimension $2$ or $3$. Let $k\in\R\setminus\{0\}$, and suppose that $S,E\in\operatorname{End}(V)$ satisfy
	\[
	S^*=S,\qquad E^*=-E,\qquad [S,E]=S^2+k\Id_V.
	\]
	Assume that $S$ is not diagonalizable over $\R$. Then exactly one of the following Jordan types occurs.
	
	\smallskip\noindent\textbf{(I) Dimension $2$.} There exists a basis $\B$ of $V$ such that either
	\[
	\Mat(\prs,\B)=\begin{pmatrix}1&0\\0&-1\end{pmatrix},\qquad
	\Mat(S,\B)=\begin{pmatrix}0&a\\-a&0\end{pmatrix},\qquad E=0,
	\qquad a\ne0,\quad a^2=k>0,
	\]
	or
	\[
	\Mat(\prs,\B)=\begin{pmatrix}0&1\\1&0\end{pmatrix},\qquad
	\Mat(S,\B)=\begin{pmatrix}a&1\\0&a\end{pmatrix},\qquad
	\Mat(E,\B)=\begin{pmatrix}-a&0\\0&a\end{pmatrix},
	\qquad a^2=-k>0.
	\]
	
	\smallskip\noindent\textbf{(II) Dimension $3$.} Necessarily $k<0$, and there exists a basis $\B$ of $V$ such that either 
	\begin{equation*}\label{eq:2block}
		\Mat(\prs,\B)=\begin{pmatrix}1&0&0\\0&0&1\\0&1&0\end{pmatrix},\quad
		\Mat(S,\B)=\begin{pmatrix}a&0&0\\0&b&1\\0&0&b\end{pmatrix},\quad
		\Mat(E,\B)=\begin{pmatrix}0&0&0\\0&-b&0\\0&0&b\end{pmatrix},
		\qquad a^2=b^2=-k,
	\end{equation*}
	or
	\begin{equation*}\label{eq:3block}
		\Mat(\prs,\B)=\begin{pmatrix}0&0&1\\0&1&0\\1&0&0\end{pmatrix},\quad
		\Mat(S,\B)=\begin{pmatrix}a&1&0\\0&a&1\\0&0&a\end{pmatrix},\quad
		\Mat(E,\B)=\begin{pmatrix}-2a&-\frac12&0\\0&0&\frac12\\0&0&2a\end{pmatrix},
		\qquad a^2=-k.
	\end{equation*}
\end{Le}

\begin{proof}
	We begin with two observations. Taking the trace of
	\[
	[S,E]=S^2+k\Id_V
	\]
	gives
	\begin{equation}\label{eq:trace-SE}
		\tr(S^2)=-k\dim V.
	\end{equation}
	Moreover, if $v$ is an eigenvector of $S$ with eigenvalue
	$\lambda$, then
	\[
	\langle [S,E]v,v\rangle
	=\langle Ev,Sv\rangle-\langle SEv,v\rangle=0.
	\]
	Consequently,
	\[
	(\lambda^2+k)\langle v,v\rangle=0.
	\]
	Thus, whenever $v$ is non-isotropic,
	\begin{equation}\label{eq:eigen-SE}
		\lambda^2=-k.
	\end{equation}
	
	We now examine the possible non-diagonalizable normal
	forms of a self-adjoint endomorphism of a Lorentzian
	vector space.
	
	\smallskip
	\noindent\textbf{(I) Dimension $2$.}
	There are two possibilities.
	
	\smallskip
	\noindent\emph{Case 1: $S$ has non-real eigenvalues.}
	There exists a pseudo-orthonormal basis in which
	\[
	G=\begin{pmatrix}1&0\\0&-1\end{pmatrix},
	\qquad
	S=\begin{pmatrix}a&b\\-b&a\end{pmatrix},
	\qquad b\ne0.
	\]
	Since $E$ is skew-adjoint, it has the form
	\[
	E=\begin{pmatrix}0&t\\t&0\end{pmatrix}.
	\]
	A direct computation yields
	\[
	[S,E]=
	\begin{pmatrix}-2bt&0\\0&2bt\end{pmatrix}
	\]
	and
	\[
	S^2+k\Id=
	\begin{pmatrix}
		a^2-b^2+k&2ab\\
		-2ab&a^2-b^2+k
	\end{pmatrix}.
	\]
	Comparing the entries gives
	\[
	ab=0,\qquad bt=0,\qquad a^2-b^2+k=0.
	\]
	Since $b\ne0$, we obtain
	\[
	a=t=0,\qquad b^2=k>0.
	\]
	This is the first normal form.
	
	\smallskip
	\noindent\emph{Case 2: $S$ has a nontrivial Jordan block.}
	There exists a null basis in which
	\[
	G=\begin{pmatrix}0&1\\1&0\end{pmatrix},
	\qquad
	S=\begin{pmatrix}a&1\\0&a\end{pmatrix}.
	\]
	Skew-adjointness implies
	\[
	E=\begin{pmatrix}t&0\\0&-t\end{pmatrix}.
	\]
	Hence
	\[
	[S,E]=\begin{pmatrix}0&-2t\\0&0\end{pmatrix},
	\qquad
	S^2+k\Id=
	\begin{pmatrix}a^2+k&2a\\0&a^2+k\end{pmatrix}.
	\]
	It follows that
	\[
	a^2=-k,\qquad t=-a.
	\]
	Since $k\ne0$, necessarily $a\ne0$ and $k<0$.
	This gives the second normal form.
	
	\smallskip
	\noindent\textbf{(II) Dimension $3$.}
	The classification of self-adjoint endomorphisms
	of Lorentzian three-dimensional spaces yields
	three possible non-diagonalizable types.
	
	\smallskip
	\noindent\emph{Case 1: $S$ has a pair of non-real
		conjugate eigenvalues.}
	There exists a pseudo-orthonormal basis in which
	\[
	G=\diag(1,1,-1),\qquad
	S=\begin{pmatrix}
		c&0&0\\
		0&a&b\\
		0&-b&a
	\end{pmatrix},
	\qquad b\ne0.
	\]
	The first basis vector is non-isotropic and is an
	eigenvector of $S$. By \eqref{eq:eigen-SE},
	\[
	c^2=-k.
	\]
	Since $S$ preserves the orthogonal decomposition
	\[
	V=\mathbb Re_1\oplus\Span\{e_2,e_3\},
	\]
	the restriction of the equation to the second summand
	is
	\[
	[S_0,E_0]=S_0^2+k\Id,
	\]
	where
	\[
	S_0=\begin{pmatrix}a&b\\-b&a\end{pmatrix}
	\]
	and $E_0$ is the corresponding skew-adjoint block of $E$.
	
	Indeed, the diagonal block of a commutator with a
	block-diagonal endomorphism depends only on the
	corresponding diagonal block of the other factor.
	The two-dimensional computation above therefore gives
	\[
	a=0,\qquad b^2=k>0.
	\]
	This contradicts $c^2=-k$.
	Hence this Jordan type cannot occur.
	
	\smallskip
	\noindent\emph{Case 2: $S$ has a Jordan block of size $2$.}
	There exists a basis in which
	\[
	G=\begin{pmatrix}
		1&0&0\\
		0&0&1\\
		0&1&0
	\end{pmatrix},
	\qquad
	S=\begin{pmatrix}
		a&0&0\\
		0&b&1\\
		0&0&b
	\end{pmatrix}.
	\]
	The first basis vector is non-isotropic, so
	\[
	a^2=-k.
	\]
	Furthermore, \eqref{eq:trace-SE} gives
	\[
	a^2+2b^2=-3k.
	\]
	Consequently,
	\[
	a^2=b^2=-k>0.
	\]
	
	The condition $E^TG+GE=0$ implies that
	\[
	E=\begin{pmatrix}
		0&x&y\\
		-y&t&0\\
		-x&0&-t
	\end{pmatrix}.
	\]
	Substituting into $[S,E]=S^2+k\Id$ yields
	\[
	t=-b,\qquad x=0,\qquad (a-b)y=0.
	\]
	Thus
	\[
	E=\begin{pmatrix}
		0&0&y\\
		-y&-b&0\\
		0&0&b
	\end{pmatrix}.
	\]
	
	If $a\ne b$, then $y=0$, and we obtain the
	announced normal form.
	
	Suppose now that $a=b$. Since $a^2=-k>0$,
	we have $a\ne0$. Consider the change of basis
	\[
	e'_1=e_1+ce_2,\qquad e'_2=e_2,\qquad
	e'_3=e_3-ce_1-\frac{c^2}{2}e_2.
	\]
	It preserves the Gram matrix $G$ and the matrix
	of $S$. In the new basis, the matrix of $E$ has
	the same form, with $y$ replaced by $y+ac$.
	Choosing
	\[
	c=-\frac ya
	\]
	eliminates this parameter. Hence the second
	Jordan type gives exactly the first normal form
	in dimension $3$.
	
	\smallskip
	\noindent\emph{Case 3: $S$ has a Jordan block of size $3$.}
	There exists a basis in which
	\[
	G=\begin{pmatrix}
		0&0&1\\
		0&1&0\\
		1&0&0
	\end{pmatrix},
	\qquad
	S=\begin{pmatrix}
		a&1&0\\
		0&a&1\\
		0&0&a
	\end{pmatrix}.
	\]
	Taking the trace gives
	\[
	3a^2=-3k,
	\]
	hence
	\[
	a^2=-k>0.
	\]
	
	Every skew-adjoint endomorphism has the form
	\[
	E=\begin{pmatrix}
		u&v&0\\
		w&0&-v\\
		0&-w&-u
	\end{pmatrix}.
	\]
	Writing $S=a\Id+N$, where
	\[
	N=\begin{pmatrix}
		0&1&0\\
		0&0&1\\
		0&0&0
	\end{pmatrix},
	\]
	the defining equation becomes
	\[
	[N,E]=2aN+N^2.
	\]
	A direct calculation gives
	\[
	[N,E]=
	\begin{pmatrix}
		w&-u&-2v\\
		0&-2w&-u\\
		0&0&w
	\end{pmatrix}.
	\]
	Comparing coefficients yields
	\[
	w=0,\qquad u=-2a,\qquad v=-\frac12.
	\]
	Therefore
	\[
	E=\begin{pmatrix}
		-2a&-\frac12&0\\
		0&0&\frac12\\
		0&0&2a
	\end{pmatrix},
	\qquad a^2=-k.
	\]
	This is the final normal form and completes the proof.
\end{proof}

\begin{pr}\label{prop:Lorentzian-derived-low-dimensional}
	Let $(\G,\br,\prs)$ be a 3-dimensional Lorentzian Lie algebra  of nonzero constant sectional curvature $k$. Assume that
	$[\G,\G]\neq\G$ is nondegenerate and Lorentzian. Then $\G$ is
	isomorphic to a Lie algebra of the type described in
	Proposition~\ref{euclideanbis}, with a spacelike vector $u$.
	More precisely, there exist a unit spacelike vector $e$, a
	skew-adjoint endomorphism $B\in\mathfrak{so}(\G)$ satisfying
	$Be=0$, and a nonzero scalar $\lambda$ such that
	\[
	[x,y]
	=
	\langle x,e\rangle(B+\lambda\Id_{\G})y
	-\langle y,e\rangle(B+\lambda\Id_{\G})x.
	\]
	In particular,
	$
	k=-\lambda^2<0.
	$
\end{pr}

	\begin{proof}
		Choose a unit spacelike vector
		$e\in[\G,\G]^\perp$ and put $\h=e^\perp$.
		Then $\h$ is a Lorentzian ideal and
		\[
		\G=\R e\overset{\perp}{\oplus}\h.
		\]
		The Levi-Civita product of $\G$ is given by
		\begin{equation}\label{eq:LC-lor-3}
				e\cdot a=Ea,\;
				a\cdot e=Sa,\;
				a\cdot b=a\star b-\langle Sa,b\rangle e,
		\end{equation}
		where $E\in\so(\h)$, $S$ is self-adjoint,
		$D=E-S=\ad_e|_\h$ is a derivation of $\h$,
		and $\star$ is the Levi-Civita product
		of the induced Lorentzian metric Lie algebra
		$(\h,[\, ,\,]_\star,\langle\, ,\,\rangle)$.
		
		As in the proof of Theorem~\ref{D-Euclidean},
		the constant-curvature condition is equivalent to
		\begin{equation}\label{eq:curvature-lor-3}
			\begin{cases}
				K^\star(x,y)
				=Sx\wedge Sy+k\,x\wedge y,\\
				S([x,y]_\star)+y\star Sx-x\star Sy=0,\\
				\Lu_{Sx-Ex}=[\Lu_x,E],\\
				[S,E]=S^2+k\Id_\h,
			\end{cases}
			\qquad x,y\in\h.
		\end{equation}
		
		We shall prove that $S$ is a scalar multiple
		of the identity.
		
		Since $\dim\h=2$, we shall repeatedly use
		the following elementary observation:
		for every derivation $D$ of $\h$ and every
		basis $(e_1,e_2)$,
		\begin{equation}\label{eq:trace-derivation-2}
			D([e_1,e_2]_\star)
			=\operatorname{tr}(D)[e_1,e_2]_\star.
		\end{equation}
		Indeed, this follows immediately from
		the derivation identity.

		Suppose that $S$ is not diagonalizable. By Lemma~\ref{lem:lorentzian-normal-forms},
		there are two possibilities.
		
		\smallskip
		\noindent
		\emph{First possibility.}
		There exists a basis $\B=(e_1,e_2)$ of
		$\h$ such that
		\[
		\Mat(\langle\, ,\,\rangle,\B)
		=\begin{pmatrix}1&0\\0&-1\end{pmatrix},
		\quad
		\Mat(S,\B)
		=\begin{pmatrix}0&a\\-a&0\end{pmatrix},
		\]
		with $a\ne0$, $a^2=k>0$, and $E=0$.
		
		Hence $D=-S$ is invertible and
		$\operatorname{tr}D=0$.
		By \eqref{eq:trace-derivation-2},
		$
		D([e_1,e_2]_\star)=0.
		$
		Since $D$ is invertible, $\h$ is abelian,
		and consequently $K^\star=0$.
		
		On the other hand, the first equation in
		\eqref{eq:curvature-lor-3} yields
		\[
		\begin{aligned}
			0
			&=Se_1\wedge Se_2+k\,e_1\wedge e_2
			=(\det S+k)e_1\wedge e_2
			=2k\,e_1\wedge e_2,
		\end{aligned}
		\]
		a contradiction.
		
		\smallskip
		\noindent
		\emph{Second possibility.}
		There exists a basis $\B=(e_1,e_2)$
		such that
		\[
		\Mat(\langle\, ,\,\rangle,\B)
		=\begin{pmatrix}0&1\\1&0\end{pmatrix},
		\quad
		\Mat(S,\B)
		=\begin{pmatrix}a&1\\0&a\end{pmatrix},
		\]
		and
		\[
		\Mat(E,\B)
		=\begin{pmatrix}-a&0\\0&a\end{pmatrix},
		\qquad a^2=-k>0.
		\]
		Therefore
		\[
		\Mat(D,\B)
		=\begin{pmatrix}-2a&-1\\0&0\end{pmatrix}.
		\]
		
		Since $\operatorname{tr}D=-2a$,
		equation~\eqref{eq:trace-derivation-2}
		implies
		$
		[e_1,e_2]_\star\in\R e_1.
		$
		Furthermore,
		$
		\operatorname{Im}D=\R e_1.
		$
		Since
		\[
		[\G,\G]
		=\operatorname{Im}D+[\h,\h]_\star,
		\]
		we obtain
		$
		[\G,\G]=\R e_1.
		$
		But $e_1$ is isotropic, contradicting
		the assumption that $[\G,\G]$ is
		nondegenerate.
		
		We conclude that $S$ is diagonalizable
		over $\R$.

		Since $S$ is self-adjoint, there exists
		a Lorentzian orthonormal basis
		$\B=(e_1,e_2)$ such that
		\[
		\Mat(\langle\, ,\,\rangle,\B)
		=\begin{pmatrix}1&0\\0&-1\end{pmatrix},
		\qquad
		\Mat(S,\B)
		=\begin{pmatrix}a&0\\0&b\end{pmatrix}.
		\]
		
		As $E$ is skew-adjoint, its matrix is
		\[
		\Mat(E,\B)
		=\begin{pmatrix}0&\alpha\\
			\alpha&0\end{pmatrix}.
		\]
		
		The last equation in
		\eqref{eq:curvature-lor-3} reads
		\[
		\begin{pmatrix}
			0&\alpha(a-b)\\
			-\alpha(a-b)&0
		\end{pmatrix}
		=
		\begin{pmatrix}
			a^2+k&0\\
			0&b^2+k
		\end{pmatrix}.
		\]
		Consequently,
		\begin{equation}\label{eq:diagonal-constraints}
			a^2=b^2=-k,
			\qquad
			\alpha(a-b)=0.
		\end{equation}
		
		In particular, $k<0$ and $a,b\ne0$.
		
		\smallskip
		\noindent
		\emph{Subcase 2.1: $b=-a$.}
		
		Since $a\ne0$, equation
		\eqref{eq:diagonal-constraints}
		implies $\alpha=0$. Thus $E=0$ and
		\[
		D=-S=
		\begin{pmatrix}-a&0\\0&a\end{pmatrix}.
		\]
		This derivation is invertible and
		$\operatorname{tr}D=0$.
			By \eqref{eq:trace-derivation-2},
		$\h$ is abelian, so $K^\star=0$.
		However,
		\[
			0
			=Se_1\wedge Se_2+k\,e_1\wedge e_2
			=(-a^2+k)e_1\wedge e_2
			=2k\,e_1\wedge e_2,
			\]
		which is impossible.
		
		\smallskip
		\noindent
		\emph{Subcase 2.2: $b=a$.}
		
		We have
		\[
		S=a\Id_\h,\; \Mat(D,\B)
		=\begin{pmatrix}
			-a&\alpha\\
			\alpha&-a
		\end{pmatrix}\esp k=-a^2,
		\]
		The eigenvalues of $D$ are
		$(
		-a+\alpha, -a-\alpha)
		$
		with corresponding eigenvectors
		$(e_1+e_2,,e_1-e_2)$.
		
		Moreover, by
		\eqref{eq:trace-derivation-2},
		\begin{equation}\label{eq:bracket-eigenvalue}
			D([e_1,e_2]_\star)
			=-2a[e_1,e_2]_\star.
		\end{equation}
		
		Suppose first that $\alpha^2\ne a^2$.
		Then $D$ is invertible and $-2a$
		is not an eigenvalue of $D$.
		Equation~\eqref{eq:bracket-eigenvalue}
		therefore implies
		$
		[e_1,e_2]_\star=0.
		$
		Hence $\h$ is abelian.
		
		Suppose now that $\alpha^2=a^2$.
		Then $D$ has rank one and its
		eigenvalues are $0$ and $-2a$.
		
		If $\alpha=a$, then
		$
		\operatorname{Im}D=\R(e_1-e_2).
		$
		Equation~\eqref{eq:bracket-eigenvalue}
		gives
		$
		[\h,\h]_\star\subset\R(e_1-e_2),
		$
		and consequently
		$
		[\G,\G]=\R(e_1-e_2).
		$
		
		If $\alpha=-a$, the same argument gives
		$
		[\G,\G]=\R(e_1+e_2).
		$
		
		In both cases, the derived ideal is a
		totally isotropic line.
		This contradicts its nondegeneracy.
		
		We conclude that
		\[
		S=a\Id_\h,\qquad
		k=-a^2<0,
		\qquad
		[\h,\h]=0.
		\]
		
		Finally, extend $E$ to a skew-adjoint
		endomorphism $B$ of $\G$ by setting
		$Be=0$, and put $\lambda=-a$.
		
		Since $\h$ is abelian and
		\[
		[e,x]=(E-S)x=(B+\lambda\Id_\h)x
		\qquad(x\in\h),
		\]
		the bracket on $\G$ is given by
		\[
		[x,y]
		=
		\langle x,e\rangle
		(B+\lambda\Id_\G)y
		-
		\langle y,e\rangle
		(B+\lambda\Id_\G)x
		\]
		for all $x,y\in\G$.
		
		Furthermore,
		$
		k=-a^2=-\lambda^2<0.
		$
		This completes the proof.
	\end{proof}

	\begin{pr}\label{prop:four-dim}
		Let $(\g,[\,,\,],\langle\,,\,\rangle)$ be a four-dimensional Lorentzian Lie algebra of nonzero constant sectional curvature $k$. Suppose that $[\g,\g]\neq\g$ is nondegenerate and Lorentzian. Then $k<0$, and exactly one of the following alternatives occurs.
		\begin{enumerate}
			\item The derived ideal is abelian. There exist a unit spacelike vector $e$, an endomorphism $B\in\so(\g)$ with $Be=0$, and a nonzero real number $\lambda$ such that
			\[
			[x,y]=\langle x,e\rangle(B+\lambda\Id)y-\langle y,e\rangle(B+\lambda\Id)x,
			\qquad k=-\lambda^2.
			\]
			\item The derived ideal is isomorphic to the three-dimensional Heisenberg Lie algebra. There is a basis $\mathcal B_0=(X_1,X_2,X_3,X_4)$ with nonzero brackets
			\begin{equation}\label{eq:A48}
				[X_2,X_3]=X_1,\qquad [X_1,X_4]=\tfrac32X_1,\qquad
				[X_2,X_4]=X_2,\qquad [X_3,X_4]=\tfrac12X_3,
			\end{equation}
			and metric matrix
			\begin{equation}\label{eq:metricA48}
				\Mat(\langle\,,\,\rangle,\mathcal B_0)=
				\begin{pmatrix}
					0&0&1&0\\0&1&0&0\\1&0&0&0\\0&0&0&s^{-2}
				\end{pmatrix},\qquad s\neq0,\quad k=-s^2.
			\end{equation}
			The brackets in \eqref{eq:A48} identify $\g$ with the solvable Lie algebra $A_{4,8}^{1/2}$ in the usual Mubarakzyanov--Patera--Winternitz notation.
		\end{enumerate}
	\end{pr}
	
	\begin{proof} Choose a unit spacelike vector
		$e\in[\G,\G]^\perp$ and put $\h=e^\perp$. Then $\h$ is a
		Lorentzian ideal and $
		\G=\mathbb Re\overset{\perp}{\oplus}\h.$
		The Levi-Civita product takes the form
		\begin{equation}\label{eqbis}
			e\cdot a=Ea,\qquad
			a\cdot e=Sa,\qquad
			a\cdot b=a\star b-\langle Sa,b\rangle e,
		\end{equation}
		where $E\in\so(\h)$, $S$ is self-adjoint,
		$D=E-S=\ad_e|_\h$ is a derivation of $\h$, and $\star$ is the
		Levi-Civita product of the induced Lorentzian metric Lie algebra
		$(\h,\br_\star,\prs)$ where $[a,b]_\star=a\star b-b\star a$.
		
		A similar computation as in the proof of Theorem \ref{D-Euclidean} shows that the constant-curvature condition $\Ki(u,v)=ku\wedge v$ for any $u,v\in\G$ is equivalent to
		\begin{equation}\label{eqg4} \begin{cases}\Ki^*(a,b)= Sa\wedge Sb+k (a\wedge b)
				\\
				S([a,b]_\star)+b\star Sa-a\star Sb=0,\\
				\Lu_{Sa-Ea}=[\Lu_a,E],\\
				[S,E]=S^2+ k\mathrm{Id}_\h,
		\end{cases} \end{equation}for any $a,b\in\h$ and $\Lu$ is the left multiplication operator of $\star$ and $\Ki^*$ is the curvature of $(\h,\br_\star,\prs)$.
		
		Let us show, by using Lemma \ref{lem:lorentzian-normal-forms}, that $S$ is diagonalizable and conclude in this case.

	 If $S$ is not diagonalizable then, by virtue of Lemma \ref{lem:lorentzian-normal-forms},  necessarily $k<0$, and there exists a basis $\B=(e_1,e_2,e_3)$ of $\h$ such that either 
	\begin{equation*}
		\Mat(\prs,\B)=\begin{pmatrix}1&0&0\\0&0&1\\0&1&0\end{pmatrix},\quad
		\Mat(S,\B)=\begin{pmatrix}a&0&0\\0&b&1\\0&0&b\end{pmatrix},\quad
		\Mat(E,\B)=\begin{pmatrix}0&0&0\\0&-b&0\\0&0&b\end{pmatrix},
		\qquad a^2=b^2=-k,
	\end{equation*}
	or
	\begin{equation*}
		\Mat(\prs,\B)=\begin{pmatrix}0&0&1\\0&1&0\\1&0&0\end{pmatrix},\quad
		\Mat(S,\B)=\begin{pmatrix}a&1&0\\0&a&1\\0&0&a\end{pmatrix},\quad
		\Mat(E,\B)=\begin{pmatrix}-2a&-\frac12&0\\0&0&\frac12\\0&0&2a\end{pmatrix},
		\qquad a^2=-k.
	\end{equation*}
	
	In the first case, 
	$
	\Mat(D,\B)=\begin{pmatrix}-a&0&0\\0&-2b&-1\\0&0&0\end{pmatrix}
	$ and hence
	\[ D([e_1,e_2]_\star)=-(a+2b)[e_1,e_2]_\star,\; D([e_1,e_3]_\star)=-a[e_1,e_3]_\star-[e_1,e_2]_\star\esp D([e_2,e_3]_\star)=-2b[e_2,e_3]_\star
	\]The eigenvalues of $D$ are $-a,-2b,0$, and none equals $-(a+2b)$, since $a=\pm b\neq0$. It follows that $[e_1,e_2]=0$. Moreover, the $-a$-eigenspace is $\R e_1$ and the $-2b$-eigenspace is $\R e_2$. Hence $[e_1,e_3]\in\R e_1$ and $[e_2,e_3]\in\R e_2$. In particular,
	\[
	[\g,\g]=\operatorname{Im}D+[\h,\h]=\Span\{e_1,e_2\},
	\]
	which is degenerate.  This contradicts the hypothesis on $[\G,\G]$.

	In the second case, $
	D=\begin{pmatrix}-3a&-\frac32&0\\0&-a&-\frac12\\0&0&a\end{pmatrix}
	$
	is diagonalizable over $\R$ with distinct eigenvalues $-3a,-a,a$. No sum of two distinct eigenvalues lies in this spectrum; indeed the possible sums are $-4a,-2a,0$. Because $D$ is a derivation, $\h$ is abelian and $K^\star=0$. But $Se_1=ae_1$, $Se_3=e_2+ae_3$, and $k=-a^2$, so the first equation of \eqref{eqg4} yields
	\[
	0=\Ki^*(e_1,e_3)=a\,e_1\wedge e_2,
	\]
	a contradiction. Thus $S$ is diagonalizable.

	Since $S^2=-k\Id_\h$, $k=-s^2<0$ and the eigenvalues of $S$ belong to $\{s,-s\}$. The $S$-eigenspaces are orthogonal and nondegenerate, and $[S,E]=0$; hence $E$ preserves them. Up to permuting the spacelike basis vectors and simultaneously replacing $(S,E,e)$ by $(-S,-E,-e)$, there are just two non-scalar possibilities:
	\begin{equation}\label{eq:two-S}
		\Mat(S,\B)=\operatorname{diag}(s,s,-s),\;
		 \Mat(S,\B)=\operatorname{diag}(s,-s,s)\quad\text{or}\quad S=s\mathrm{Id}_\h
	\end{equation}
	relative to a Lorentzian orthonormal basis $\B=(e_1,e_2,e_3)$, where $\langle e_3,e_3\rangle=-1$.

	\subsection*{Case I: $\Mat(S,\B)=\operatorname{diag}(s,s,-s)$}
	Since $E$ commutes with $S$ and is skew-adjoint, it is a rotation on the Euclidean plane $\Span\{e_1,e_2\}$ and vanishes on $\R e_3$:
	\[
	\Mat(E,\B)=\begin{pmatrix}0&b&0\\-b&0&0\\0&0&0\end{pmatrix},
	\qquad
	\Mat(D,\B)=\begin{pmatrix}-s&b&0\\-b&-s&0\\0&0&s\end{pmatrix}.
	\]
	The complex eigenvalues of $D$ are $-s+ib$, $-s-ib$, and $s$. We have
	\[ D([e_1,e_2]_\star)=-2s[e_1,e_2]_\star,\; D([e_1,e_3]_\star)=-b[e_2,e_3]_\star  \esp D([e_2,e_3]_\star)=b[e_1,e_3]_\star.\]
	We deduce that $[e_1,e_2]_\star=0$ and
	\[  D^2([e_1,e_3]_\star)=-b^2[e_1,e_3]_\star  \esp D^2([e_2,e_3]_\star)=-b^2[e_2,e_3]_\star. \]
	But the eigenvalues of $D^2$ are $s^2-b^2-2sb i$, $s^2-b^2+2sb i$ and $s^2$. Since $s\not=0$ we conclude that
	$\h$ is abelian, so $\Ki^*=0$.
	
	But $Se_1=se_1$ and $Se_3=-se_3$ yield, from the first identity of \eqref{eqg4},
	\[
	0=\Ki^*(e_1,e_3)
	=(-s^2+k)e_1\wedge e_3
	=-2s^2e_1\wedge e_3,
	\]
	a contradiction.

	\subsection*{Case II: $S=\operatorname{diag}(s,-s,s)$}
	The $s$-eigenspace is now the Lorentzian plane $\Span\{e_1,e_3\}$. Consequently $E$ is a boost on this plane and vanishes on $\R e_2$:
	\[
	\Mat(E,\B)=\begin{pmatrix}0&0&b\\0&0&0\\b&0&0\end{pmatrix},\qquad
	\Mat(D,\B)=\begin{pmatrix}-s&0&b\\0&s&0\\b&0&-s\end{pmatrix}.
	\]
	The eigenvalues of $D$ are $\{b-s,-b-s,s\}$.
	We have
	\[ D([e_1,e_2]_\star)=-b[e_2,e_3]_\star,\; D([e_1,e_3]_\star)=-2s[e_1,e_3]_\star\esp D([e_2,e_3]_\star)=-b[e_1,e_2]_\star. \]
	We deduce that
	\[ D^2([e_1,e_2]_\star)=b^2[e_1,e_2]_\star\esp D^2([e_2,e_3]_\star)=b^2[e_2,e_3]_\star. \]
	Note that $b^2$ is an eigenvalue of $D^2$ if and only if $b=\pm\frac{s}2$.
	
	If $b\notin\{s,-s,\frac{s}2,-\frac{s}2\}$ then $\h$ is abelian and $\Ki^*=0$.  But
	\[ Se_1\wedge Se_2+ke_1\wedge e_2=(-s^2+k)e_1\wedge e_2+bs e_3\wedge e_2\not=0. \]
	
	Suppose $b=s$. We have $\ker D=\R(e_1+e_3)$, $\ker(D+2s\mathrm{Id}_\h)=\R(e_1-e_3)$ and $\ker(D-s\mathrm{Id}_\h)=\R e_2$. So there exists $\al,\be,\ga$ such that
	\[ [e_1,e_2]_\star=\al e_2,\;[e_1,e_3]_\star=\be(e_1-e_3)\esp [e_2,e_3]_\star=\ga e_2. \]
	So the nonvanishing Lie brackets for $\G=\Span(e,e_1,e_2,e_3)$ are
	\[ \begin{cases}
		[e,e_1]=-s(e_1-e_3),\; [e,e_2]=s e_2,\; [e,e_3]=s(e_1-e_3),\\
		[e_1,e_2]=\al e_2,\;[e_1,e_3]=\be(e_1-e_3),\; [e_2,e_3]=\ga e_2.
	\end{cases} \]Thus $[\G,\G]=\Span(e_2,e_1-e_3)$ which is degenerate and this contradicts our hypothesis on $[\G,\G]$.
	
	The case $b=-s$ gives the same contradiction.
	
	Suppose that $b=\frac{s}2$. The eigenvalues of $D$ are $(-\frac{s}2,-\frac32s,s)$ and hence $[e_1,e_3]_\star=0$ and $[e_1,e_2]_\star,[e_2,e_3]_\star\in\ker(D^2-\frac{s^2}4)=\R (e_1+e_3)$. So $[e_1,e_2]_\star=\al(e_1+e_3)$ and $[e_2,e_3]_\star=\be(e_1+e_3)$. The relation
	$D([e_1,e_2]_\star)=-\frac{s}2[e_2,e_3]$ implies that $\al=\be$. Thus $e_1+e_3$ is central and $\h$ is isomorphic to the 3-dimensional Heisenberg Lie algebra. Moreover, $e_1+e_3$ is isotropic so $\Ki^*=0$ (see \cite[Proposition 4.1]{Chakkar}). However,
	\[ 0=Se_1\wedge Se_2+ke_1\wedge e_2=(-s^2+k)e_1\wedge e_2+be_2\wedge e_2 \]
	which is impossible.
	
	The case $b=-\frac{s}2$ gives in a similar way a contradiction.
	
		\subsection*{Case III: $S=s\mathrm{Id}_\h$}
	
	Suppose now that $S=s\mathrm{Id}_\h$ and $s^2=-k$. Then $D=E-s\mathrm{Id}_\h$ and \eqref{eqg4} becomes
	\begin{equation}\label{ed} \Ki^*(a,b)= 0\esp 
				\Lu_{sa-Ea}=[\Lu_a,E] \end{equation}for any $a,b\in\h$ and where $E\in\so(\h).$
				
				Let us show that $D$ must be invertible. Suppose the contrary. Then $s$ is a real eigenvalue of $E$ and hence there exists a basis $\B=(e_1,e_2,e_3)$ such that
				\[ \Mat(\prs,\B)=\begin{pmatrix}0&1&0\\1&0&0\\0&0&1
				\end{pmatrix}, \Mat(E,\B)=\begin{pmatrix}s&0&0\\0&-s&0\\0&0&0
				\end{pmatrix}\esp \Mat(D,\B)=\begin{pmatrix}0&0&0\\0&-2s&0\\0&0&-s
				\end{pmatrix}.  \]
				We get
				\[ D([e_1,e_2]_\star)=-2s[e_1,e_2]_\star,\;D([e_1,e_3]_\star)=-s[e_1,e_3]_\star\esp D([e_2,e_3]_\star)=-3s[e_2,e_3]_\star. \]
				We get that $[\G,\G]=\Span(e_2,e_3)$ which is degenerate and which contradicts our assumption.
				
				Thus $D$ is invertible and, by virtue of Jacobson's Theorem,  $\h$ is nilpotent. If $\h$ is abelian then $\G$ isomorphic to a Lie algebra of the type described in
				Proposition~\ref{euclideanbis}, with a spacelike vector $u$.
				
				Suppose now that $\h$ is isomorphic to the Heisenberg Lie algebra. Since the curvature vanishes then the center is isotropic (see \cite[Proposition 4.1]{Chakkar}). So there exists a basis $\B=(e_1,e_2,e_3)$ such that
				\[ \Mat(\prs,\B)=\begin{pmatrix}0&1&0\\1&0&0\\0&0&1
				\end{pmatrix}\esp [e_2,e_3]=\la e_1,\;\la\not=0.\]
				Since $D$ is a derivation then $D$ preserves $e_1$ and hence $E(e_1)=a e_1$. Since $E$ is skew-symmetric we get
				\[ \Mat(E,\B)=\begin{pmatrix}
					a&0&b\\
					0&-a &0\\
					0&-b&0
				\end{pmatrix}\esp 
				\Mat(D,\B)=\begin{pmatrix}
				 	a-s&0&b\\
				 	0&-a-s &0\\
				 	0&-b&-s
				 \end{pmatrix}. \]
				 Since $D$ is a derivation, one must have $a-s=-a-s-s$ and hence $a=-\frac{s}2$. Thus
				 \[ \Mat(D,\B)=\begin{pmatrix}
				 	-\frac32 s&0&b\\
				 	0&-\frac12 s &0\\
				 	0&-b&-s
				 \end{pmatrix}. \] 
	A computation shows that the second equation in \eqref{ed} is satisfied. Finally, the nonvanishing Lie brackets of $\G$ are
	\[ [e,e_1]=-\frac32 s e_1,\; [e,e_2]=-\frac12 s e_2-b e_3,\;[e,e_3]=be_1-se_3\esp [e_2,e_3]=\la e_1,\; s^2=-k,\;\la\not=0. \]
	
Let us show that	 $\mathfrak g$ is isomorphic to the
	four-dimensional solvable Lie algebra $A_{4,9}^{1/2}$. More precisely, there exists a basis $\B_0=(E_1,E_2,E_3,E_4)$ satisfying
	\begin{equation}\label{A48bis}
		[E_2,E_3]=E_1,\quad
		[E_1,E_4]=\frac32 E_1,\quad
		[E_2,E_4]=E_2,\quad
		[E_3,E_4]=\frac12 E_3.
	\end{equation}
	and
	\begin{equation}\label{metricA48bis}
					\Mat(\prs,\B_0)
			=\begin{pmatrix}
				0&0&1&0\\
				0&1&0&0\\
				1&0&0&0\\
				0&0&0&s^{-2}
			\end{pmatrix}.
	\end{equation}

	\textbf{Step 1: remove the parameter $b$.}
	  Define
	\begin{equation*}\label{shear}
		\begin{aligned}
				f_1&=e_1,\;
				f_2=e_2-\frac{2b}{s}e_3-\frac{2b^2}{s^2}e_1,\;
				f_3=e_3+\frac{2b}{s}e_1.
		\end{aligned}
	\end{equation*}
	Direct substitution  gives
	\begin{equation*}\label{bracketsf}
		[e,f_1]=-\frac{3s}{2}f_1,\qquad
		[e,f_2]=-\frac{s}{2}f_2,\qquad
		[e,f_3]=-sf_3,\qquad
		[f_2,f_3]=\lambda f_1.
	\end{equation*}
	
	The change of basis \eqref{shear} is an isometry of the metric on
	$\mathfrak h$. 
	
	\textbf{Step 2: normalize the Heisenberg bracket.}
	Let
	\[
	L=|\lambda|>0,\qquad \varepsilon=\operatorname{sgn}(\lambda)\in\{-1,1\}.
	\]
	Set
	\begin{equation*}\label{changeX}
		\begin{aligned}
				E_1&=\sqrt L\,f_1,\;
				E_2=-\varepsilon f_3,\;
				E_3=\frac1{\sqrt L}\,f_2,\;
				E_4=\frac1s\,e.
		\end{aligned}
	\end{equation*}
	  Therefore \eqref{A48bis} and \eqref{metricA48bis}
	hold, and this is the conventional real solvable Lie algebra $A_{4,9}^{1/2}$.
	\end{proof}

\section{Lorentzian Lie groups of nonzero constant curvature whose derived ideal is degenerate}\label{section6}

We begin by establishing a useful decomposition of the Levi-Civita product for Lorentzian Lie algebras of nonzero constant sectional curvature whose derived ideal is degenerate.

\begin{pr}\label{PD}
Let $(\G,\br,\prs)$ be a Lorentzian Lie algebra of nonzero constant
sectional curvature $k$ such that $[\G,\G]\neq\G$. Assume that the
derived ideal $[\G,\G]$ is degenerate. Then:
\begin{enumerate}
	\item $\G$ admits a Witt decomposition
	\[
	\G=\R e\oplus\h\oplus\R f,
	\]
	where $\h=(\R e\oplus\R f)^\perp$ is Euclidean,
	$\langle e,e\rangle=\langle f,f\rangle=0$, and
	$\langle e,f\rangle=1$.
	
	\item There exist linear maps $A,E,F:\h\longrightarrow\h$, with
	$E,F\in\so(\h)$, vectors $u,v,w\in\h$, $\la\in\R$ and a bilinear product
	$\star:\h\times\h\longrightarrow\h$ such that the associated left multiplication operator $\Lu_a\in\so(\h)$ and the Levi-Civita
	product of $\G$ is given, for every $a,b\in\h$, by
	\begin{equation}\label{Levigg}
		\begin{cases}
			e\cdot a=Ea+\langle v,a\rangle e,\;
			a\cdot e=\langle u,a\rangle e,\\
			a\cdot b=a\star b+\langle Aa,b\rangle e
			,\\
			e.e=0,\;e\cdot f=-v,\;
			f\cdot e=u+\lambda e,\;
			f\cdot f=-w-\lambda f,\\
			f\cdot a=Fa+\langle a,w\rangle e-\langle a,u\rangle f,\;
			a\cdot f=-Aa-\langle u,a\rangle f,
		\end{cases}
	\end{equation}and
	\begin{equation}\label{jacobi}
		[[a,b]_\star,c]_\star+[[b,c]_\star,a]_\star+[[c,a]_\star,b]_\star=-\sigma(a,b)E(c)-\sigma(b,c)E(a)-\sigma(c,a)E(b),
	\end{equation}where $\sigma(a,b)=\langle (A-A^*)(a),b\rangle$.
	\item The constant-sectional-curvature condition
	$
	\mathcal K(x,y)=k\,x\wedge y,
	\; x,y\in\mathfrak g,
	$
	is equivalent to the following system of identities, valid for all \(a,b,c\in\mathfrak h\).
	\begin{equation}\label{curvature}
		\begin{cases}
			\Ki^*(a,b)c
			=
			\bigl(
			\langle Ab,a\rangle-\langle Aa,b\rangle
			\bigr)Ec
			+k(a\wedge b)c,
			\\
			A([a,b]_\star)+b\star Aa-a\star Ab
			+
			\bigl(
			\langle Aa,b\rangle-\langle Ab,a\rangle
			\bigr)v
			+\langle b,u\rangle Aa-\langle a,u\rangle Ab=0,
			\\
			[E,\Lu_a]
			=
			\Lu_{E(a)}+\langle a,v-u\rangle E,
			\\
			[E,A]
			=
			k\Id_{\h}+\langle\,\cdot\,,v\rangle v+\Ru_v,
			\\
			[F,\Lu_a]
			=
			\Lu_{Fa+Aa}+\langle a,w\rangle E+u\wedge Aa,
			\\
			[F,A]a
			=
			A^2a-\lambda Aa+\langle a,u\rangle w
			+\langle a,w\rangle v+a\star w,
			\\
			[F,E]
			=
			\lambda E+\Lu_{u+v}+u\wedge v,
			\\
			a\star u-\langle a,u\rangle u
			=ka,
			\\
			Av-Fv+Ew+Au=0,
			\;
			Eu=0,
			\;
			2A^*u-Fu=0,
			\;
			2\langle u,v\rangle+|u|^2=k,
			\;
			\langle u,[a,b]_\star\rangle=0,
		\end{cases}
	\end{equation}
	where \(\Lu_a\) and \(\Ru_a\) denote the left and right multiplication operators associated with \(\star\), respectively, and \(\Ki^*\) denotes its curvature operator.
	
\end{enumerate}

\end{pr}
\begin{proof}
	We have
$[\G,\G]\cap[\G,\G]^\perp=\R e$
for some nonzero isotropic vector $e$. Choose a
complementary subspace $\h$ of $\R e$ in $e^\perp$, so that
\[
e^\perp=\R e\oplus\h.
\]
Since $\G$ is Lorentzian, $\h$ is Euclidean. Moreover, $\h^\perp$ is
a two-dimensional Lorentzian subspace containing $e$. We may therefore
choose a null vector $f\in\h^\perp$ such that
$\langle e,f\rangle=1$. Hence
$
\G=\R e\oplus\h\oplus\R f.
$
and the restriction of the metric to $\h$ is positive definite.

Since $e\in[\G,\G]^\perp$, the right multiplication operator $\Ru_e$
is self-adjoint. Moreover, $e\cdot e=0$. It follows that both $\Lu_e$
and $\Ru_e$ preserve
$
e^\perp=\R e\oplus\h.
$
Consequently, there exist endomorphisms
\[
E,S,A,F\colon\h\longrightarrow\h,
\]
vectors $u,v,w\in\h$, and a scalar $\lambda\in\R$ such that, for all
$a,b\in\h$, the Levi--Civita product is given by
\begin{equation*}\label{Leviggbis}
	\left\{
	\begin{aligned}
		e\cdot a&=Ea+\langle v,a\rangle e,\;
		a\cdot e=Sa+\langle u,a\rangle e,\\
		a\cdot b&=a\star b+\langle Aa,b\rangle e
		-\langle Sa,b\rangle f,\\
		e\cdot f&=-v,\;
		f\cdot e=u+\lambda e,\;
		f\cdot f=-w-\lambda f,\\
		f\cdot a&=Fa+\langle a,w\rangle e-\langle a,u\rangle f,\;
		a\cdot f=-Aa-\langle u,a\rangle f.
	\end{aligned}
	\right.
\end{equation*}
Here $\star$ denotes the Levi--Civita product associated with the
induced Euclidean metric on $\h$. Furthermore, $E$ and $F$ are
skew-symmetric, whereas $S$ is symmetric.

To get \eqref{Levigg}, let us show that $S=0$. 

For any $a,b\in\h$, we have
	\begin{align*}
		\ass(a,e,b)&=(a.e).b-a.(e.b)\\
		&=(Sa+\langle u,a\rangle e).b-a.(Eb+\langle v,b\rangle e)\\
		&=Sa\star b+\langle ASa,b\rangle e-\langle S^2a,b\rangle f
		+\langle u,a\rangle(Eb+\langle v,b\rangle e)
		-a\star Eb-\langle Aa,Eb\rangle e+\langle Sa,Eb\rangle f
		-\langle v,b\rangle Sa\\&-\langle v,b\rangle\langle u,a\rangle e\\
		&=Sa\star b+\langle u,a\rangle Eb-a\star Eb-\langle v,b\rangle Sa
		+\left(\langle ASa,b\rangle+\langle u,a\rangle\langle v,b\rangle
		-\langle Aa,Eb\rangle-\langle v,b\rangle\langle u,a\rangle\right)e\\
		&-\langle (S^2+ES)a,b\rangle f.
	\end{align*}
	\begin{align*}
		\ass(e,a,b)&=(e.a).b-e.(a.b)\\
		&=(Ea+\langle v,a\rangle e).b-e.(a\star b+\langle Aa,b\rangle e-\langle Sa,b\rangle f)\\
		&=Ea\star b+\langle AEa,b\rangle e-\langle SEa,b\rangle f
		+\langle v,a\rangle(Eb+\langle v,b\rangle e)
		-E(a\star b)-\langle v,a\star b\rangle e-\langle Sa,b\rangle v\\
		&=Ea\star b+\langle v,a\rangle Eb -E(a\star b)
		-\langle Sa,b\rangle v+\left( \langle AEa,b\rangle+\langle v,a\rangle\langle v,b\rangle-\langle v,a\star b\rangle \right)e
		-\langle SEa,b\rangle f.
	\end{align*}The relation 
	\[ \Ki(a,e,b)=\ass(a,e,b)-\ass(e,a,b)=k (a\wedge e)(b) \]
	implies $[S,E]=S^2$ and hence $\tr(S^2)=0$.  Since $S$ is symmetric and $\h$ is Euclidean we deduce that $S=0$. 
	
	On the other hand, for any $a,b\in\h$,
	\[ [a,b]=[a,b]_\star+\langle (A-A^*)a,b\rangle. \]
	The Jacobi identity for $\br$ implies \eqref{jacobi}.

 The computation needed to show that the  constant-curvature condition $\Ki(x,y)=k x\wedge y$ for any $x,y\in\G$ is equivalent to \eqref{curvature} are performed in 
 Section \ref{section9} and Subsection \ref{subsection3}.
\end{proof}

We now solve the system \eqref{curvature}. We proceed step by step, beginning with the following proposition.

\begin{pr} One has $E=0$ and the Euclidean Lie algebra $(\h,\br_\star,\prs)$ has constant curvature $k$ so there exist a nonzero vector $z\in\h$, a
	skew-symmetric endomorphism $B\in\mathfrak{so}(\h)$ satisfying $Bz=0$, and a
	nonzero scalar $\mu$ such that, for any $a\in\h$,
	\[
	\Lu_a=\langle a,z\rangle B+\mu\,a\wedge z\esp k=-\mu^2|z|^2.
	\]
	Moreover,  we have
	\[
	u=\mu z,
	\qquad
	v=-\mu z.
	\]
\end{pr}

\begin{proof}
	Suppose, by contradiction, that $E\neq0$.
	For every $a\in\ker E$, the third identity in \eqref{curvature}
	reduces to
	\[
	[E,\Lu_a]=\langle a,v-u\rangle E,\quad a\in\ker E.
	\]
	By virtue of Lemma \ref{adB},
		$
	\langle a,v-u\rangle=0,$ for any
	$a\in\ker E.$
	Hence
	$v-u\in(\ker E)^\perp=\mathrm{Im} E.$
	Thus there exists $x\in\h$ such that
	$v-u=Ex.$
	
	From the last sequences of identities in \eqref{curvature} $Eu=0$ and $E$ is skew-symmetric, we obtain
	\[
	\langle u,v-u\rangle
	=
	\langle u,Ex\rangle
	=
	-\langle Eu,x\rangle
	=
	0.
	\]
	Consequently,
	$\langle u,v\rangle=|u|^2.$
	From the last sequences of identities in \eqref{curvature}
	$2\langle u,v\rangle+|u|^2=k$
	ans hence $k=3|u|^2.$
	We have also  $u\in[\h,\h]^\perp$. By the Koszul formula,
	$u\star u=0.$
	
	Taking $a=u$ in the eighth relation in \eqref{curvature}
	\[
	a\star u-\langle a,u\rangle u=ka
	\]
	therefore yields $-|u|^2u=ku.$
	Since $k\neq0$, the vector $u$ is nonzero, and hence
	$k=-|u|^2.$
	This contradicts $k=3|u|^2$. We conclude that
	$E=0.$
	
	It now follows from \eqref{jacobi} and the first relation in \eqref{curvature}
	that $(\h,\br_\star,\prs)$ is   a Euclidean Lie algebra of constant sectional curvature
	$k$. By Theorem~\ref{constant}, there exist a nonzero vector $z\in\h$, a
	skew-symmetric endomorphism $B\in\mathfrak{so}(\h)$ satisfying $Bz=0$, and a
	nonzero scalar $\mu$ such that
	\begin{equation}\label{Lu}
		\Lu_a=\langle a,z\rangle B+\mu\,a\wedge z\esp k=-\mu^2|z|^2.
	\end{equation}
	Furthermore, taking $a=u$ in
	\[
	a\star u-\langle a,u\rangle u=ka
	\]
	and using $u\star u=0$, we obtain
	$k=-|u|^2<0$ and 
	$|u|^2=\mu^2|z|^2.$
	
	From the last sequences of identities in \eqref{curvature}
	\[
	2\langle u,v\rangle+|u|^2=k
	\]
	then implies
	$\langle u,v\rangle=-|u|^2.$
	
	Applying the identity
	\[
	a\star u-\langle a,u\rangle u=ka
	\]
	with $a=z$, we obtain
	\[
	|z|^2B(u)
	=
	\langle z,u\rangle u-\mu^2|z|^2z.
	\]
	Taking the scalar product with $z$, and using $Bz=0$, gives
	\[
	\langle z,u\rangle^2
	=
	\mu^2|z|^4
	=
	|u|^2|z|^2.
	\]
	Equality holds in the Cauchy--Schwarz inequality; hence $u$ and $z$
	are collinear. Thus $u=\alpha z$
	for some $\alpha\in\R$. But
	$|u|^2=\mu^2|z|^2$ and then $\alpha^2=\mu^2.$

	The relation $\langle u,v\rangle=-|u|^2$ implies that $v=v_0-\al z$ where $v_0\in z^\perp$.
	
	On the other hand, the fourth identity in \eqref{curvature}
	\[
	\Ru_v(a):=\langle a,z\rangle B(v)+\mu (a\wedge z)(v) =     -ka-\langle\,a,,v\rangle v
	\]
	Taking $a=z$, and having in mind $B(z)=0$, we get
	\[ \langle z,z\rangle B(v_0)=-k z+\al\langle z,z\rangle(v_0-\al z). \]
	Thus $B(v_0)=\al v_0$. Since $B$ is skew-symmetric we get $v_0=0$. 
	
	If we take $a\in z^\perp$ and having in mind $k=-\mu^2|z|^2$ and $v=-\al z$, we get
	\[ \mu\al |z|^2=\mu^2|z|^2 a. \]
	Thus $\al=\mu$. This completes the proof.
\end{proof}

Having in mind the solution found in this proposition, the system \eqref{curvature} reduces to
\begin{equation}\label{curvaturebis}
	\begin{cases}
		A([a,b]_\star)+b\star Aa-a\star Ab
		+
		\bigl(
		\langle Aa,b\rangle-\langle Ab,a\rangle
		\bigr)v
		+\langle b,u\rangle Aa-\langle a,u\rangle Ab=0,
		\\
		[F,\Lu_a]
		=
		\Lu_{Fa+Aa}+u\wedge Aa,
		\\
		[F,A]a
		=
		A^2a-\lambda Aa+\langle a,u\rangle w
		+\langle a,w\rangle v+a\star w,
		\\
		A^*z=Fz=0.
	\end{cases}
\end{equation}

Let us solve now the first equation in \eqref{curvaturebis} with the condition $A^*z=0$. Put $V=z^\perp$ and $B_0=B_{|V}$.

\begin{pr} 
	The condition $A^*z=0$ and 
	the first equation of \eqref{curvaturebis} are equivalent to $Az=p\in V$, $A(V)\subset V$ and 
	\[
	[A_{0},B_0]=0\]where $A_0=A_{|V}$.
\end{pr}

\begin{proof}
	Since $A^*z=0$, we have $\operatorname{Im}A\subset V$.
	In particular, there exist $p\in V$ and $A_0\in\mathrm{End}(V)$ such that
	$Az=p$ and $A_{|V}=A_0.$

	The Lie bracket and the Levi--Civita product on $\h$ are given by
	\begin{align*}
		[a,b]_\star
		&=
		\langle a,z\rangle(B+\mu\Id_{\h})b
		-\langle b,z\rangle(B+\mu\Id_{\h})a,
		\\
		a\star b
		&=
		\langle a,z\rangle Bb
		+\mu\bigl(
		\langle a,b\rangle z-\langle b,z\rangle a
		\bigr).
	\end{align*}
	Since $u=\mu z$ and $v=-\mu z$, the first equation of
	\eqref{curvaturebis} becomes
	\begin{equation}\label{A-initial}
		A([a,b]_\star)+b\star Aa-a\star Ab 
		-\mu\bigl(
		\langle Aa,b\rangle-\langle Ab,a\rangle
		\bigr)z+\mu\langle b,z\rangle Aa
		-\mu\langle a,z\rangle Ab=0.
	\end{equation}
	By virtue of $A^*z=0$, we have
	$\langle Aa,z\rangle=\langle Ab,z\rangle=0.$
	Consequently,
	\begin{align*}
		A([a,b]_\star)
		&=
		\langle a,z\rangle A(B+\mu\Id_{\h})b
		-\langle b,z\rangle A(B+\mu\Id_{\h})a,\\
		b\star Aa
		&=
		\langle b,z\rangle BAa
		+\mu\langle b,Aa\rangle z,\\
		a\star Ab
		&=
		\langle a,z\rangle BAb
		+\mu\langle a,Ab\rangle z.
	\end{align*}
	Substituting these expressions into \eqref{A-initial}, for $a,b\in V$ we get $0=0$.
	
	If we take $a=z$, $b\in V$, we get
	\[ [A_0,B_0]=0. \] This completes the proof.
\end{proof}

Let us complete our study by solving the second,  third,  and fourth equations in \eqref{curvaturebis}.		

\begin{pr}\label{prop:sixth-equation}
	\begin{enumerate}\item The second,  third, and  fourth equations in \eqref{curvaturebis} are equivalent to
		\[ \begin{cases} Fz=0,\; F(V)\subset V,\; [F_0,B_0]=0,\\
			[F_0,A_0]=A_0^2-\la A_0-\mu\rho|z|^2 \mathrm{Id}_{V},\\
			w_0=\frac1{|z|^2}(B_0+\mu\mathrm{Id}_{V})^{-1}(Fp-A_0p+\la p),
		\end{cases} \]where $w=w_0+\rho z$, $p=Az$ and $F_0=F_{|V}$.
	\end{enumerate}	
\end{pr}

\begin{proof}
	
	\begin{enumerate}
		\item	Since $Fz=0$ and $F$ is skew-symmetric, then $F(V)\subset V$.
		\[ [F,\Lu_a]
		=
		\Lu_{Fa+Aa}+u\wedge Aa, \]
		We have, by virtue of \eqref{Lu} and \eqref{wedge} and having in mind $F(z)=A^*z=0$ and $u=\mu z$,
		\begin{align*}
			&[F,\Lu_a]=\langle a,z\rangle [F,B]+\mu Fa\wedge z,\\
			&\Lu_{Fa+Aa}+u\wedge Aa=\langle Fa+Aa,z\rangle B+\mu Fa\wedge z+\mu Aa\wedge z+\mu z\wedge Aa
			=\mu Fa\wedge z.
		\end{align*}So the second equation in \eqref{curvaturebis} is equivalent to
		\[ \langle a,z\rangle [F,B]=0 \]for all $a\in\h$.
		Since $Fz=Bz=0$ this is equivalent to $[F_0,B_0]=0$.
		
		\item  Since $u=\mu z$ and $v=-\mu z$ and 
		$$a\star w=\langle a,z\rangle B(w)+\mu\langle a,w\rangle z-\mu\langle w,z\rangle a, $$	the third equation in \eqref{curvaturebis} is
		\[
		[F,A]a
		=
		A^2a-\lambda Aa
		+\mu\langle a,z\rangle w
		+\langle a,z\rangle B(w)-\mu\langle w,z\rangle a,
		\]for any $a\in\h$. Put $w=w_0+\rho z$ where $w_0\in V$. By taking $a\in V$ and $a=z$ this equation is equivalent to
		\[ \begin{cases}
			[F_0,A_0]=A_0^2-\la A_0-\mu\rho|z|^2 \mathrm{Id}_{V},\\
			|z|^2(B_0(w_0)+\mu w_0)=Fp-A_0p+\la p,\quad p=Az.
		\end{cases} \]

	\end{enumerate}

\end{proof}

So far, we have shown the following result.

\begin{theo}\label{DD} Let $(\G,\br,\prs)$ be a Lorentzian Lie algebra of nonzero constant
	sectional curvature $k$ such that $[\G,\G]\neq\G$. Assume that the
	derived ideal $[\G,\G]$ is degenerate. Then:
	
	\begin{enumerate}
		\item $\G$ admits a Witt decomposition $\G=\R e\oplus\h\oplus\R f$ where $\h=\R z\stackrel{\perp}\oplus V$ is a Euclidean vector space, $\h^\perp=\mathrm{span}\{e,f\}$, $\langle e,e\rangle=\langle f,f\rangle=0$ and $\langle e,f\rangle=1$,
		\item There exists $\mu,\la,\rho\in\R$ with $\mu\not=0$, $F,B\in\so(V)$,  $A\in\mathrm{End}(V)$ and $p\in V$ such that, for any $a,b\in V$, the  nonvanishing Lie brackets are
		\begin{equation}\label{eq:degenerate-final-brackets}
			\begin{cases}
				[e,z]
				&=-2\mu |z|^2e,\;
				[e,f]
				=-\lambda e,\;
				[a,b]
				=
				\bigl(
				\langle Aa,b\rangle-\langle Ab,a\rangle
				\bigr)e,\\
				[z,a]
				&=
				|z|^2(B+\mu\Id_V)a+\langle p,a\rangle e,\;
				[f,a]
				=(F+A)a+\langle a,w\rangle e,\;
				[f,z]
				=p+\rho |z|^2e,
			\end{cases}
		\end{equation}
		and
		\[ \begin{cases}[F,B]=[A,B]=0,\;
			[F,A]=A^2-\la A-\mu\rho|z|^2 \mathrm{Id}_{V},\\
			w=\frac1{|z|^2}(B+\mu\mathrm{Id}_{V})^{-1}(Fp-Ap+\la p).
		\end{cases} \]The sectional curvature is
		$k=-\mu^2|z|^2.$
		
	\end{enumerate} 
	
\end{theo}

The preceding theorem reduces the classification of Lorentzian Lie algebras under consideration to a system of algebraic constraints, the central one being a quadratic commutator equation of the form
\[
[F,A]=A^2-\lambda A-\rho\mu|z|^2\Id_V,
\qquad F\in\mathfrak{so}(V).
\]
We now investigate the more general equation
\[
[F,A]=A^2+aA+b\Id_V
\]
on a Euclidean vector space. As we shall see, the sign of the discriminant \(a^2-4b\) plays a decisive role in the structure of its solutions.
\subsection{The equation $[F,A]=A^2+aA+b\Id_V$}

Let $(V,\langle\,\cdot\,,\,\cdot\,\rangle)$ be a Euclidean vector
space, let $F\in\mathfrak{so}(V)$, and consider the equation
\begin{equation}\label{eq:quadratic-commutator}
	[F,A]=A^2+aA+b\Id_V,
	\qquad A\in\mathrm{End}(V).
\end{equation}
Put
\[
Q(X)=X^2+aX+b,
\qquad
\Delta=a^2-4b.
\]

\begin{pr}\label{prop:quadratic-commutator}
	Assume that $\Delta\geq0$. Then every solution of
	\eqref{eq:quadratic-commutator} satisfies
	\begin{equation}\label{eq:quadratic-commutator-conclusion}
		Q(A)=A^2+aA+b\Id_V=0,
		\qquad
		[F,A]=0.
	\end{equation}
	
	More precisely, put $r_{\pm}=\frac{-a\pm\sqrt{\Delta}}2.$
	Then the following assertions hold.
	\begin{enumerate}
		\item If $\Delta>0$, then $A$ is diagonalizable over
		$\mathbb R$, with $
		\operatorname{Spec}(A)\subset\{r_-,r_+\}.$
		With respect to the orthogonal decomposition
		\[
		V=\ker(A-r_-\Id_V)
		\oplus\ker(A-r_-\Id_V)^\perp,
		\]
		the operators $F$ and $A$ have the block forms
		\[
		F=
		\begin{pmatrix}
			F_-&0\\
			0&F_+
		\end{pmatrix},
		\qquad
		A=
		\begin{pmatrix}
			r_-\Id&U\\
			0&r_+\Id
		\end{pmatrix},
		\]
		where $F_\pm$ are skew-symmetric and
		$F_-U=UF_+.$
		
		\item If $\Delta=0$ and $r=-a/2$, then
		\[
		(A-r\Id_V)^2=0
		\esp
		[F,A]=0.
		\]
		With respect to the orthogonal decomposition
		\[
		V=\ker(A-r\Id_V)
		\oplus\ker(A-r\Id_V)^\perp,
		\]
		one has
		\[
		F=
		\begin{pmatrix}
			F_1&0\\
			0&F_2
		\end{pmatrix},
		\qquad
		A=
		\begin{pmatrix}
			r\Id&U\\
			0&r\Id
		\end{pmatrix},
		\qquad
		F_1U=UF_2.
		\]
		In this case, $A$ need not be diagonalizable.
	\end{enumerate}
\end{pr}

\begin{proof} 
	Suppose first that $\Delta>0$. Put $
	\delta=\sqrt{\Delta}=r_+-r_->0$
	and define
	$C=A-r_-\Id_V.$
	Since scalar endomorphisms commute with $F$, we have
	\[
	[F,C]=[F,A].
	\]
	Moreover,
	\[
	Q(A)
	=(A-r_-\Id_V)(A-r_+\Id_V)
	=C(C-\delta\Id_V).
	\]
	Equation \eqref{eq:quadratic-commutator} therefore becomes
	\[
	[F,C]=C^2-\delta C.
	\]
	Since $\delta\neq0$, by virtue of \cite[Lemma~2.2]{Boucetta},
	\[
	C^2=\delta C
	\esp
	[F,C]=0.
	\]
	Hence
	\[
	Q(A)=C(C-\delta\Id_V)=0,
	\qquad
	[F,A]=0.
	\]
	The block forms of $F$ and $A$ follow immediately from the block
	description in \cite[Lemma~2.2]{Boucetta}, after adding $r_-\Id_V$ to $C$.
	
	Suppose now that $\Delta=0$, and put $r=-a/2$. Then
	$
	Q(X)=(X-r)^2.
	$
	Setting
	$
	C=A-r\Id_V,
	$
	equation \eqref{eq:quadratic-commutator} becomes
	\begin{equation}\label{eq:zero-discriminant}
		[F,C]=C^2.
	\end{equation}
	This is the degenerate case of Lemma~2.2 and requires a separate
	argument.
	
	Since $[F,C]=C^2$ and $C$ commutes with its powers, an induction
	gives
	\[
	[F,C^k]=kC^{k+1},
	\qquad k\geq1.
	\]
	Taking traces yields
	\[
	\tr(C^{k+1})=0,
	\qquad k\geq1.
	\]
	It follows that every eigenvalue of $C$ is zero; hence $C$ is
	nilpotent.
	
	Consider the curve
	\[
	C_t=e^{tF}Ce^{-tF}.
	\]
	Since $F$ is skew-symmetric, $e^{tF}$ is orthogonal and therefore
	\[
	\|C_t\|_{\mathrm{HS}}=\|C\|_{\mathrm{HS}},
	\qquad t\in\mathbb R.
	\]
	On the other hand,
	\[
	\frac{d}{dt}C_t=[F,C_t]=C_t^2,
	\qquad C_0=C.
	\]
	Since $C$ is nilpotent, the solution of this equation is
	\[
	C_t=C(\Id_V-tC)^{-1}
	=C+tC^2+t^2C^3+\cdots.
	\]
	The right-hand side is a polynomial in $t$. Since $||C_t||_{HS}$ is constant then $C^2=0$. Equation \eqref{eq:zero-discriminant} now gives
	\[
	[F,C]=0.
	\]
	Thus
	\[
	(A-r\Id_V)^2=0,
	\qquad
	[F,A]=0,
	\]
	which is equivalent to \eqref{eq:quadratic-commutator-conclusion}.
	
	Finally, $\ker C$ is $F$-invariant because $[F,C]=0$. Since $F$ is
	skew-symmetric, $\ker C^\perp$ is also $F$-invariant. Moreover,
	$C^2=0$ implies
	\[
	\operatorname{Im}C\subset\ker C.
	\]
	Hence, relative to
	$V=\ker C\oplus(\ker C)^\perp$,
	\[
	C=
	\begin{pmatrix}
		0&U\\
		0&0
	\end{pmatrix}.
	\]
	The relation $[F,C]=0$ is precisely $F_1U=UF_2$, which proves the
	claimed block description.
\end{proof}

\section{Complete Lorentzian semisimple Lie groups of nonzero constant curvature }\label{section7}

In this section, we prove that 
\(\widetilde{\mathrm{SL}(2,\mathbb R)}\) occupies a distinguished position among Lorentzian Lie groups. More precisely, we show that:
\begin{enumerate}\item Any left-invariant Lorentzian metric with nonzero constant curvature on \(\widetilde{\mathrm{SL}(2,\mathbb R)}\) is bi-invariant and proportional to the metric associated to the Killing form.
	\item 
	\(\widetilde{\mathrm{SL}(2,\mathbb R)}\)
	is the only connected simply connected Lie group admitting a left-invariant Lorentzian metric of nonzero constant curvature and possessing a  noncentral, spacelike left-invariant Killing vector field.
	\item 
	\(\widetilde{\mathrm{SL}(2,\mathbb R)}\)	is the only connected simply connected semisimple Lie group  which carries a complete left-invariant Lorentzian metric of nonzero constant curvature.
\end{enumerate}

\begin{pr}\label{thm-dim3}
	Let $h$ be a left-invariant Lorentzian metric of nonzero constant sectional curvature on
	$\widetilde{\mathrm{SL}(2,\mathbb R)}$. Then $h$ is bi-invariant and, up to a nonzero scalar multiple, coincides with the left-invariant metric induced by the Killing form.
\end{pr}

\begin{proof}
	The classification of left-invariant Lorentzian metrics on
	$\widetilde{\mathrm{SL}(2,\mathbb R)}$ together with the computation of their sectional curvature was established in \cite{Chakkar}. In particular, Propositions~4.3--4.8 of \cite{Chakkar} show that the only left-invariant Lorentzian metrics with nonzero constant sectional curvature are the bi-invariant ones, which are necessarily proportional to the metric induced by the Killing form. This proves the result.
\end{proof}

Let \((G,h)\) be a Lie group endowed with a left-invariant pseudo-Riemannian metric. Every right-invariant vector field on \(G\) is a Killing vector field. On the other hand, a left-invariant vector field \(X\) is Killing if and only if the adjoint operator
$
\ad_X:\mathfrak g\to\mathfrak g
$
is skew-symmetric with respect to the scalar product induced by \(h\) on the Lie algebra \(\mathfrak g\). In particular, every central left-invariant vector field is Killing. Moreover, if \(h\) is bi-invariant, then every left-invariant vector field is Killing.

For Lorentzian Lie groups of nonzero constant curvature, we have the following result.

\begin{theo}
	Let \((G,h)\) be a connected and simply connected Lorentzian Lie group of nonzero constant curvature. Assume that \(G\) admits a  noncentral, spacelike left-invariant Killing vector field. Then \((G,h)\) is isometric to the universal covering group \(\widetilde{\mathrm{SL}(2,\mathbb R)}\) endowed with a bi-invariant Lorentzian metric.
	
	Consequently, \(\widetilde{\mathrm{SL}(2,\mathbb R)}\) is the unique connected simply connected Lie group admitting a bi-invariant Lorentzian metric of nonzero constant curvature.
\end{theo}

\begin{proof} Note that a left-invariant vector field $X$ is a Killing vector field if and only if $\ad_X$ is skew-symmetric. Since $\ad_X=\Lu_X-\Ru_X$ and $\Lu_X$ is skew-symmetric, this is equivalent to $\Ru_X$ is skew-symmetric.
	
	Let $e\in\G$ be a left-invariant spacelike Killing vector field. After rescaling, we may assume that $e$ is a unit
	spacelike left-invariant Killing vector field.
	Then $\Ru_e$ is skew-symmetric and $e.e=0$. So $\Ru_e$ and $\Lu_e$ leave invariant $\h=e^\perp$. We have $\G=\h\oplus\R e$ and, for any $a,b\in\h$, the Levi-Civita product is given by
	\[ e.a=Ea,\; a.e=Fa\esp a.b=a\star b-\langle Fa,b\rangle e, \]
	where $E,F:\h\too\h$ are two skew-symmetric endomorphisms and $\star$ is  the Levi-Civita product of   $(\h,\br_\star,\prs)$ and $[a,b]_\star=a\star b-b\star a$.
	
	Let us establish the conditions under which the curvature is constant.
	
	$\bullet$ $a,b,c\in\h$.
	\begin{align*}
		\ass(a,b,c)&=(a\star b-\langle Fa,b\rangle e).c-a.(b\star c-\langle Fb,c\rangle e)\\
		&=(a\star b)\star c-\langle F(a\star b),c\rangle e-\langle Fa,b\rangle  Ec
		-a\star(b\star c)+\langle Fa,b\star c\rangle e+\langle Fb,c\rangle Fa,\\
		\ass(a,b,e)&=(a\star b).e-a.(b.e)\\
		&=F(a\star b)-a\star Fb+\langle Fa,Fb\rangle e.
	\end{align*}
	So $\Ki(a,b)=k a\wedge b$ is equivalent to
	\[ \begin{cases}
		\Ki^*(a,b)=Fa\wedge Fb+k a\wedge b+2\langle Fa,b\rangle E,\\
		F([a,b]_\star)+b\star Fa-a\star Fb=0.
	\end{cases} \]
	\begin{align*}
		\ass(a,e,b)&=Fa\star b-\langle F^2a,b\rangle e-a\star Eb+\langle Fa,Eb\rangle e,\\
		\ass(e,a,b)&=Ea\star b-\langle FEa,b\rangle e-E(a\star b),\\
		\ass(a,e,e)&=F^2a,\\
		\ass(e,a,e)&=FEa-EFa.
	\end{align*}So $\Ki(a,e)=k a\wedge e$ if and only if
	\[ \begin{cases}
		Fa\star b+E(a\star b)-a\star Eb-Ea\star b=0,\\
		[F,E]=F^2+k\mathrm{id}_\h.
	\end{cases} \]
	Finally, the relation $\Ki(u,v)=ku\wedge v$ for any $u,v\in\G$ is equivalent to
	\begin{equation}\label{eqgg} \begin{cases}\Ki^*(a,b)=Fa\wedge Fb+k a\wedge b+2\langle Fa,b\rangle E,\\
			F([a,b]_\star)+b\star Fa-a\star Fb=0,\\
			[E,\Lu_a]=\Lu_{Ea-Fa},\\
			[F,E]=F^2+ k\mathrm{Id}_\h.
	\end{cases} \end{equation}
	Since $[F,E]$ is skew-symmetric and $F^2+k\mathrm{Id}_\h$ is symmetric, we deduce that
	$F^2=- k\mathrm{Id}_\h$ and $[E,F]=0$. Since $\h$ is Lorentzian, according to the normal forms of skew-symmetric endomorphisms on a Lorentzian vector space (see \cite{MP}), there exists $\{z,\bar{z}\}$ satisfying $\langle z,z\rangle=\langle \bar{z},\bar{z}\rangle=0$, $\langle z,\bar{z}\rangle=1$ and
	$\h=\R z\oplus\h_0\oplus\R\bar{z}$ such that, for any $a\in\h_0$,
	\[ Fz=\al z,\; Fa=\overline{F}a\esp F\bar{z}=-\al \bar{z}. \]The relation $F^2=- k\mathrm{Id}_\G$ implies
	$$F^2z=\al^2 z \esp F^2a= \overline{F}^2a.$$ Thus $\overline{F}^2=-k \mathrm{Id}_{\h_0}$ and $k=-\al^2$. This is impossible unless $\h_0=0$.
	
	So $\h$ is 2-dimensional and in the basis $(z, \bar{z})$, we have
	\[ F=\begin{pmatrix}
		\al&0\\0&-\al
	\end{pmatrix},\; E=\begin{pmatrix}
		\mu&0\\0&-\mu
	\end{pmatrix}\esp z\wedge\bar{z}=\begin{pmatrix}
		-1&0\\0&1
	\end{pmatrix}. \]
	Since $\so(\h)$ is abelian then \eqref{eqgg} reduces to
	\[ \Lu_{[z,\bar{z}]_\star}=-2\al(\al+\mu)z\wedge\bar{z}\esp (\al-\mu)\Lu_{z}=(\al-\mu)\Lu_{\bar{z}}=0. \]
	We distinguish two cases:
	\begin{itemize}
		\item $\al\not=\mu$. Then $\Lu_z=\Lu_{\bar{z}}=0$ and hence $\al=-\mu$. Then the Lie bracket on $\G$ is given by
		\[ [e,z]=-2\al z,\; [e,\bar{z}]=2\al \bar{z}\esp [z,\bar{z}]=-2\al e. \]
		This Lie algebra is isomorphic to $\mathfrak{sl}(2,\R)$ and the metric is proportional to the Killing form.
		\item $\al=\mu$. Put $\Lu_z= xz\wedge \bar{z}$ and $\Lu_{\bar{z}}= yz\wedge \bar{z}$. Then $2xy z\wedge \bar{z}=-4\al^2 z\wedge\bar{z}$. So $xy=-2\al^2$. So
		\[ [e,z]=[e,\bar{z}]=0\esp [z,\bar{z}]=x\bar{z}-2\frac{\al^2}x z-2\al e. \]
		Then $e$ is central. This contradicts the assumption that \(e\) is noncentral. Hence only the first case can occur.\qedhere
	\end{itemize}
\end{proof}

Let us now tackle the second part of this section.

Let $(M,g)$ be a simply connected, complete Lorentzian manifold of dimension $n$ and constant sectional curvature $k\neq 0$. Then:
\begin{itemize}
	\item if $k>0$, then $(M,g)$ is isometric to the de Sitter space 
	$$\dS^n
	=\left\{x\in\R^{n+1},-x_1^2+\sum_{j=2}^{n+1}x_j^2=r^2 \right\}\esp k=\frac1{r^2};$$
	\item if $k<0$, then $(M,g)$ is isometric to the universal anti-de Sitter space $\widetilde{\AdS}^n$ which is the universal covering of
	$$H_1^n=\left\{x\in\R^{n+1},-x_1^2-x_2^2+\sum_{j=3}^{n+1}x_j^2=-r^2 \right\}\esp k=-\frac1{r^2}.$$
\end{itemize}
Moreover, the Lie algebras of isometry groups of this space forms are
\[
\mathfrak{isom}(dS^n)\simeq\mathfrak{so}(n,1)
\esp
\mathfrak{isom}(\widetilde{AdS}^n)\simeq\mathfrak{so}(n-1,2),
\]
and the isotropy algebra at a point is, in both cases,
\[
\so(n-1,1).
\]Topologically, $\dS^n\simeq \R\times S^{n-1}$ and $\widetilde{\AdS}^n\simeq\R^n$.

Let \((G,h)\) be a connected and simply connected Lie group of dimension \(n\) endowed with a complete left-invariant Lorentzian metric of nonzero constant curvature \(k\). Then the following properties hold:
\begin{enumerate}
	\item \(G\) acts simply transitively on the corresponding Lorentzian space form \(M\), where \(M=\mathrm{dS}^{n}\) if \(k>0\) and \(M=\widetilde{\mathrm{AdS}}^{n}\simeq \mathbb{R}^{n}\) if \(k<0\). Moreover, $G$ is realized as a subgroup of \(\mathrm{Isom}(M)\), and at the level of Lie algebras one has the decomposition
	\[
	\mathfrak{isom}(M)=\mathfrak{g}\oplus \mathfrak{so}(n-1,1).
	\]
	
	\item \(G\) is diffeomorphic to \(M\).
\end{enumerate}

As we shall see, these two properties impose strong restrictions on the algebraic structure of \(G\), leading to a sharp characterization of the Lie groups that can carry complete left-invariant Lorentzian metrics of nonzero constant curvature.

\begin{theo}\label{thm-semisimple}
	Let $(G,h)$ be a connected and simply connected semisimple Lie group endowed with a complete left-invariant Lorentzian metric of nonzero constant sectional curvature.
	Then the curvature is negative and $(G,h)$ is isometric to
	$
	\widetilde{\mathrm{SL}(2,\mathbb R)}
	$
	endowed with a bi-invariant Lorentzian metric.
\end{theo}

\begin{proof}
	Let $n=\dim G$.
	We have seen that $G$ is diffeomorphic either to $\R^n$ or $\R\times S^{n-1}$. There are only two three dimensional connected and simply connected semisimple Lie groups, namely, $\widetilde{\mathrm{SL}(2,\mathbb R)}\simeq \R^3$ and $\mathrm{SU}(2)\simeq S^3$ and hence if $n=3$ then $G\simeq\widetilde{\mathrm{SL}(2,\mathbb R)}$.

	If $n\ge 4$, then necessarily $n\ge 6$, since there are no semisimple Lie algebras of dimensions $4$ or $5$. Since 
	$
	\pi_3(\mathbb R\times S^{n-1})=\pi_3(\R^n)=0,
	$
	then $\pi_3(G)=0.$
	
	Let $K$ be a maximal compact subgroup of $G$. By the Iwasawa decomposition, $G$ has the homotopy type of $K$, and therefore
	\[
	\pi_3(K)=0.
	\]
	On the other hand, if
	$
	K=K_1\times\cdots\times K_s\times \mathbb T^m,
	$
	where the $K_i$ are compact simple Lie groups, then
	$
	\pi_3(K)\simeq \mathbb Z^s.
	$
	Hence $s=0$, and $K$ is a torus. But $\pi_1(G)=0$ and hence $K$ is trivial. From the theory of semisimple Lie groups the only semisimple simply-connected Lie group with no compact subgroup is a product of copies of $\widetilde{\mathrm{SL}(2,\mathbb R)}$. 
	Consequently,
	\[
	G\simeq \big(\widetilde{\mathrm{SL}(2,\mathbb R)}\big)^r
	\]
	for some integer $r\ge1$,  $k<0$ and $n=3r$. We can suppose that $G= \big(\widetilde{\mathrm{SL}(2,\mathbb R)}\big)^r$ and at the Lie algebras level
	\[
	\mathfrak{so}(3r-1,2)
	=
	\mathfrak{so}(3r-1,1)
	+
	\mathfrak{sl}(2,\mathbb R)^r.
	\]
	The monotonicity of the real rank of real semisimple Lie algebras (see \cite[pp. 424]{Knapp})
	yields
	\[
	r
	=
	\operatorname{rank}_{\mathbb R}
	\big(\mathfrak{sl}(2,\mathbb R)^r\big)
	\le
	\operatorname{rank}_{\mathbb R}
	\big(\mathfrak{so}(3r-1,2)\big)
	=2.
	\]
	Hence $r\leq 2$.
	
	Assume by contradiction that $r=2$ and hence
	\[
	\mathfrak{so}(5,2)=\mathfrak{so}(5,1)+\mathfrak g,
	\]
	where
	\[
	\mathfrak g\simeq \mathfrak{sl}(2,\mathbb R)\oplus \mathfrak{sl}(2,\mathbb R).
	\]
	Since both \(\mathfrak{so}(5,1)\) and \(\mathfrak g\) are semisimple, this would define a proper semisimple decomposition of the simple Lie algebra \(\mathfrak{so}(5,2)\).
	
	However, by Onishchik's classification of proper semisimple decompositions of simple real Lie algebras \cite[Theorem~4.1 and Table~2]{Oniscik}, every such decomposition is explicitly listed. The Lie algebra $\mathfrak{so}(5,2)$ does not appear in this classification. Therefore, $\mathfrak{so}(5,2)$ admits no proper semisimple decomposition, yielding a contradiction.
	Consequently, $r=1$
	and hence
	\[
	G\simeq \widetilde{\mathrm{SL}(2,\mathbb R)}.
	\]
	Finally,  Proposition \ref{thm-dim3} implies that $h$ is, up to a nonzero constant factor, the bi-invariant Lorentzian metric induced by the Killing form. This completes the proof.
\end{proof}

\section{Tables}

The classification tables below summarize the list of Lorentzian Lie algebras of constant curvature  obtained  so far (see Proposition \ref{Zspacelike}, \ref{Ztimelike}, \ref{prop:Lorentzian-derived-low-dimensional}, \ref{prop:four-dim} and Theorems \ref{ZDC}, \ref{D-Euclidean} and \ref{DD}). In 
dimensions two, three, and four,  our analysis provides normal forms for the Lie brackets and the corresponding Lorentzian metrics of all non-perfect Lie algebras with nonzero constant sectional curvature, by distinguishing the cases where the derived ideal is degenerate or nondegenerate. Starting from these normal forms, we perform suitable changes of basis to identify the underlying Lie algebras with the standard models appearing in the existing classifications of low-dimensional real Lie algebras, while expressing the corresponding metrics in the same bases. This leads to the classification tables presented below. Since these changes of basis involve routine but sometimes lengthy computations, we omit their details in order to avoid unnecessarily burdening the exposition.

{\renewcommand*{\arraystretch}{2}
	\begin{center}
		\begin{tabular}{|l|l|l|l|}
			\hline
			Name&Nonvanishing Lie brackets&Metric&Curvature\\
			\hline
			$\aff(\R)\oplus\R$& $[E_1,E_2]=E_2$&$\begin{pmatrix}
				\frac1{4\al^2}&0&0\\0&0&1\\0&1&\e
			\end{pmatrix},\;\e\in\{0,1\}$& $k=-\al^2<0$.\\
			\hline
			$\R e\stackrel{\perp}\oplus\R h_0\stackrel{\perp}\oplus\R h_1\stackrel{\perp}\oplus\af$&
			$[h_0,a]=\la a+S(a),\; [h_0,h_1]=2\la(h_1+e)$&$\langle e,e\rangle=-1$&$k=-\la^2<0$\\
			&$[a,b]=2\la\langle Ja,b\rangle(h_1+e),\; a,b\in\af$&$\langle h_i, h_i\rangle=1,i=0,1$&\\
			&$J,S\in\so(\af),\; J^2=-\mathrm{Id}_\af,[J,S]=0,\;\la>0$&$\af$ is Euclidean
			&\\
			\hline
		$\mathfrak r_{3,\frac12}\oplus\R$	&$[E_1,E_3]=E_1,
		\;
		[E_2,E_3]=\frac12E_2$&$\begin{pmatrix}
				0&0&0&1\\[2mm]
				0&1&0&0\\[1mm]
				0&0&\dfrac{1}{4\lambda^2}&0\\[2mm]
				1&0&0&0
			\end{pmatrix}$&$k=-\la^2<0$\\
			\hline
			$\G=\R e\oplus\h\oplus\R f$&$[u,a]
			=Ba+\lambda a+\langle a_0,a\rangle e$&$\langle e,e\rangle=\langle f,f\rangle=0$&$k=-\la^2<0$\\
		$\h=\R u\oplus u^\perp$	&$[f,a]
			=\frac{\alpha}{2\lambda}
			\bigl(Ba+\lambda a+\langle a_0,a\rangle e\bigr)$&$\langle e,f\rangle=1,\h=(\R e\oplus\R f)^\perp$&\\
		$\dim\h\geq3$	&$[f,u]
			=\alpha u+a_0+\frac{\alpha^2}{4\lambda}e-2\lambda f$&$\h$ is Euclidean&\\
			&$\al,\la\in\R$, $a_0\in u^\perp$, $B\in\so(u^\perp)$&&\\
			\hline		
			
		\end{tabular}
	\end{center}
	\captionof{table}{ Lorentzian Lie algebras of constant curvature with nontrivial center \label{1}}
}

{\renewcommand*{\arraystretch}{2}
	\begin{center}
		\begin{tabular}{|l|l|l|l|}
			\hline
			Name&Nonvanishing Lie brackets&Metric&Curvature\\
			\hline
	$\G$ & $[x,y]=\langle x,u\rangle E(y)-\langle y,u\rangle E(x)$& $(\G,\prs)$ is a Lorentzian& $k=-\langle u,u\rangle\la^2$\\
	& $u\in\G$, $B\in\so(\G)$, $Bu=0$& vector space&\\
	&$E=B+\la\mathrm{Id_\G}$&&\\
	\hline
	$\G=\R e\oplus\h\oplus\R f$&$[e,z]
	=-2\mu |z|^2e,\;
	[e,f]
	=-\lambda e,
	$&$\langle e,e\rangle=\langle f,f\rangle=0$&$k=-\mu^2|z|^2<0$\\
$\h=\R z\stackrel{\perp}\oplus V$	&$[a,b]
	=
	\bigl(
	\langle Aa,b\rangle-\langle Ab,a\rangle
	\bigr)e$&$\langle e,f\rangle=1$, $\h=(\R e\oplus\R f)^\perp$&\\
	&$[z,a]
	=
	|z|^2(B+\mu\Id_V)a+\langle p,a\rangle e$&$\h$ is Euclidean&\\
	&$[f,a]
	=(F+A)a+\langle a,w\rangle e$&&\\
	&$[f,z]
	=p+\rho |z|^2e,\; a,b\in V$&&\\
	\cline{2-2}
	&$\mu,\la,\rho\in\R$ with $\mu\not=0$&&\\
	&$F,B\in\so(V)$,  $A\in\mathrm{End}(V)$ and $p\in V$&&\\
	\cline{2-2}
	&$[F,B]=[A,B]=0$&&\\
	&$[F,A]=A^2-\la A-\mu\rho|z|^2 \mathrm{Id}_{V}$&&\\	
	&$w=\frac1{|z|^2}(B+\mu\mathrm{Id}_{V})^{-1}(Fp-Ap+\la p)$&&\\
	\hline			
			
		\end{tabular}
	\end{center}
	\captionof{table}{ Lorentzian Lie algebras of constant curvature with Euclidean or degenerate derived ideal \label{2}}\footnote{The equation $[F,A]=A^2-\la A-\mu\rho|z|^2 \mathrm{Id}_{V}$ has been solved completely when $\la^2+4\mu\rho|z|^2\geq0$ (See Proposition \ref{prop:quadratic-commutator})}
}

{\renewcommand*{\arraystretch}{1.5}
	\begin{center}
		\begin{tabular}{|l|l|l|l|}
			\hline
			The Lie algebra&Nonvanishing Lie brackets&Metric&Curvature\\
			\hline	
			$\G=\aff(\R)$&$[E_1,E_2]= E_1$&$\begin{pmatrix}
				\e&0\\0&-\frac\e{\la^2}
			\end{pmatrix}$&$k=\e\la^2$\\
			\hline
			$\G=\mathrm{sl}(2,\R)$& $[E_1,E_2]=2E_3,\;[E_3,E_2]=2E_1,\;[E_3,E_1]=2E_2$&
			$M=\la^2\mathrm{Diag}(-1,1,1)$&$k=-\frac1{\la^2}$\\
			\hline
			$\mathfrak r_{3,1}$&$[E_2,E_1]=E_2,\;
			[E_3,E_1]=E_3.$&$\begin{pmatrix}
				-\dfrac1{\lambda^2}&0&0\\[2mm]
				0&1&0\\
				0&0&1
			\end{pmatrix},$&$
			k=\lambda^2>0$\\
			\hline
		$\mathfrak r'_{3,a},\; a\not=0$&$[E_1,E_2]=aE_2+E_3,\qquad
		[E_1,E_3]=-E_2+aE_3$&$\begin{pmatrix}
			-\dfrac{a^2}{\la^2}&0&0\\[2mm]
			0&1&0\\
			0&0&1
		\end{pmatrix}$&$k=\la^2$\\
		\hline
		$\mathfrak r_{3,1}$& $[E_1,E_3]=E_1,\qquad
		[E_2,E_3]=E_2$&	$\begin{pmatrix}
			0&1&0\\
			1&0&0\\
			0&0&\dfrac1{\lambda^2}
		\end{pmatrix}$&$
		k=-\lambda^2<0$\\
		\hline
		$\aff(\R)\oplus\R$&$[E_1,E_3]=E_1$&
		$\begin{pmatrix}
			0&1&0\\
			1&0&0\\
			0&0&\dfrac1{4\lambda^2}
		\end{pmatrix}$&
		$
		k=-\lambda^2<0$\\
	\hline
	$\mathfrak r_{3,r},
	\qquad
	r=\frac{\lambda-\mu}{\lambda+\mu}$&$[E_1,E_3]=E_1,\qquad
	[E_2,E_3]=rE_2$&$\begin{pmatrix}
		0&1&0\\
		1&0&0\\
		0&0&\dfrac1{(\lambda+\mu)^2}
	\end{pmatrix}$&$k=-\lambda^2<0.$\\
	$\la+\mu\not=0$&&&\\
	\hline
	$\aff(\R)\oplus\R$&$[E_1,E_2]=E_2$&
	$\begin{pmatrix}
		0&1&0\\
		1&0&y\\
		0&y&\la
	\end{pmatrix}$&
	$
	k=-\frac{y^2}{4\la}<0$\\
	\hline

		\end{tabular}
	\end{center}
	\captionof{table}{ Lorentzian Lie algebra with nonzero constant curvature of dimension $\leq 3$\label{3}}
}

{\renewcommand*{\arraystretch}{2}
	\begin{center}
		\begin{tabular}{|l|l|l|l|}
			\hline
			The Lie algebra&Nonvanishing Lie brackets&Metric&Curvature\\
			\hline	
			$A_{4,6}^{a,a}$, &$[E_1,E_4]=aE_1,\;
			[E_2,E_4]=aE_2+E_3,\;
			$& $\diag(\e,1,1,-\dfrac{\e a^2}{\la^2})
				,\; \e^2=1$&$k=\e\la^2$\\
		$a\not=0$	&$[E_3,E_4]=-E_2+aE_3.$&&\\
			\hline
			$A_{4,5}^{1,1}$&$[E_i,E_4]=E_i,\; i=1,2,3$&
			$\diag(1,1,\e,-\dfrac\e{\lambda^2})
			,\;\e^2=1$&$k=\e\la^2$\\
			\hline
			$A_{4,5}^{p,2-p},$&$[E_1,E_4]=E_1,\;
			[E_2,E_4]=pE_2,\;
			$&$\begin{pmatrix}
				1&0&0&0\\
				0&0&1&0\\
				0&1&0&0\\
				0&0&0&\dfrac1{\lambda^2}
			\end{pmatrix}$&$k=-\la^2$\\
			&$[E_3,E_4]=(2-p)E_3$&&\\
			\hline
			$\G_{\la,\ga}$& $[E_3,E_1]=2E_1,[E_3,E_2]=E_2$&$\begin{pmatrix}
				0&0&0&1\\
				0&1&0&0\\
				0&0&c&0\\
				1&0&0&\ga(\ga-\la)c
			\end{pmatrix}$, $c>0$&$k=-\frac1c<0$\\
		$\ga(\ga-\la)\not=0$	&$[E_4,E_1]=\la E_1,\;[E_4,E_2]=\ga E_2$&&\\
			\hline
			$A_{4,9}^{\frac12}$& $[E_2,E_3]=E_1,\quad
			[E_1,E_4]=\frac32 E_1,\quad
			$&$\begin{pmatrix}
				0&0&1&0\\
				0&1&0&0\\
				1&0&0&0\\
				0&0&0&s^{-2}
			\end{pmatrix}$&$k=-s^2$\\
			&$[E_2,E_4]=E_2,\quad
			[E_3,E_4]=\frac12 E_3.$&&\\
			\hline
			Remark& $\G_{2\ga,\ga}\simeq A_{3,5}^{\frac12}\oplus\R$&
			$\G_{\la,\ga}\simeq\aff(\R)\oplus\aff(\R)$ if $\la\not=2\ga$&\\
			\hline

		\end{tabular}
	\end{center}
	\captionof{table}{ Lorentzian Lie algebra with nonzero constant curvature of dimension $4$\label{4}}
}

\section{Appendix}\label{section9}

\subsection{Proof of Proposition \ref{ZD}}\label{subsection1}

We have $\G=\R e\oplus\h\oplus\R f$. Since $D=0$, \eqref{LeviZ} reduces to
\begin{equation*}
	\begin{cases}
		e\cdot a=a\cdot e=\langle w,a\rangle e,\\[1mm]
		a\cdot b=a\star b+\langle Aa,b\rangle e,\\[1mm]
		e\cdot e=0,\qquad
		e\cdot f=f\cdot e=-w,\qquad
		f\cdot f=v,\\[1mm]
		f\cdot a=Fa-\langle a,v\rangle e+\langle a,w\rangle f,\\[1mm]
		a\cdot f=-Aa-\langle w,a\rangle f.
	\end{cases}
\end{equation*}
The relation $\Ki(u,v)=k u\wedge v$ is equivalent to
\[ \ass(u,v,w)-\ass(v,u,w)=k\left(\langle u,w\rangle v-\langle v,w\rangle u\right), \]
where $\ass(u,v,w)=(u.v).w-u.(v.w)$.

$\bullet$ $a,b,c\in\h$.
\begin{align*}
	\ass(a,b,c)&=(a.b).c-a.(b.c)\\
	&=(a\star b+\langle Aa,b\rangle e).c-a.(b\star c+\langle Ab,c\rangle e)\\
	&=(a\star b)\star c+\langle A(a\star b),c\rangle e+\langle Aa,b\rangle\langle w,c\rangle e-a\star(b\star c)-\langle Aa,b\star c\rangle e-\langle Ab,c\rangle\langle a,w\rangle e\\
	&=\ass_\star(a,b,c)+\langle A(a\star b)+\langle Aa,b\rangle w+b\star Aa
	-\langle a,w\rangle Ab,c\rangle e.
\end{align*}
So
\[ \begin{cases}
	\Ki^*(a,b)=k a\wedge b,\\\; A([a,b]_\star)+b\star Aa-a\star Ab=
	\langle (A-A^*)b,a\rangle w+\langle a,w\rangle Ab-\langle b,w\rangle Aa.
\end{cases} \]

$\bullet$ $a,b\in\h$.
\begin{align*}
	\ass(a,b,e)&=(a.b).e-a.(b.e)\\
	&=(a\star b).e-\langle b,w\rangle  a.e\\
	&=\langle a\star b,w\rangle e-\langle b,w\rangle\langle a,w\rangle e,\\
	\ass(a,b,f)&=(a.b).f-a.(b.f)\\
	&=(a\star b).f+\langle Aa,b\rangle e.f+a.(Ab+\langle w,b\rangle f)\\
	&=-A(a\star b)-\langle w,a\star b\rangle f-\langle a,b\rangle w
	+a\star Ab+\langle Aa,Ab\rangle e-\langle w,b\rangle Aa-\langle w,b\rangle\langle w,a\rangle f\\
	&=-A(a\star b)+a\star Ab-\langle Aa,b\rangle w-\langle w,b\rangle Aa
	+\langle a\star w-\langle w,a\rangle w,b\rangle f+\langle Aa,Ab\rangle e
\end{align*}

So 
\[ \begin{cases}\langle [a,b]_\star,w\rangle =0,\\
	A([a,b]_\star)-a\star Ab+b\star Aa=\langle (A-A^*)b,a\rangle w+
	\langle w,a\rangle Ab
	-\langle w,b\rangle Aa.
\end{cases} \]

\begin{align*}
	\ass(a,e,e)&=(a.e).e-a.(e.e)=0,\\
	\ass(e,a,e)&=(e.a).e-e.(a.e)=0,\\
	\ass(a,e,f)&=(a.e).f-a.(e.f)\\
	&=-\langle w,a\rangle w+a\star w+\langle Aa,w\rangle e,\\
	\ass(e,a,f)&=(e.a).f-e.(a.f)\\
	&=-\langle w,a\rangle w+e.(Aa+\langle a,w\rangle f)\\
	&=-\langle w,a\rangle w+\langle Aa,w\rangle e-\langle a,w\rangle w
\end{align*}	So $a\star w+\langle a,w\rangle w=-k a$.

\begin{align*}
	\ass(a,f,b)&=(a.f).b-a.(f.b)\\
	&=-(Aa+\langle a,w\rangle f).b-a.(Fb+\langle b,w\rangle f-\langle b,v\rangle e)\\
	&=-Aa\star b-\langle A^2a,b\rangle e-\langle a,w\rangle(Fb+\langle b,w\rangle f-\langle b,v\rangle e)-a\star Fb-\langle Aa,Fb\rangle e\\&+\langle b,w\rangle(Aa+\langle a,w\rangle f )
	+\langle b,v\rangle\langle a,w\rangle e\\
	&=-Aa\star b-\langle a,w\rangle Fb-a\star Fb+\langle b,w\rangle Aa
	-\langle A^2a,b\rangle e-\langle Aa,Fb\rangle e+2\langle b,v\rangle\langle a,w\rangle e.
\end{align*}
\begin{align*}
	\ass(f,a,b)&=(f.a).b-f.(a.b)\\
	&=(Fa+\langle a,w\rangle f-\langle a,v\rangle e).b-f.(a\star b+\langle Aa,b\rangle e)\\
	&=Fa\star b+\langle AFa,b\rangle e+ \langle a,w\rangle(Fb+\langle b,w\rangle f-\langle b,v\rangle e)-\langle a,v\rangle \langle b,w\rangle e\\&
	-F(a\star b)-\langle a\star b,w\rangle f+\langle a\star b,v\rangle e+\langle Aa,b\rangle w\\
	&=Fa\star b+\langle a,w\rangle Fb-F(a\star b)+\langle Aa,b\rangle w
	+\langle AFa,b\rangle e-\langle a,w\rangle\langle b,v\rangle e-\langle a,v\rangle \langle b,w\rangle e+\langle a\star b,v\rangle e\\
	&+ \langle a,w\rangle \langle b,w\rangle f-\langle a\star b,w\rangle f.
\end{align*}So
\[ \begin{cases}
	[F,\Lu_a]=\Lu_{Fa+Aa}+2\langle a,w\rangle F+Aa\wedge w,\\
	[F,A]=A^2-3\langle \bullet, w\rangle v-\langle \bullet, v\rangle w-\Ru_v,\\
	\langle a,w\rangle w+a\star w=-ka.
\end{cases} \]

\begin{align*}
	\ass(a,f,e)&=(a.f).e-a.(f.e)\\
	&=-(Aa+\langle a,w\rangle f).e+a\star w+\langle Aa,w\rangle e\\
	&=-\langle Aa,w\rangle e+\langle a,w\rangle w +a\star w+\langle Aa,w\rangle e\\
	&=a\star w+\langle a,w\rangle w,\\
	\ass(f,a,e)&=(f.a).e-f.(a.e)\\
	&=(Fa+\langle a,w\rangle f-\langle a,v\rangle e).e+\langle w,a\rangle w\\
	&=\langle Fa,w\rangle e-\langle a,w\rangle w +\langle w,a\rangle w
\end{align*}So $Fw=0$ and $a\star w+\langle a,w\rangle w=-ka$

\begin{align*}
	\ass(a,f,f)&=(a.f).f-a.(f.f)\\
	&=-(Aa+\langle w,a\rangle f)f-a\star v-\langle Aa,v\rangle e\\
	&=A^2a+\langle Aa,w\rangle f-\langle w,a\rangle v-a\star v-\langle Aa,v\rangle e,\\
	\ass(f,a,f)&=(f.a).f-f.(a.f)\\
	&=(Fa+\langle a,w\rangle f-\langle a,v\rangle e).f+f.(Aa+\langle w,a\rangle f)\\
	&=-AFa-\langle Fa,w\rangle f+\langle a,w\rangle v+\langle a,v\rangle w
	+FAa+\langle Aa,w\rangle f-\langle Aa,v\rangle e+\langle w,a\rangle v\\
	&=[F,A]a+2\langle a,w\rangle v+\langle a,v\rangle w-\langle Fa,w\rangle f+
	\langle Aa,w\rangle f-\langle Aa,v\rangle e.
\end{align*}So
\[ [F,A]=A^2-3\langle \bullet, w\rangle v-\langle \bullet, v\rangle w-\Ru_v\esp Fw=0. \]

\begin{align*}
	\ass(e,f,a)&=(e.f).a-e.(f.a)\\
	&=-w\star a-\langle Aw,a\rangle e-e.(Fa+\langle w,a\rangle f-\langle v,a\rangle e)\\
	&=-w\star a-\langle Aw,a\rangle e-\langle Fa,w\rangle e+\langle w,a\rangle w,\\
	\ass(f,e,a)&=(f.e).a-f.(e.a)\\
	&=-w\star a-\langle Aw,a\rangle e+\langle a,w\rangle w
\end{align*}So $Fw=0$.

\begin{align*}
	\ass(e,f,e)-\ass(f,e,e)&=-e.(f.e)+f.(e.e)\\
	&=|w|^2 e
\end{align*}So $k=-|w|^2$.
\begin{align*}
	\ass(e,f,f)-\ass(f,e,f)&=-e.(f.f)+f.(e.f)\\
	&=-\langle v,w\rangle e-Fw-\langle w,w\rangle f+\langle w,v\rangle e\\
	&=-Fw-|w|^2 f=k f.
\end{align*}

\subsection{Computation for Theorem \ref{D-Euclidean}}\label{subsection2}

We have $\G=\R e\oplus\h$ and, for any $a,b\in\h$,
\[ e.a=Ea,\; a.e=Sa\esp a.b=a\star b+\langle Sa,b\rangle e \]where $S^*=S$ and $E^*=-E$.

$\bullet$ $a,b,c\in\h$.
\begin{align*}
	\ass(a,b,c)&=(a\star b+\langle Sa,b\rangle e).c-a.(b\star c+\langle Sb,c\rangle e)\\
	&=(a\star b)\star c+\langle S(a\star b),c\rangle e+\langle Sa,b\rangle  Ec
	-a\star(b\star c)-\langle Sa,b\star c\rangle e-\langle Sb,c\rangle Sa.
\end{align*}
So
\[ \begin{cases}
	\Ki^*(a,b)c=\langle Sb,c\rangle Sa-\langle Sa,c\rangle Sb+k (a\wedge b)(c)
	+(\langle Sb,a\rangle-\langle Sa,b\rangle)Ec,\\
	S([a,b]_\star)+b\star Sa-a\star Sb=0.
\end{cases} \]
\begin{align*}
	\ass(a,b,e)&=(a\star b).e-a.(b.e)\\
	&=S(a\star b)-a\star Sb-\langle Sa,Sb\rangle e.
\end{align*}
So
\[ S([a,b]_\star)+b\star Sa-a\star Sb=0.\]
\begin{align*}
	\ass(a,e,b)&=Sa\star b+\langle S^2a,b\rangle e-a\star Eb-\langle Sa,Eb\rangle e,\\
	\ass(e,a,b)&=Ea\star b+\langle SEa,b\rangle e-E(a\star b)
\end{align*}So
\[ \begin{cases}
	Sa\star b+E(a\star b)-a\star Eb-Ea\star b=0,\\
	S^2+[E,S]=k\mathrm{id}_\h.
\end{cases} \]

\begin{align*}
	\ass(a,e,e)&=S^2a,\\
	\ass(e,a,e)&=SEa-ESa.
\end{align*}So $S^2+[E,S]= k\mathrm{Id}_\G$.

\subsection{Computation for Proposition \ref{PD}}\label{subsection3}

\begin{equation} \begin{cases}
		e.a= Ea+\langle v,a\rangle e,\\
		a.e=\langle u,a\rangle e,\\
		a.b=a\star b+\langle Aa,b\rangle e,\\
		e.f=-v,\;
		f.e=u+\la e,\; f.f=-w-\la f,\\
		f.a=Fa+\langle a,w\rangle e-\langle a,u\rangle f,\\
		a.f=-Aa-\langle u,a\rangle f.
\end{cases} \end{equation}

Let us determine the conditions under which the Lie algebra
$
\G=\mathbb Re\oplus\mathfrak h\oplus\mathbb Rf,
$
endowed with the Levi-Civita product $\cdot$ given by \eqref{Levigg}, has constant curvature; that is, such that for all $u,v,w\in\G$,
\[
Q(u,v,w):=\ass(u,v,w)-\ass(v,u,w)-k(u\wedge v)(w)=0.
\]
Observe that
\[
\langle Q(u,v,w),x\rangle=-\langle Q(u,v,x),w\rangle,
\]
for all $u,v,w,x\in\G$. Consequently, $(\G,\cdot,\langle\cdot,\cdot\rangle)$ is  has constant curvature if and only if, for all $a,b,c\in\h$,
\begin{equation}\label{Q}
	Q(a,b,c)=Q(a,b,e)=Q(a,e,b)=Q(a,e,e)=Q(a,f,b)=Q(a,f,e)=Q(e,f,a)=Q(e,f,e)=0.
\end{equation}
Indeed, under these assumptions, for all $u,v\in\G$ and $a\in\h$, one has
\[
\langle Q(u,v,f),e\rangle=0,
\qquad
\langle Q(u,v,f),a\rangle=0.
\]
Moreover, the quantity $\langle Q(u,v,f),f\rangle$ vanishes identically. Hence
\[
Q(u,v,f)=0,
\]
which proves the claim. Let us expand the equations \eqref{Q}.
\begin{enumerate}
	\item For $a,b,c\in\h$.
	\begin{align*}
		\ass(a,b,c)&=(a\star b+\langle Aa,b\rangle_\h e).c-a.(b\star c+\langle Ab,c\rangle_\h e)\\
		&=(a\star b)\star c+\langle A(a\star b),c\rangle_\h e+
		\langle Aa,b\rangle_\h(Ec+\langle c,v\rangle_\h e)\\
		&-a\star (b\star c)-\langle Aa,b\star c\rangle_\h e-\langle Ab,c\rangle_\h\langle a,u\rangle_\h e\\
		&=\ass_{\star}(a,b,c)+\langle Aa,b\rangle_\h Ec+\langle  A(a\star b)+b\star Aa+\langle Aa,b\rangle v-\langle a,u\rangle Ab,c\rangle_\h e.
	\end{align*}
	Then $Q(a,b,c)=0$ if and only if
	\[ \begin{cases}
		\ass_{\star}(a,b,c)-\ass_{\star}(b,a,c)=\left(\langle Ab,a\rangle_\h-\langle Aa,b\rangle_\h \right)Ec+k(a\wedge b)(c),\\
		A([a, b]_\star)+b\star Aa-a\star Ab+\left(\langle Aa,b\rangle_\h-\langle Ab,a\rangle_\h\right) v
		+\langle b,u\rangle_\h Aa-\langle a,u\rangle_\h Ab=0.
	\end{cases} \]
	We obtain the first and the second equation in \eqref{curvature}.
	\item $a,b\in\h$.
	\begin{align*}
		\ass(a,b,e)&=(a\star b+\langle Aa,b\rangle_\h e).e-\langle b,u\rangle_\h \langle a,u\rangle_\h e
		=\langle a\star b,u\rangle_\h e
		-\langle b,u\rangle_\h \langle a,u\rangle_\h e\\
		\ass(b,a,e)&=\langle b\star a,u\rangle_\h e
		-\langle b,u\rangle_\h \langle a,u\rangle_\h e.
	\end{align*}
	Then $Q(a,b,e)=0$ if and only if
	$$\langle [a,b]_\star,u\rangle=0.$$
	We obtain the second relation in the last equation in \eqref{curvature}
	
	\item $a,b\in\h$.
	\begin{align*}
		\ass(a,e,b)&=\langle a,u\rangle_\h e.b-a.(Eb+\langle b,v\rangle_\h e)\\
		&=\langle a,u\rangle_\h(Eb+\langle b,v\rangle_\h e)-a\star Eb-\langle Aa,Eb\rangle_\h e
		-\langle b,v\rangle_\h\langle a,u\rangle_\h e\\
		&=\langle a,u\rangle_\h Eb-a\star Eb
		+\langle EAa,b\rangle_\h e,\\
		\ass(e,a,b)&=(Ea+\langle a,v\rangle_\h e).b-e.(a\star b+\langle Aa,b\rangle_\h e)\\
		&=Ea\star b+\langle AEa,b\rangle_\h e+\langle a,v\rangle_\h(Eb+\langle b,v\rangle_\h e)
		-E(a\star b)-\langle a\star b,v\rangle_\h e\\
		&= Ea\star b+\langle a,v\rangle_\h Eb-E(a\star b)
		+\left( \langle AEa,b\rangle_\h+\langle a,v\rangle_\h\langle b,v\rangle_\h
		-\langle a\star b,v\rangle_\h  \right)e.
	\end{align*}
	Hence $Q(a,e,b)=0$ if and only if
	\[ \begin{cases}
		E(a\star b)	-a\star Eb-Ea\star b=
		(\langle a,v\rangle_\h-\langle a,u\rangle_\h) Eb
		,\\
		[E,A]a-\langle a,v\rangle_\h v-a\star v=k a.
	\end{cases} \]We obtain the third and fourth equation in \eqref{curvature}.
	
	\item $a\in\h$.
	\begin{align*}
		\ass(a,e,e)&=0,\\
		\ass(e,a,e)&=(Ea+\langle a,v\rangle_\h e).e-e.(\langle a,u\rangle_\h e)\\
		&=\langle Ea,u\rangle_\h e
	\end{align*}
	So $Q(a,e,e)=0$ if and only if $E(u)=0$. We obtain the last relation in \eqref{curvature}.

	\item $a,b\in\h$.
	\begin{align*}
		\ass(a,f,b)&=(-\langle a,u\rangle_\h f-Aa).b-a.(Fb+\langle b,w\rangle_\h e-\langle u,b\rangle_\h f) ,\\
		&=-\langle a,u\rangle_\h(Fb+\langle b,w\rangle_\h e-\langle u,b\rangle f)-Aa\star b-\langle A^2a,b\rangle_\h e
		-a\star Fb-\langle Aa,Fb\rangle_\h e-\langle b,w\rangle_\h\langle a,u\rangle_\h e\\
		&+\langle u,b\rangle(-Aa-\langle a,u\rangle_\h f)\\
		&=-\langle a,u\rangle_\h Fb-Aa\star b-a\star Fb-\langle u,b\rangle Aa-\left(2\langle a,u\rangle_\h\langle b,w\rangle_\h+\langle A^2a,b\rangle_\h+\langle Aa,Fb\rangle_\h \right)e,\\
		\ass(f,a,b)&=(Fa+\langle a,w\rangle_\h e-\langle u,a\rangle_\h f).b-f.(a\star b+\langle Aa,b\rangle_\h e)\\
		&=Fa\star b+\langle AFa,b\rangle_\h e+\langle a,w\rangle_\h(Eb+\langle b,v\rangle_\h e)-\langle u,a\rangle_\h(Fb+\langle w,b\rangle_\h e-\langle u,b\rangle_\h f)\\
		&-(F(a\star b)+\langle a\star b,w\rangle_\h e-\langle a\star b,u\rangle f)-\langle Aa,b\rangle_\h(\lambda e+u)\\
		&=Fa\star b+\langle a,w\rangle_\h Eb-\langle u,a\rangle_\h Fb-F(a\star b)-\langle Aa,b\rangle_\h u\\&+\left(
		\langle AFa,b\rangle_\h+\langle a,w\rangle_\h\langle b,v\rangle_\h-\langle b,w\rangle_\h\langle a,u\rangle_\h
		-\langle a\star b,w\rangle_\h-\langle Aa,b\rangle_\h\lambda
		\right)e\\&+\left(\langle a,u\rangle_\h\langle b,u\rangle_\h+\langle a\star b,u\rangle \right)f
	\end{align*}
	So $Q(a,f,b)=0$ if and only if
	\[ \begin{cases}
		[F,\Lu_a]=\Lu_{Fa+Aa}+ \langle a,w\rangle_\h E+u\wedge Aa.\\
		[F,A]a=A^2a-\la Aa+\langle a,u\rangle_\h w+\langle a,w\rangle_\h v+a\star w.	\\
		a\star u-\langle a,u\rangle_\h u=ka.
	\end{cases} \]
	We get the fifth and sixth in \eqref{curvature}.

	\item $a\in\h$.
	\begin{align*}
		\ass(a,f,e)&=(-\langle a,u\rangle_\h f-Aa).e-\la a.e-a.u\\
		&=
		-\la\langle a,u\rangle_\h e-\langle a,u\rangle_\h u-\langle Aa,u\rangle e-\la \langle a,u\rangle_\h e-a\star u-\langle Aa,u\rangle_\h e\\
		\ass(f,a,e)&=(Fa+\langle a,w\rangle_\h e-\langle a,u\rangle_\h f).e-\langle a,u\rangle_\h(u+\la e)\\
		&=\langle Fa,u\rangle e-\langle a,u\rangle_\h(u+\la e)-\langle a,u\rangle_\h(u+\la e)
	\end{align*}So $Q(a,f,e)=0$ if and only if

	$$2A^*u-Fu=0\esp a\star u-\langle a,u\rangle_\h u=ka.$$ We obtain the ninth equation in \eqref{curvature}.

	\item $a\in\h$.
	\begin{align*}
		\ass(e,f,a)&=(-v).a-e.(Fa+\langle a,w\rangle_\h e-\langle u,a\rangle f)\\
		&=-v\star a-\langle Av,a\rangle_\h e
		-EFa-\langle Fa,v\rangle_\h e-\langle a,u\rangle_\h v\\
		&=-v\star a-EFa-\langle a,u\rangle_\h v-\left(\langle Av,a\rangle_\h
		+\langle Fa,v\rangle_\h\right)e ,\\
		\ass(f,e,a)&=\la (Ea+\langle a,v\rangle_\h e)+u\star a+\langle Au,a\rangle_\h e
		-f.(Ea+\langle a,v\rangle_\h e)
		\\
		&=\la (Ea+\langle a,v\rangle_\h e)+u\star a+\langle Au,a\rangle_\h e-FEa-\langle Ea,w\rangle e+\langle Ea,u\rangle f-\la \langle a,v\rangle_\h e-\langle a,v\rangle_\h u\\
		&=\la Ea-FEa+u\star a-\langle a,v\rangle_\h u+\langle Au,a\rangle_\h e- \langle Ea,w\rangle_\h+\langle Ea,u\rangle f
		.
	\end{align*}
	Then $Q(e,f,a)=0$ if and only if
	\[ \begin{cases}[F,E]=\la E+\Lu_{u+v}+u\wedge v,\\
		Av-Fv+Ew+Au=0,\; Eu=0.\end{cases} \]We obtain the seventh and eighth equations in \eqref{curvature}.
	
	\item Finally,
	\begin{align*}
		\ass(e,f,e)&=(-v).e-e.(\la e+u)=-\langle u,v\rangle e-Eu-\langle v,u\rangle e\\
		\ass(f,e,e)&=(\la e+u).e=\langle u,u\rangle e.
	\end{align*}
	Then $Q(e,f,e)=0$ if and only if $-2\langle u,v\rangle-|u|^2=-k.$ We obtain the first relation in the last equation in \eqref{curvature}.
	
\end{enumerate}

\bibliographystyle{elsarticle-num}

\end{document}